\documentclass[10pt]{amsart}
\usepackage{verbatim, amsmath,amstext,amssymb,mathtools}

\usepackage{graphicx,xcolor}

\usepackage{tikz-cd}

\usetikzlibrary{calc}

\usepackage{enumerate}

\usepackage{tikz}

\usetikzlibrary{arrows.meta, decorations.markings}

\usepackage[hyperindex=true,plainpages=false,colorlinks=false,pdfpagelabels]{hyperref}

\theoremstyle{plain}
\newtheorem{The}{Theorem}
\newtheorem*{The*}{Theorem}
\newtheorem{Pro}{Proposition}[section]
\newtheorem{Lem}[Pro]{Lemma}
\newtheorem{Cor}[Pro]{Corollary}

\newtheorem*{Cor*}{Corollary}

\theoremstyle{definition}
\newtheorem{Def}[Pro]{Definition}
\newtheorem{Rem}[Pro]{Remark}

\newtheorem{Exa}[Pro]{Example}
\newtheorem*{Exa*}{Example}
\newtheorem*{Rem*}{Remark}

\numberwithin{equation}{section}

\DeclareMathOperator{\Hom}{Hom}

\DeclareMathOperator{\SL}{SL}

\DeclareMathOperator{\PSU}{PSU}

\DeclareMathOperator{\SU}{SU}          

\DeclareMathOperator{\Id}{Id}

\DeclareMathOperator{\Ad}{Ad}

\DeclareMathOperator{\im}{im}
\DeclareMathOperator{\Li2}{Li_2}
\DeclareMathOperator{\li}{Li}
\DeclareMathOperator{\tr}{tr}
\DeclareMathOperator{\diag}{diag}

\renewcommand{\Im}{\operatorname{Im}}
\renewcommand{\Re}{\operatorname{Re}}

\newcommand{\dbar}{\bar\partial}
\newcommand{\del}{\partial}

\DeclareMathOperator{\Aut}{Aut}

\DeclareMathOperator{\Jac}{Jac}
\DeclareMathOperator{\Res}{Res}

\newcommand{\R}{\mathbb{R}}

\newcommand{\D}{\mathbb{D}}
\newcommand{\C}{\mathbb{C}}
\newcommand{\N}{\mathbb{N}}
\newcommand{\Z}{\mathbb{Z}}
\renewcommand{\S}{\mathbb{S}}
\newcommand{\RP}{\mathbb{RP}}
\newcommand{\CP}{\mathbb{CP}}

\newcommand{\sslash}{\mathbin{/\mkern-6mu/}}

\begin{document}

\author{Sebastian Heller}
\author{Charles Ouyang}
\author{Franz Pedit}
\title[Embedded special Legendrian surfaces]{Embedded special Legendrian surfaces in $\S^5$}

\date{\today}
\maketitle
\begin{abstract}
We construct the first smooth embedded compact special Legendrian surfaces in
\(\S^5\) of genus greater than one.  More precisely, for every sufficiently
large integer \(k\), we construct an embedded special Legendrian surface whose
conformal structure is the Fermat curve of degree \(k\) and genus
\(\tfrac12(k-1)(k-2)\).
Our approach combines an elementary implicit function theorem with the description of special Legendrian surfaces via loop algebra-valued meromorphic connections and a characterization of the unitarizability locus in the ${SL}_{3}(\C)$-character variety of the thrice-punctured sphere.
\end{abstract}
\setcounter{tocdepth}{1}

\section{Introduction}

Calabi--Yau manifolds have occupied a central role in both differential geometry and theoretical physics. In complex dimension three, they are particularly prominent because the Strominger--Yau--Zaslow conjecture \cite{SYZ} predicts that, in the large complex structure limit, they admit fibrations by special Lagrangian three-tori.
 From this perspective, certain point singularities of Calabi--Yau threefolds can be viewed as arising by ``pinching off'' a generator of the 3-torus fiber. Upon rescaling near the singularity, one expects to obtain a special Lagrangian cone in $\C^{3}$, whose link is a closed surface. While it is expected that generically the link will have genus one, higher genus links cannot completely be ruled out. Indeed, Joyce \cite{Joyce2} defined a stability index, showing that eigenvalues in a certain range of the induced Laplacian on the link will correspond to how likely its cone will appear as a singularity---regardless of genus. Intersecting the cone with the round 5-sphere $\S^5$ yields a special Legendrian surface, which, being calibrated, is automatically minimal \cite{HL}. Taking the quotient under the Hopf projection then produces a minimal Lagrangian surface in $\CP^{2}$. Conversely, starting from a minimal Lagrangian surface in $\CP^{2}$, one may lift---possibly after passing to a threefold cover---to a special Legendrian surface in $\S^{5}$, whose cone in $\C^3$ is special Lagrangian. Hence, constructing special Lagrangian cones in $\C^{3}$, special Legendrian surfaces in $\S^5$, and  minimal Lagrangian surfaces in  $\CP^{2}$  is equivalent.

From a differential-geometric point of view, minimal Lagrangian surfaces in $\CP^2$ are natural objects because they combine the Riemannian and symplectic geometry of the K\"ahler manifold $\CP^2$. In genus zero the theory is rigid: an immersed minimal Lagrangian $2$-sphere in $\CP^2$ is the double cover of a totally geodesic $\RP^2$ \cite{Yau}. More generally, classical holomorphic-differential methods lead to the classification of minimal $2$-spheres in spheres and complex projective spaces \cite{Calabi,Calabi2,Chern,EW}. In genus one, by contrast, integrable-systems methods become effective: minimal Lagrangian tori in $\CP^2$ admit a complete algebro-geometric description in terms of spectral curves and linear flows in Jacobians \cite{Sharipov,CastroUrbano,Ha,MM,CMcI}. The Clifford torus is the simplest example.

Beyond genus one, the theory is far less developed.  For minimal
Lagrangian surfaces in \(\CP^2\), Joyce emphasized the lack of
compact higher-genus examples \cite{Joyce}.  Prior to the present work, the
only smooth compact higher-genus examples known were those of
Haskins and Kapouleas \cite{HaK}, obtained by gluing special Legendrian
cylinders to an equatorial \(2\)-sphere in \(\S^5\).  Their cyclic symmetry
forces the genus to be odd, except for one tetrahedrally symmetric example
of genus four, and embeddedness of the corresponding special Legendrian
surfaces in $\S^5$ was not established.
One should also mention Wang's variational construction \cite{Wang}, which was
proposed as a method to obtain compact special Legendrian surfaces
of high genus.  
However, the result only provides possibly branched
Hamiltonian stationary Lagrangian surfaces in \(\CP^2\).

The examples constructed in this paper appear to be the first smooth embedded compact
special Legendrian surfaces in \(\S^5\) of genus greater than one. 
Our main result is the following.
\begin{The*}
There exists \(k_0\in\N\) such that for every \(k\geq k_0\) there exists
an embedded special Legendrian surface in $\S^{5}$ of genus 
$g = \tfrac{(k-1)(k-2)}{2}$ whose  conformal type is that of the Fermat curve.
\end{The*}
In particular, there are infinitely many new examples of minimal
Lagrangian surfaces in \(\CP^2\), including infinitely many examples of
even genus.  And when \(g\) is odd, the Fermat-curve symmetries show that our
examples are distinct from those of Haskins--Kapouleas \cite{HaK}.  In fact, the
areas of our examples grow on the order of \(\sqrt g\), in contrast with
the linear growth forced by the handle-gluing construction of
\cite{HaK}.  Moreover, the Haskins--Kapouleas construction requires small
neck sizes and hence conformal structures near the boundary of the moduli
space, whereas our surfaces have the conformal structure of the Fermat
curve itself.

Refining our implicit function theorem argument in the parameter $t=1/k$, we expect the threshold in the theorem above to be $k_0=3$, 
corresponding to the Clifford torus in $\S^5$. This expectation is supported both by the analysis in \cite{CHHT}, which reconstructs Lawson's 
minimal surfaces $\xi_{1,g}$ for all $g\ge 3$ using Fuchsian systems, and by numerical experiments based on our construction. We also expect that our
Fuchsian system description will lead to power series expansions of the areas in $1/g$, with alternating multiple zeta values as 
coefficients, analogous to those found in \cite{CHHT}.

\medskip
\noindent
{\bf Overview.}
The construction of our surfaces is guided by a general philosophy familiar both from integrable surface geometry \cite{Hi,DPW,H} and from the twistorial viewpoint on non-Abelian Hodge theory via $\lambda$-connections \cite{Simpson1997,Sim22}: instead of attacking the nonlinear Gauss-Codazzi type equations directly, one encodes the geometry in a holomorphic family of flat connections depending on a spectral parameter. 
For minimal Lagrangian surfaces in $\CP^2$, this leads to a cyclic Higgs bundle and an associated $\Z_6$-twisted family
of flat connections
\begin{equation}\label{eq:dlambdafamily}
d_{\lambda}=D+\lambda^{-1}\Phi-\lambda\Phi^{\dagger},\qquad \lambda\in\C^\times,
\end{equation}
whose reality and symmetry properties encode the minimal Lagrangian condition. The $\tau$-twisted formalism
developed in Section~\ref{sec:loop-groups} is the natural loop-group framework for families of flat connections with cyclic Higgs data \cite{Bar, Li}: its eigenspace decomposition reflects the cyclic form of $\Phi$ together with the additional symmetry coming from the Lagrangian condition. Conversely, such a family reconstructs the immersion once suitable intrinsic and extrinsic closing conditions are imposed.
The surface-theoretic and the basic gauge-theoretic background for this picture is explained in Section~\ref{sec:MLgeneral}.

Despite this formal similarity, the relevant gauge-theoretic equations are not Hitchin's self-duality equations: the signed self-duality equations differ from Hitchin's equations by the sign of the quadratic Higgs term. Thus the classical non-Abelian Hodge correspondence does not by itself produce the relevant harmonic map, or the closed minimal Lagrangian immersion, from the holomorphic Higgs data. Instead, the main task is to construct a $\lambda$-family of flat connections with prescribed local asymptotics and monodromy, satisfying the intrinsic unitarization condition for $\lambda\in\S^1$ and the extrinsic closing condition at a distinguished {\em Sym point}. Once such a family has been constructed, one applies the DPW recipe \cite{DPW,DM}: loop-group Iwasawa factorization recovers the unitary associated family, and hence the harmonic map and the minimal Lagrangian surface itself.

In this sense, our construction behaves like a local non-Abelian Hodge correspondence for a compact symmetric target: once the implicit function theorem produces a unique meromorphic $\lambda$-family for a given Higgs field which satisfies the required monodromy properties, those properties already determine the 
 geometry. In the case at hand, the cyclic Higgs field essentially forces the $\Z_6$-twisted form of the resulting family of flat connections, and,
after solving the monodromy problem,  leads to minimal Lagrangian surfaces---a phenomenon well known in higher Teichm\"uller theory
\cite{HiTeich,Sim}. What fails, in comparison with the classical non-Abelian Hodge correspondence, is that the signed self-duality equations do not arise from Hitchin's equations, so the relevant harmonic metric is not obtained abstractly from the holomorphic Higgs data, but must instead be constructed indirectly through this monodromy problem. On the other hand, a complex analytic, monodromy-based approach to the non-Abelian Hodge correspondence was already suggested by Deligne and Simpson (see \cite{Simpson1997}), and has been partially implemented in certain situations \cite{GMN,HHT1,HHM}.

The existence problem of minimal Lagrangian surfaces is therefore reformulated 
as a non-Abelian period (monodromy) problem. One seeks  families of meromorphic $\SL_3(\C)$-connections whose singularities become apparent on the relevant branched cover, whose monodromy is unitarizable for every $\lambda\in \S^1$, and whose monodromy at a distinguished Sym point satisfies the extrinsic closing condition. 
For a general higher genus surface this is a difficult global problem, 
governed by  the Riemann--Hilbert map from the moduli space of flat connections into the
 $\SL_3(\C)$-character variety. 
We simplify this by imposing the symmetries of the Fermat curve
\[
\Sigma_k=\{X^k+\zeta Y^k+\zeta^2 Z^k=0\}\subset \CP^2,\qquad \zeta=\exp(\tfrac{2\pi i}{3}),
\]
whose symmetry group
$
\Gamma_{sym}=\Z_k\times\Z_k\rtimes \Z_3
$
realizes $\Sigma_k$ as a branched cover of $\CP^1$ with three branch values. This reduces the global monodromy problem from a higher-genus surface to the relative $\SL_3(\C)$-character variety of the thrice-punctured sphere
$
S=\CP^1\setminus\{0,1,\infty\}.
$

We therefore work on $S$ and construct a $t$-dependent family, with $t=1/k$, of irreducible Fuchsian systems
\[
D_\lambda^t=d+\eta_\lambda^t
\]
with local conjugacy classes prescribed by the covering data. Using Lawton's description of the relative character variety in trace coordinates \cite{Law}, we identify a unitary locus explicitly and solve the intrinsic unitarization condition, together with the extrinsic closing condition at a distinguished Sym point, by an implicit function theorem starting from explicit initial data at $t=0$. There are two suitable initial conditions, leading to two geometrically distinct families of
minimal Lagrangian surfaces of the same conformal type and with the same ambient symmetries. They can be distinguished by
their area, which can be computed in terms of the residues of the meromorphic connection data.
The character variety analysis and the implicit function theorem argument are carried out in Sections~\ref{sec:chava} and~\ref{sec:solvingIFT}, respectively.
After pulling back to $\Sigma_k$ and removing the resulting apparent singularities, we obtain compact minimal Lagrangian immersions
$
f\colon \Sigma_k\to\CP^2
$
for all sufficiently large $k$, carrying the symmetries of the Fermat curve.
These surfaces and their basic geometric properties are studied in Section~\ref{sec:construction}.

It is worth noting that our construction completely avoids delicate gluing techniques. We exploit the prescribed symmetries of the desired surfaces to construct meromorphic loop Weierstrass data from which the surfaces are then derived. This approach is very much in the spirit of classical minimal surface theory, where one also uses a priori knowledge of symmetries and conformal type to determine the Weierstrass data. Integration of these data in the classical case leads to an Abelian period closing problem which, in our setting, becomes a non-Abelian unitarization problem for representations into the twisted loop group ${\bf\SL}^{\tau}(\S^1)$.  

The final part of the paper is devoted to the study of the special Legendrian lifts of the minimal Lagrangian surfaces constructed above. Since topological obstructions prevent higher-genus oriented Lagrangian surfaces from being embedded in $\CP^2$, embeddedness can only be expected after passing to the Legendrian lift in $\S^5$. We show that for one of the two families, and for all sufficiently large $k$, the corresponding special Legendrian lifts
$
\hat f_k\colon \Sigma_k\to \S^5
$
are embedded.

Our embeddedness arguments rely on an explicit asymptotic analysis as $t=1/k\to 0$, equivalently as $g\to\infty$, of solutions to the Fuchsian equation
$
d\Psi_{\lambda}^t+\eta_{\lambda}^{t}\Psi_{\lambda}^t=0
$
on the intermediate quotient
$
\Sigma_k/(\Z_k\times\Z_k)\cong \CP^1,
$
punctured at the three branch values $\{1,\zeta,\zeta^2\}$. Expanding the solutions by iterated integrals, and carrying out the corresponding loop Iwasawa factorization to sufficiently high order, yields precise asymptotic control of the unitary frame. At a distinguished Sym point, this frame closes on $\Sigma_k$ and gives the special Legendrian lift.
Two asymptotic regimes appear. Away from the branch values of 
$
\Sigma_k/(\Z_k\times\Z_k)$, the blow-up limit is an embedded Scherk-type minimal Lagrangian surface in $\C^2$, identified with a horizontal tangent space of the Hopf fibration $\S^5\to\CP^2$. 
This blow-up occurs at each point $q\in\S^5$ which lies on the Cliffod torus. Near each of the three branch values of our covering, the surfaces converge locally to special Legendrian spherical caps. The main difficulty is to control uniformly the transition between these two 
regions. This is achieved by uniform first-order estimates for the frame in $t$, and it yields embeddedness for all sufficiently large $k$. 
The asymptotic analysis and the embeddedness proof are carried out in Section~\ref{sec:em}.

\section*{Acknowledgments}{
The first author was supported by the Beijing Natural Science Foundation grant IS23003.  The second author  was funded in part from the National Science Foundation through DMS-2202832. The third author was supported by a research visiting program from SMRI at the University of Sydney. 
All three authors thank the TU Berlin, especially the SFB/TRR 109, for their hospitality and support during research visits, where parts of this work were carried out.}

\section{The fundamental equations for minimal Lagrangian surfaces}\label{sec:MLgeneral}

\subsection{Minimal Lagrangian cones in $\C^3$, special Legendrian surfaces in $\S^5$, and minimal Lagrangian surfaces in $\CP^2$}
For background on calibrations, special Lagrangian geometry, and Legendrian geometry, see \cite{HL,HiSL,Joyce}. We restrict attention to the three-dimensional case and follow the exposition in \cite{Ha,HaK}.

Consider $\C^3$ with complex coordinates $z=(z_1,z_2,z_3)$
equipped with the standard Euclidean metric. The  corresponding K\"ahler form is
\[
\omega_{\C^3}=\tfrac{i}{2}\sum_j dz_j\wedge d\bar z_j=\sum_j dx_j\wedge dy_j
\]
where
$z_j=x_j+i y_j$, and the holomorphic volume form is
\[
\Omega_{\C^3}=dz_1\wedge dz_2\wedge dz_3.
\] 
The real 3-form $\Re\Omega_{\C^3}$ is a calibration on $\C^3$, that is, $\Re\Omega_{\C^3}$ is closed and for every oriented real 3-dimensional subspace $W\subset \C^3=T_z\C^3$ one has
$(\Re\Omega_{\C^3})|_W\leq\det_W$
where $\det_W\in\Lambda^3 W^*$ is the oriented volume form induced from the Euclidean metric. An immersion $f\colon N^3\to \C^3$ of a connected, oriented real 3-manifold $N$ is  {\em special Lagrangian} if it is Lagrangian, that is $f^*\omega_{\C^3}=0$, and $f^*\Im\Omega_{\C^3}=0$. This is  equivalent \cite{HL} to 
\[
f^*\Omega_{\C^3}= {\det}_N
\]
 with $\det_N$ the volume form of the induced metric on $N$. 
 Thus $f^*\Re\Omega_{\C^3}=\det_N$, so $f$ is calibrated by $\Re\Omega_{\C^3}$ and hence is minimal.

 On the other hand, if $f\colon N\to\C^3$ is a Lagrangian immersion then $f^*\Omega_{\C^3}=\varphi \det_N$ for a phase $\varphi\colon N\to\S^1\subset \C$. Now
 \[
 i\,d\log \varphi=\omega_{\C^3}(H,df)
 \]
by \cite{HL} with $H$ the mean curvature of  $f$ and thus the Lagrangian immersion $f\colon N\to \C^3$ is minimal if and only if its phase $\varphi$ is constant.  Therefore, up to a $\mathrm U_3$-transformation, a minimal Lagrangian immersion $f\colon N\to \C^3$  is special Lagrangian.
 
Now consider the cone $C\colon (0,\infty)\times \Sigma\to \C^3$ 
over an immersed surface $\hat{f}\colon \Sigma\to \S^5$. 
Then $C$ is Lagrangian if and only if $\hat{f}$ is Legendrian, that is, $\hat{f}^*\alpha=0$, where
\[
\alpha=\omega_{\C^3}(X,\cdot)\big|_{\S^5}
\]
is the standard contact form on $\S^5$ and
\[
X=\tfrac{1}{2}\sum_j \left(x_j\frac{\partial}{\partial x_j}+y_j\frac{\partial}{\partial y_j}\right)
\]
is the Liouville vector field on $\C^3$.
 In this case the phase $\varphi\colon (0,\infty)\times \Sigma\to \S^1$ only depends on $\Sigma$, and  $\varphi\colon \Sigma\to \S^1$ satisfies
 \begin{equation}\label{eq:sLeg}
 \hat{f}^*(\Omega_{\C^3}(X,\cdot))=\varphi\,{\det}_\Sigma\,.
 \end{equation}
 As in the Lagrangian case, the Legendrian immersion $\hat{f}\colon \Sigma\to \S^5$ is minimal if and only if $\varphi$ is constant. 
 Therefore, the link in $\S^5$ of a special Lagrangian cone in $\C^3$ is special Legendrian, and conversely the cone over a special Legendrian surface is special Lagrangian.
 
 The Hopf fibration $\pi\colon \S^5\to\CP^2$ can be seen as the unit circle bundle in the tautological line bundle $L\to \CP^2$. Its connection 1-form is the contact structure $\alpha$ on $\S^5$ with curvature a multiple of the K\"ahler form $\omega$ of $\CP^2$.
 In particular, parallel unit length sections of $L$  take  values in the horizontal subbundle $\ker\alpha\subset T\S^5$. Therefore, a Legendrian immersion $\hat{f}\colon \Sigma\to \S^5$ projects via $\pi$ to a Lagrangian immersion $f\colon \Sigma\to \CP^2$. Since the Hopf fibration $\pi\colon \S^5\to \CP^2$ is a Riemannian submersion, minimal Legendrians in $\S^5$ project to minimal Lagrangians in $\CP^2$.  
 
 Lifting a Lagrangian immersion $f\colon \Sigma\to \CP^2$ via $\pi$ is somewhat more subtle: being Lagrangian, $f^*\omega=0$ implies that the pullback of the tautological bundle, again denoted by $L\to\Sigma$, is a flat Hermitian line bundle.
 Prescribing a unit-length vector at some fiber defines, via parallel transport, a unique Legendrian immersed lift $\tilde{f}\colon \tilde{\Sigma}\to \S^5$ on the universal cover of $\Sigma$.  If one additionally assumes $f$ to be minimal, then $L\to \Sigma$ satisfies $L^3=\underline{\C}$ and thus the Legendrian lift of a minimal Lagrangian immersion closes on at most a threefold cover
 $\hat{\Sigma}$ of $\Sigma$.  To summarize:
 \begin{Pro}\label{pro:upanddown}
 Minimal Legendrian immersions, and in particular special Legendrian immersions,
  $\hat{f}\colon \Sigma\to \S^5$ project via the Hopf fibration $\pi\colon \S^5\to\CP^2$ to a minimal Lagrangian immersion $f\colon \Sigma\to \CP^2$. Conversely, a minimal Lagrangian immersion $f\colon \Sigma\to\CP^2$
 horizontally lifts, on at most a threefold cover $\hat{\Sigma}$, to a minimal, not necessarily special, Legendrian immersion $\hat{f}\colon \hat{\Sigma}\to \S^5$. Choosing the initial condition for the lift appropriately, one can arrange for $\hat{f}$ to be special Legendrian. 
 \end{Pro}
 
\subsection{The fundamental equations of minimal Lagrangian surfaces in $\CP^2$}
Let $L\subset \underline{\C}^{3}$ denote the tautological line bundle over the complex projective plane $\CP^2$.  
Then the trivial bundle
\[
\underline{\C}^{3}=L\oplus L^{\perp}
\]
splits as an orthogonal sum with respect to the standard Hermitian product $\langle\,,\,\rangle$ on $\C^{3}$,  which we assume to be sesquilinear in the first entry. In this decomposition, the trivial connection has the form 
\begin{equation}\label{eq:d}
d=\begin{pmatrix} \nabla&-\delta^{\dagger}\\ \delta& \nabla^{\perp}\end{pmatrix}
\end{equation}
with $\nabla$, $\nabla^{\perp}$ Hermitian connections on $L$, $L^{\perp}$ and 
\begin{align*}
\delta\colon T\CP^2 &\to \Hom(L,L^{\perp}),\quad
v\mapsto (d_v)^{\perp}
\end{align*}
the bundle isomorphism identifying the tangent bundle of $\CP^2$ with the bundle $\Hom(L,L^{\perp})$. The K\"ahler structure of $\CP^{2}$ is given by the complex structure $J(v)=iv$ and  the Hermitian metric 
\begin{equation}\label{hermitian_metric}
h_{\CP^2}(v,w)=\tr({v^{\dagger}w}),\qquad v,w\in T\CP^2=\Hom(L,L^{\perp})\,.
\end{equation}
The real and imaginary parts 
\[
h_{\CP^2}=g_{\CP^2}+i\,\omega
\]
yield  the Fubini-Study metric $g_{\CP^2}$  and the K\"ahler 2-form $\omega$. The Levi-Civita connection of $g_{\CP^2}$ is the connection $\nabla^*\otimes\nabla^{\perp}$ on $\Hom(L,L^{\perp})$.

A map $f\colon \Sigma \to \CP^2$ from a Riemann surface $\Sigma$ corresponds---via pullback of the tautological bundle over $\CP^2$---to a line subbundle $L\subset \underline{\C}^{3}$ of the trivial bundle over $\Sigma$. Then $df=f^*\delta$ is the derivative of $f$  and we have the type decomposition
\begin{equation}\label{eq:df-type}
\begin{gathered}
df=\del f+\dbar f\in \Gamma(K\Hom(L,L^{\perp}))\oplus \Gamma(\bar{K}\Hom(L,L^{\perp}))\,,\\
\del f=\tfrac{1}{2}(df-i*df)\,,\qquad \dbar f =\tfrac{1}{2}(df+i*df)\,.
\end{gathered}
\end{equation}
Here $K$, $\bar{K}$ denote the canonical and anti-canonical line bundles of $\Sigma$ and 
 $*=(J_\Sigma)^*\colon T\Sigma^*\to T\Sigma^*$ is the complex structure of $\Sigma$.
The Hermitian connections $\nabla$, $\nabla^{\perp}$  induce -- via type decomposition -- holomorphic and anti-holomorphic structures on the bundles $L$, $L^{\perp}$.  We have the following characterizations:
\begin{enumerate}
  \refstepcounter{equation}
  \makeatletter \def\theenumi{\theequation\@alph\c@enumi} \makeatother
 \item \label{2.4a} $f$ is (branched) conformal if and only if $h_{\CP^2}(\dbar f,\del f)=0$, in which case
\[
2|\del f|^2=|df|^2+\omega(df,*df)\quad\text{and}\quad 2|\dbar f|^2 =|df|^2-\omega(df,*df)\,;
\]
\item\label{2.4b}
$f$ is Lagrangian, that is $f^*\omega=0$, if and only if $|\del f|^2=|\dbar f|^2$;
\item\label{2.4c}
$f$ is (branched) conformal and Lagrangian if and only if $
|\del f|^2=|\dbar f|^2 =\tfrac{1}{2}|df|^2$.
\end{enumerate}
Here  $|\alpha |^2=g_{\CP^2}(\alpha,\alpha)\in\Gamma(|K|^2)$ denotes the 
quadratic form  $|\alpha |^2_v=g_{\CP^2}(\alpha(v),\alpha(v))$, $v\in T\Sigma$,  induced via  $\alpha\in\Omega^1(\Sigma, f^* T\CP^2)$. 

The following properties of harmonic maps into $\CP^2$  are well-documented in the literature \cite{durhamboys,BPW}:
{\bf (1)}
A map $f\colon \Sigma\to \CP^2$ is {\em harmonic}  if and only if $\del f \in \Gamma(K\Hom(L,L^{\perp}))$ is holomorphic, equivalently, 
$\dbar f\in \Gamma(\bar{K}\Hom(L,L^{\perp}))$ is anti-holomorphic. If neither of those bundle maps is trivial, that is, $f$ is neither anti-holomorphic nor holomorphic, then there exist unique line subbundles  $L_{\pm}\subset  L^{\perp}$ with $\del f\in \Gamma(K\Hom(L,L_{+}))$ and $\dbar f\in \Gamma(\bar{K}\Hom(L,L_{-}))$. The corresponding maps $f_{\pm}\colon \Sigma\to \CP^{2}$ are again harmonic and $\dbar f=-(\del f_{-})^{\dagger}$,  $\dbar f_{+}=-(\del f)^{\dagger}$.

{\bf (2)}
For a harmonic map $f\colon \Sigma\to \CP^2$, the quadratic Hopf differential 
\[
Q_2=(f^*g_{\CP^2})^{2,0} = h_{\CP^2}(\dbar f,\del f)\in H^0(K^2)
\]
is holomorphic. By \eqref{2.4a}, its vanishing $Q_2\equiv0$ characterizes  the conformality of the harmonic map $f\colon \Sigma\to \CP^2$, which thus becomes a branched minimal immersion. Assuming that $f$ is neither holomorphic nor anti-holomorphic, we have the orthogonal splitting 
\[
\underline{\C}^3=L_{-} \oplus  L\oplus L_{+}
\]
of the trivial $\C^3$-bundle over $\Sigma$.  The derivatives of the harmonic maps  $f_{\pm}, f$ give rise to the harmonic sequence 
 \begin{equation}\label{eq:harmseq}
\begin{tikzcd}
L_{-}  \arrow[r, "\del f_{-}"] & KL  \arrow[r, "\del f"] & K^2 L_{+} \arrow[r, "\del f_{+}"] & K^3 L_{-} 
\end{tikzcd}
\end{equation}
and the cubic differential 
\begin{equation}\label{eq:cubic}
Q=-h_{\CP^2}(\dbar f,\del f_{+}\del f)=\tr(\del f_{+}\del f\del f_{-}) \in H^0(K^{3})
\end{equation}
is holomorphic. Either $Q\equiv0$, in which case $f_{+}$ is anti-holomorphic and the sequence \eqref{eq:harmseq} terminates -- the  {\em superminimal} 
case \cite{durhamboys} -- or $Q\not\equiv 0$, and the  sequence \eqref{eq:harmseq}  is periodic with period three -- the {\em superconformal} case \cite{BPW}. 
Since the main focus of this paper is on minimal Lagrangian immersions 
$f\colon \Sigma\to \CP^2$ from a Riemann surface $\Sigma$, we collect the basic characteristics of their harmonic sequences \eqref{eq:harmseq}.

\begin{Lem}\label{lem:mLag}
Let $f\colon \Sigma\to \CP^2$ be a minimal Lagrangian immersion. Then the following properties hold:
\begin{enumerate}[(i)]
\item
 The holomorphic cubic differential \eqref{eq:cubic} vanishes identically if and only if $f(\Sigma)\subset \RP^2$. If $\Sigma$ is compact and of nonzero genus, the minimal Lagrangian immersion $f$  is superconformal, $Q\not \equiv 0$, and the harmonic sequence \eqref{eq:harmseq} is periodic.
\item
The quadratic forms
\(
\tfrac{1}{2}|df|^2=|\del f|^2=|\dbar f|^2 =|\del f_{-}|^2=|\dbar f_{+}|^2
\)
are nowhere vanishing.  
\item
The bundle maps
\begin{equation}\label{eq:isomorph}
\del f_{-}\colon L_{-}\cong KL\quad\text{and}\quad \del f\colon L \cong KL_{+}
\end{equation}
are holomorphic isometries if the Hermitian metric on $K$ comes from the induced metric $\tfrac{1}{2}f^*g_{\CP^2}$ on $\Sigma$.
\item
The zeros of $Q$ are the same as the zeros of $\del f_+$, 
i.e.  the divisors  $(Q)=(\del f_+)$ agree.
\end{enumerate}
\end{Lem}
\begin{proof}
As shown in \cite{LoftinMcI}, a  minimal Lagrangian immersion $f$ has  $Q=ig_{\CP^2}(S,Jdf)^{(3,0)}$, where $S$ denotes the normal bundle valued second fundamental form of $f$. Since $g_{\CP^2}(S,Jdf)^{(2,1)} =0$ by minimality,  $Q\equiv0$ is equivalent to  the vanishing of the symmetric 3-form $g_{\CP^2}(S,Jdf)=0$. The Lagrangian condition then implies $S\equiv 0$ and the map $f$ is totally geodesic which shows (i).

The relations $\tfrac{1}{2}|df|^2=|\del f|^2=|\dbar f|^2$ follow from \eqref{2.4b}, \eqref{2.4c}. Furthermore,  
$\dbar f_{+}=-(\del f)^{\dagger}$ and $\del f_{-}=-(\dbar f)^{\dagger}$, which provide the remaining  identities in (ii).
Since $f$ is assumed to be immersed, $|df|^2$ is nowhere vanishing. 

Item (iii) follows from (ii), and item (iv) follows from
$
Q=\tr(\del f_{+}\del f\del f_{-})
$
together with (iii).
\end{proof}
The holomorphic isomorphisms \eqref{eq:isomorph} identify $L_{-}\cong KL$ and $L_{+}\cong K^{-1}L$ so that the orthogonal decomposition
\begin{equation}\label{eq:C3decom}
\underline{\C}^3=L_{-}\oplus L\oplus L_{+}
\end{equation}
gives rise to the trivial rank $3$ bundle 
\begin{equation}\label{eq:cyclic-bundle}
V=KL\oplus L\oplus K^{-1}L\,.
\end{equation}
Via these identifications, the Hermitian metric on $V$ is diagonal: the metrics on $KL$ and $K^{-1}L$ are the product metrics coming from the Hermitian metrics on $K$ and $K^{-1}$, respectively, induced by $\tfrac12 f^*g_{\CP^2}$ \footnote{The factor $1/2$ is due to item (ii) of Lemma \ref{lem:mLag}.}, together with the induced Hermitian metric on $L\subset\underline{\C}^3$.
 The Hermitian connections on $L_{\pm}$ hence become  $\nabla_{\pm}=\nabla^{K^{\mp 1}}\otimes\nabla$ with $\nabla^{K^{\pm 1}}$ induced by the Levi-Civita connection of $\tfrac12f^*g_{\CP^2}$ on $\Sigma$ and $\nabla$ is the Hermitian connection on $L$ induced via the decomposition \eqref{eq:C3decom}.

We can now reformulate and refine the decomposition \eqref{eq:d} of the flat connection $d$ on the trivial
$\C^3$-bundle over $\Sigma$ in the orthogonal splitting \eqref{eq:cyclic-bundle} of $V$ as
\begin{equation}\label{eq:d-cyclic}
d_V=D+\Phi - \Phi^{\dagger}\,.
\end{equation}
Here  $D=\nabla_{-}\oplus\nabla\oplus \nabla_{+}$ is the Hermitian diagonal connection of the splitting \eqref{eq:cyclic-bundle} and 
\begin{equation}\label{eq:Higgsfield}
\Phi=\begin{pmatrix} 0 &0&Q\\ {\bf 1}&0&0 \\ 0&{\bf 1}& 0\end{pmatrix}\in \Gamma(\Sigma,K{\mathfrak{sl}}(V))
\end{equation}
is a cyclic Higgs field, see \cite{Bar,Li}. Note that due to \eqref{eq:cyclic-bundle}, the entry ${\bf 1}\in\C$ is the constant section of the trivial bundle $\underline{\C}=K\Hom(KL,L)=K\Hom(L,K^{-1}L)$ and the entry $Q\in H^0(K^3)$ is the holomorphic cubic differential \eqref{eq:cubic} since $K^3=K\Hom(K^{-1}L, KL)$. The adjoint of the Higgs field $\Phi$, with respect to the splitting \eqref{eq:cyclic-bundle}, is given by
\begin{equation}\label{eq:Higgsadj}
\Phi^{\dagger}=
\begin{pmatrix} 0 &\tfrac{1}{2}h&0\\ 0&0&\tfrac{1}{2}h \\ \tfrac{4\bar{Q}}{h^{2}}&0& 0\end{pmatrix}\in \Gamma(\Sigma,\bar{K}{\mathfrak{sl}}(V))\,,\end{equation}
where we view the conformal metric $h=|df|^2\in\Gamma(\Sigma,\bar KK)$ on $\Sigma$ as a quadratic form, i.e. as a section of \[|K|^2=\bar{K}\Hom(L,KL)=\bar{K}\Hom(K^{-1}L,L).\] Similarly, $4\bar{Q} h^{-2}$ is a section of $\bar{K}^3(K\bar{K})^{-2}= \bar{K}\Hom(KL,K^{-1}L)$ as required. 
From \eqref{eq:Higgsfield} and \eqref{eq:Higgsadj} we see that branch points of $f$ are characterized by the vanishing of $\Phi^2$. Since $f$ is assumed to be immersed, $\Phi^2\in\Gamma(K^2\mathfrak{sl}(V))$ is nowhere vanishing.

The zero curvature equation for the connection $d_V$ splits into diagonal and off-diagonal parts
\begin{equation}\label{eq:zero-curvature}
F^D=[\Phi\wedge \Phi^{\dagger}]\quad\text{and} \quad d^{D}(\Phi -\Phi^{\dagger}) =0\,.
\end{equation}
The holomorphicity of the cubic differential $Q$ is equivalent to the holomorphicity of the Higgs field, $\dbar^{\,D}\Phi=0$. The latter implies the off-diagonal part of the zero curvature equation since the diagonal connection $D$ is Hermitian. Thus a minimal Lagrangian immersion $f\colon \Sigma\to \CP^2$ gives rise to a solution to a special reduction of the {\em signed self-duality equations}
\begin{equation}\label{eq:signedSD}
F^D=[\Phi\wedge \Phi^{\dagger}]\quad\text{and} \quad \dbar^{\,D} \Phi=0\,
\end{equation}
for the group $\SU_3$. 
Despite their formal similarity, these equations are not Hitchin's self-duality equations \cite{HiSD} due to the opposite sign in the quadratic term of the first equation. 
To unravel the first equation of \eqref{eq:signedSD}, we identify 2-forms with quadratic forms via 
\begin{equation}\label{eq:2forms}
\alpha\wedge \beta=\alpha (*\beta)-(*\alpha)\beta
\end{equation}
for 1-forms $\alpha,\beta$ on $\Sigma$. In particular, under this prescription the area 2-form of a conformal metric on $\Sigma$ gets identified with the quadratic form of the metric. We thus obtain
\begin{equation}\label{eq:diagSD1}
[\Phi\wedge \Phi^{\dagger}]=(-2i)[\Phi,\Phi^{\dagger}]=i\bigl((|df|^2-8|Q|^2|df|^{-4})\oplus 0\oplus (-|df|^2+8|Q|^2|df|^{-4})\bigr)
\end{equation}
and the diagonal connection $D$ has curvature
\[
F^D=F^{\nabla_{-}}\oplus F^{\nabla}\oplus F^{\nabla_{+}}\,.
\]
Hence the first equation of \eqref{eq:signedSD} is equivalent to 
\begin{equation}\label{eq:diagSD}
F^{\nabla}=0\quad\text{and}\quad F^{\nabla^K}=i(|df|^2-8|Q|^2|df|^{-4})\,,
\end{equation}
where we used that $F^{\nabla^{K^{-1}}}=-F^{\nabla^{K}}$ and $F^{\nabla_{-}}=F^{\nabla{^{K}}\otimes \nabla}=F^{\nabla^{K}}$ since $F^{\nabla}=0$.
The curvature of the Hermitian line bundle $K$ is $F^{\nabla^{K}}=iG|df|^2$, where $G$ denotes the Gauss curvature of the induced metric $f^*g$.
We obtain
\[
G+|Q|^2_{f^*g}-1=0\,.
\]
A more familiar form of this equation may be obtained by expressing the induced metric $f^{*}g_{\CP^2}=e^{2u} g_0$ in a fixed conformal background metric $g_0$ on $\Sigma$. Then $G=(-\triangle_{g_0}u+G_{g_0})e^{-2u}$ and the first equation of \eqref{eq:signedSD} becomes the well-known (signed) {\em Tzitz{\'e}ica equation}
\begin{equation}\label{eq:Tzitz}
\triangle_{g_0} u-|Q|^2_{g_0} e^{-4u}+e^{2u}-G_{g_0}=0
\end{equation}
for the conformal factor $u\colon \Sigma\to\R$.

It is worth pointing out that the flatness of $\nabla$ on $L$ should not come as a surprise: the Lagrangian condition $f^*\omega=0$ is equivalent to the vanishing of the curvature of the pullback, via $f\colon \Sigma\to \CP^2$,  of the
dual  tautological bundle over $\CP^2$.  Moreover, by construction of $V$ its determinant bundle $\Lambda^3 V$ is trivial as a line bundle with Hermitian connection and therefore 
\begin{equation}\label{eq:detV}
L^3=\Lambda^3 V=\underline{\C}
\end {equation}
is trivial. We conclude that a Lagrangian minimal immersion $f\colon \Sigma\to \CP^2$ yields a 3-torsion point in the Jacobian of $\Sigma$. 
The special Legendrian lift to $\S^5$ is then given by a parallel section of unit length with specific phase, unique up to multiplication with $\exp(\pm\tfrac{2\pi i}{3})$.
Geometrically this means (as in Proposition \ref{pro:upanddown}) that in general one has to pass to a threefold cover $\hat{\Sigma}\to \Sigma$ in order for the Legendrian lift $\hat{f}\colon \hat{\Sigma}\to \S^5\subset \C^3$ to close, since $\hat{f}$ is given as a unit length parallel section of  $L\subset\underline{\C}^3$ with its induced Hermitian structure.

Lastly, the induced area of a minimal Lagrangian immersion $f\colon \Sigma\to\CP^2$ can be calculated in terms of the Higgs field $\Phi$. Since $\tr \Phi\wedge \Phi^{\dagger}=\tfrac{1}{2}\tr[ \Phi\wedge \Phi^{\dagger}]$, equations \eqref{eq:diagSD1} and \eqref{eq:diagSD} combine to
\begin{equation}\label{eq:areaformML}
3|df|^2=i\,\tr\, \Phi\wedge \Phi^{\dagger}-iF^{\nabla^K}.
\end{equation}
The line bundle decomposition \eqref{eq:C3decom} defines a lift $\psi\colon \Sigma\to \SU_3/T$ into the full flag variety (which happens to be a primitive harmonic map, see \cite{Burstall}). The energy density of $\psi$ calculates to
\begin{equation}\label{energydens}
\tfrac{1}{2}\langle d\psi\wedge * d\psi\rangle=i\,\tr\, \Phi\wedge \Phi^{\dagger}
\end{equation}
where $\langle\,\,,\,\rangle$ is the invariant Riemannian metric on $\SU_3/T$ coming from  the trace form on $\mathfrak{su}_3$.  If $\Sigma$ is compact, we thus obtain 
\begin{equation}\label{eq:areaML}
\mathrm{Area}(f)=\tfrac{2\pi}{3} (\tfrac{1}{2\pi}E(\psi)+ (2-2g))
\end{equation}
for the induced area of a minimal Lagrangian immersion  $f\colon \Sigma\to \CP^2$.

\subsection{The fundamental theorem for minimal Lagrangian surfaces}
Our construction of higher genus minimal Lagrangian surfaces is based on solving the signed self-duality equations \eqref{eq:signedSD}. 
This requires one to reconstruct a minimal Lagrangian surface $f\colon \Sigma\to\CP^2$ from the
Gauss-Codazzi  
data $(h,Q)$, where $h\in\Gamma(\Sigma,\bar KK)$ is a conformal metric, and $Q$ is a holomorphic cubic differential.

We begin by choosing a 3-torsion point  $L\in \Jac(\Sigma)$ in the Jacobian of $\Sigma$ and consider the rank three holomorphic  bundle
$
V=KL\oplus L\oplus K^{-1}L
$
which has trivial determinant line bundle $\Lambda^3 V=L^3=\underline{\C}$ over $\Sigma$.  The holomorphic line bundle $L$ has a unique (up to constant scaling) Hermitian metric $h_L$ with  flat Chern connection $\nabla$. We equip the line bundles $K$, $K^{-1}$ with the Hermitian structures given by the conformal metric $h$ so that their Hermitian connections $\nabla^{K^{\pm 1}}$ arise from the Levi-Civita connection of $h$. Thus, $V$ obtains the structure of a Hermitian rank three bundle with diagonal connection $D=(\nabla^K\otimes \nabla)\,\oplus \nabla\oplus \, (\nabla^{K^{-1}}\otimes \nabla)$ compatible with its holomorphic structure, in other words, $D$ is the Chern connection for the diagonal Hermitian metric $h^{-1}\oplus h_L\oplus h$ on $V$.  
Furthermore, consider the holomorphic cyclic Higgs field $\Phi$ as in \eqref{eq:Higgsfield}, which is uniquely determined by $Q$,
and  its adjoint $\Phi^\dagger$ with respect to the diagonal Hermitian metric. It is given by \eqref{eq:Higgsadj}
Then the Hermitian connection
\(
d_V=D+\Phi-\Phi^{\dagger}
\)
is flat if and only if the pair $(D,\Phi)$ solves the signed self-duality equations \eqref{eq:signedSD}.

 Let 
\[
\Upsilon\colon (V,d_V)\to (\underline{\C}^3, d)
\]
be the trivialization of the flat Hermitian connection $d_V$ on the universal cover $\tilde{\Sigma}\to \Sigma$. In other words, $\Upsilon$ satisfies $d. \Upsilon=d_V$  on $\tilde \Sigma$, and $\Upsilon$ is equivariant 
\begin{equation}\label{eq:equivara}
\gamma^* \Upsilon=\rho_{\gamma}^{-1} \Upsilon
\end{equation}
 with respect to the right action of deck transformations $\gamma\in\pi_1(\Sigma,p_0)$ on $\tilde{\Sigma}$  and the monodromy representation $\rho\colon \pi_1(\Sigma,p_0)\to \SU_3$ of $d_V$. 
 Then the image under $\Upsilon$ of the pullback $\tilde{L}$ of $L\subset V$ to $\tilde{\Sigma}$ defines a minimal Lagrangian immersion 
\begin{equation}\label{eq:f-recon}
f\colon \tilde{\Sigma}\to\CP^2\,,\quad f(p)=\Upsilon_p(\tilde{L}_p)\subset \C^3
\end{equation}
with cubic holomorphic differential $Q\in H^0(K^3)$. The equivariance of $\Upsilon$ implies that $f$ has the same equivariance property
\begin{equation}\label{eq:equiv}
\gamma^* f=\rho_{\gamma}^{-1} f
\end{equation}
with respect to the monodromy representation $\rho\colon \pi_1(\Sigma)\to \SU_3$ of $d_V$.

The immersion $f\colon \tilde{\Sigma}\to\CP^2$ descends to $\Sigma$ if and only if the monodromy representation $\rho$ of the flat connection $d_V$ is projectively trivial, that is $\rho=c\, \Id$ is a multiple of the trivial representation. The Abelian representation $c\colon \pi_1(\Sigma) \to \S^1$ has to satisfy $c^3=1$ due to $L^3=\underline{\C}$. We will utilize this projective freedom and assume for the remainder of the paper that $L=\underline{\C}$ is the trivial bundle. We summarize the discussion so far in the form of a ``fundamental theorem for minimal Lagrangian surfaces''.
\begin{The}\label{thm:reconstruct1}
Let $\Sigma$ be a Riemann surface with a conformal metric $h$,
and let $Q\in\Gamma(K^3)$ be a 
cubic differential. Consider the  holomorphic rank three bundle 
$V=K\oplus\underline{\C}\oplus K^{-1}$
over $\Sigma$ equipped with the diagonal Hermitian metric $h^{-1}\oplus h_{\C}\oplus h$, where $h_{\C}$ denotes the standard Hermitian structure on $\C$, and diagonal Hermitian connection $D=\nabla^{K}\oplus d_{\C}\oplus\nabla^{K^{-1}}$, which is the Chern connection for the holomorphic and Hermitian structures on $V$. Let $\Phi\in\Gamma(K{\mathfrak{sl}}(V))$ and its adjoint $\Phi^{\dagger}\in \Gamma(\bar{K}{\mathfrak{sl}}(V))$ be defined as in \eqref{eq:Higgsfield} and \eqref{eq:Higgsadj}. Then we have the following equivalent statements:
\begin{enumerate}[(i)]
\item
$f\colon \tilde{\Sigma}\to \CP^2$ is a minimal Lagrangian immersion with induced metric $\pi^*h$ and holomorphic cubic differential $\pi^*Q\in H^0(K^3)$ on the universal covering $\pi\colon \tilde{\Sigma}\to \Sigma$ equivariant 
with respect to a representation $\rho\colon \pi_1(\Sigma)\to \SU_{3}$, that is, $\gamma^* f=\rho_{\gamma}^{-1}f$.
\item
The Hermitian connection $d_V=D+\Phi-\Phi^{\dagger}$ on $V\to \Sigma$ satisfies the signed self-duality equations \eqref{eq:signedSD}---in particular is flat---with monodromy representation $\rho\colon \pi_1(\Sigma)\to \SU_3$.
\end{enumerate}
The minimal Lagrangian immersion {\em satisfies the extrinsic closing condition}, in other words, descends to a minimal Lagrangian immersion $f\colon \Sigma\to \CP^2$ if and only if $\rho=c\, \Id$ with the Abelian representation $c \colon \pi_1(\Sigma)\to \S^1$  satisfying $c^3=1$. 
Furthermore, $c$ is trivial, equivalently $c\equiv 1$, if and only if the minimal Lagrangian immersion $f\colon \Sigma\to \CP^2$ lifts to a minimal Legendrian immersion $\hat{f}\colon \Sigma\to \S^5$.
 \end{The}

\subsection{The associated family of flat connections}\label{sec:ass-family}
Our approach for solving the signed self-duality equations \eqref{eq:signedSD} is based on the well-known observation that 
the signed self-duality equations are equivalent to the flatness of a $\C^{\times}$-family of connections. These connections enjoy certain symmetry properties which in turn characterize them. 

As in Theorem~\ref{thm:reconstruct1}, we work on the bundle $V=K\oplus\underline{\C}\oplus K^{-1}$ over a Riemann surface $\Sigma$ with conformal metric $h$ and cubic differential $Q\in\Gamma(K^3)$. These data equip $V$ with a Hermitian bundle metric, the Chern connection $D$, and the Higgs field $\Phi\in\Gamma(K{ \mathfrak{sl}}(V))$ together with its adjoint $\Phi^{\dagger}\in \Gamma(\bar{K}{\mathfrak{sl}}(V))$ as given in  
\eqref{eq:Higgsfield} and \eqref{eq:Higgsadj}. 
\begin{Lem}\label{lem:loop}
The 
 $\SU_3$-connection $d=D+\Phi-\Phi^{\dagger}$ on $V$ solves the signed self-duality equations \eqref{eq:signedSD} if and only if the $\C^{\times}$-family of ${SL}_3(\C)$-connections
\begin{equation}\label{eq:loop}
d_{\lambda}=D+\lambda^{-1}\Phi-\lambda\Phi^{\dagger}
\end{equation}
on $V$ is flat for all $\lambda\in \C^{\times}$.

The family of flat connections  $d_{\lambda}$ has the following properties:
\begin{enumerate}[(i)]
\item
$d_{1/\bar{\lambda}}=d_{\lambda}^{\dagger}$, in other words, $d_{\lambda}$ is an $\SU_3$-connection for unimodular $\lambda\in \S^1$.
\item
$d_{\lambda}$ has a simple pole at $\lambda=0$ with residue $\Res_{0} d_{\lambda}=\Phi\in\Gamma(K{\mathfrak{sl}}(V))$. By (i) this is equivalent to $d_{\lambda}$ having a simple pole at $\lambda=\infty$ with $\Res_{\infty} d_{\lambda}=-\Phi^{\dagger}\in\Gamma(\bar{K}{\mathfrak{sl}}(V))$.
\item
$\Phi^2\in\Gamma(K^2{\mathfrak{sl}}(V))$ is nowhere vanishing. 
\item
$d_{\zeta\lambda}=d_{\lambda}. \diag(\zeta^2,1,\zeta)$ for the cube root of unity $\zeta=\exp(2\pi i/3)$. 
\item
$d_{-\lambda}=d_{\lambda}^{*}. b$ for the orthogonal structure  $b\colon V\to V^*$, $b^{*}=b$  given by
$b=\begin{psmallmatrix}0&0&{1}\\0&{1}&0\\{1}&0&0\end{psmallmatrix}$.
\end{enumerate}
For unimodular $\lambda\in \S^1$ the flat connection $d_{\lambda}$ defines a minimal Lagrangian immersion $f_{\lambda}\colon \tilde{\Sigma}\to \CP^2$ on the universal cover $\pi\colon \tilde{\Sigma}\to \Sigma$ equivariant \eqref{eq:equiv} with respect to the monodromy representation $\rho_{\lambda}\colon \pi_1(\Sigma)\to \SU_3$ of $d_{\lambda}$. The immersion $f_{\lambda}$ induces the metric $\pi^*h$  and holomorphic cubic differential $\lambda^{-3} \pi^*Q$. In particular, $\rho_{\lambda_0}=c\, \Id$ for $\lambda_0\in \S^1$ is projectively trivial  with $c\colon \pi_1(\Sigma)\to \S^1$ satisfying $c^3=1$, if and only if the corresponding minimal Lagrangian immersion $f_{\lambda_0}\colon \Sigma\to \CP^2$ is defined on $\Sigma$ (rather than just on the universal cover $\tilde{\Sigma}$). Furthermore, $\rho_{\lambda_0}=\Id$ if and only if  the minimal Lagrangian immersion $f_{\lambda_0}\colon \Sigma\to \CP^2$ lifts to  a minimal Legendrian immersion $\hat{f}_{\lambda_0}\colon \Sigma\to \S^5$.
\end{Lem}
\begin{proof}
The first assertion  follows by expanding the flatness condition of $d_{\lambda}$ in powers of $\lambda$. Properties (i)-(iii) are seen by inspection and (iv) follows from a direct calculation. For property (v) note that $V^*=K^{-1}\oplus \underline{\C}\oplus K$ so that the orthogonal structure $b$ is well-defined with ${\bf 1}$ denoting the identity maps on the respective line bundles. But  $d_{\lambda}^*=D^*-\lambda^{-1}\Phi^*+ \lambda (\Phi^{\dagger})^*$ with 
$\Phi^*=\begin{psmallmatrix} 0 &{\bf 1}&0\\ 0&0&{\bf 1} \\ Q&0& 0\end{psmallmatrix}\in \Gamma(K{\mathfrak{sl}}(V^*)) 
$ and likewise for $(\Phi^{\dagger})^*$, whence we have (v). 

The remaining statements 
follow from the discussion leading up to Theorem~\ref{thm:reconstruct1}: the flat connection $d_{\lambda}$, for unimodular $\lambda\in \S^1$, can be trivialized by an isometric bundle isomorphism 
\begin{equation}\label{Upsiloneq}
\Upsilon_{\lambda}\colon (V,d_{\lambda})\to(\underline{\C}^3, d)
\end{equation}
on  the universal cover $\tilde{\Sigma}$ giving rise, via \eqref{eq:f-recon}, to the immersion
\begin{equation}\label{lambda-recon}
f_{\lambda}=\Upsilon_{\lambda}(\underline{\C})\colon\tilde{\Sigma}\to\CP^2\,
\end{equation}
equivariant with respect to the monodromy representation $\rho_{\lambda}\colon \pi_1(\Sigma)\to \SU_3$ of $d_{\lambda}$ on $V$.
But $d_{\lambda} =\tilde{d}_{\lambda}. \diag(\lambda^{-1},1,\lambda)$ is gauge equivalent to  
\[
\tilde{d}_{\lambda}=D+\tilde{\Phi}-\tilde{\Phi}^{\dagger}\,,\quad 
\tilde{\Phi}=\begin{psmallmatrix} 0&0&\lambda^{-3} Q\\{\bf 1}&0&0\\0&{\bf 1}&0\end{psmallmatrix}
\]
via the $\SU_3$-gauge $\diag(\lambda^{-1},1,\lambda)$ which leaves  $f_{\lambda}= \Upsilon_{\lambda}(\underline{\C})$ unchanged.
Theorem~\ref{thm:reconstruct1} then  implies that  $f_{\lambda}$ is a minimal Lagrangian immersion with induced metric $\pi^*h$ and cubic differential $\lambda^{-3} \pi^*Q$. The condition for $f_{\lambda_0}$ to be defined on $\Sigma$ (rather than on $\tilde{\Sigma}$), namely $\rho_{\lambda_0}=c\,\Id$ with 
$c\colon \pi_1(\Sigma)\to \S^1$ satisfying $c^3=1$, follows from the last part of Theorem~\ref{thm:reconstruct1} as does the assertion about the minimal Legendrian lift into $\S^5$ in the case $c=1$. 
\end{proof}
The next lemma shows that any $\C^{\times}$-family of flat connections $d_{\lambda}$ satisfying  properties (i)-(v) on a  Hermitian rank three bundle $V\to \Sigma$ with trivial determinant $\Lambda^3 V=\underline{\C}$ gives rise to a minimal Lagrangian immersion. To formulate the symmetries (iv), (v) of Lemma~\ref{lem:loop}, we assume the existence of bundle isomorphisms  $g\colon \Sigma\to\SU(V)$, $g\neq \mu \Id\,$ and $b\colon V\to V^*$, $b=b^*$, $\det b=1$ satisfying the compatibility $g^{*}bg=b$. 

Then $V$ splits orthogonally into the eigenline bundles
$
V=E_{-1}\oplus E_0\oplus E_1
$
of $g$, and with respect to this splitting one has
$
g=\diag(\zeta^{-1},1,\zeta),
$
where $\zeta=\exp(2\pi i/3)$.
Using the compatibility condition, $b=\begin{psmallmatrix}0&0&a\\0&{a_0}&0\\a^*&0&0\end{psmallmatrix}$ and thus $E_0=E_0^{*}$ via $a_0$ and $E_{-1}=E_{1}^*$ via $a$. 
Since $E_0\simeq E_0^*$ and
$
\Lambda^3V=E_{-1}\otimes E_0\otimes E_1\simeq E_0
$
is trivial, we conclude that $E_0=\underline{\C}$.
Thus $b=\begin{psmallmatrix}0&0&{1}\\0&{1}&0\\{1}&0&0\end{psmallmatrix}$ on $V=E\oplus \C\oplus E^{-1}$.
\begin{Lem}\label{lem:loop2mLag}
Let $V\to \Sigma$ be a Hermitian rank three bundle with trivial determinant bundle $\Lambda^3 V=\underline{\C}$ and an orthogonal splitting into line bundles $V=E\oplus \underline{\C} \oplus E^{-1}$. Let $d_{\lambda}$ be a $\C^{\times}$-family of flat $\SL(V)$-connections satisfying the properties (i)-(v) of Lemma~\ref{lem:loop}. 
Then $E\cong K$ holomorphically, so that
$
V=K\oplus \underline{\C}\oplus K^{-1},
$
and
\[
d_{\lambda}=D+\lambda^{-1}\Phi-\lambda\Phi^{\dagger},
\]
where $D$ is a diagonal Hermitian connection and $\Phi\in\Gamma(K\mathfrak{sl}(V))$ has the form \eqref{eq:Higgsfield}. Moreover, $d_{\lambda}$ defines an $\S^1$-family of minimal Lagrangian immersions on the universal cover $\pi\colon\tilde\Sigma\to\Sigma$
with induced metric $\pi^*h$ and holomorphic cubic differential $\lambda^{-3}\pi^*Q$, where $h$ is the induced metric on $K^{-1}$ and $Q=\det\Phi\in H^0(K^3)$.
The immersion $f_{\lambda}$ is equivariant \eqref{eq:equiv} with respect to the monodromy representation $\rho_{\lambda}$ of $d_{\lambda}$. If $\rho_{\lambda_0}=\Id$ for some $\lambda_0\in \S^1$ then $f_{\lambda_0}\colon \Sigma\to\CP^2$ and its Legendrian lift $\hat{f}_{\lambda_0}\colon \Sigma\to \S^5$ are both defined on $\Sigma$.  
\end{Lem}
\begin{proof}
From property (i) of Lemma~\ref{lem:loop}, we know that $d_{\lambda}$ are $\SU(V)$-connections for $\lambda\in \S^1$. Let  $D$  be the diagonal Hermitian  connection obtained from $d_{1}$ via orthogonal projection onto the line bundles 
$V=E\oplus \underline{\C}\oplus E^{-1}$. Then property (ii) gives
\[
d_{\lambda}=D+\lambda^{-1}\Phi-\lambda\Phi^{\dagger}
\]
with $\Phi\in\Gamma(K\mathfrak{sl}(V))$.  
The gauge $g=\diag(\zeta^2,1,\zeta)$ is $D$-parallel and property (iv) yields
$
\Ad g\,  \Phi =\zeta \Phi
$
which implies
\[
\Phi=\begin{psmallmatrix} 0&0&\gamma\\\alpha&0&0\\0&\beta&0\end{psmallmatrix}.
\]
Property (v) gives us $\alpha=\beta\in\Gamma(KE^{-1})$ and thus $\det\Phi=\alpha^2\gamma$. 
The flatness of the family $d_{\lambda}$ implies that $\Phi$ is $\dbar^D$-holomorphic, hence
$
\Phi\in H^0(K\mathfrak{sl}(V)).
$
 Condition (iii) then shows that $\alpha\in H^0(KE^{-1})$ is nowhere vanishing,  thus $E=K$ holomorphically, and $\det\Phi\in H^0(K^3)$ is a holomorphic cubic differential. 

The $\S^1$-family of minimal Lagrangian immersions $f_{\lambda}$ can now be constructed as in the proof of Lemma~\ref{lem:loop}.
\end{proof}
\begin{Rem}\label{rem:mLag=loop}
Lemmas~\ref{lem:loop} and \ref{lem:loop2mLag} show that the construction of a minimal Lagrangian immersion $f\colon \Sigma\to \CP^2$ is 
equivalent to the construction of a $\C^{\times}$-family of flat connections $d_{\lambda}$ on the Hermitian rank three bundle $V=K\oplus\C\oplus K^{-1}$ satisfying the  properties (i)-(v).
\end{Rem}

The geometric origin of the family of flat connections $d_{\lambda}$ stems from the fact that the Tzitz{\'e}ica equation \eqref{eq:Tzitz} is invariant under a scaling of the cubic differential $Q$ by a unimodular factor. As we showed in the previous Lemma, scaling $Q$ by $\lambda^{-3}$ is equivalent to scaling the Higgs field $\Phi$ by $\lambda$. The resulting $\S^1$-family of flat connections $d_{\lambda}$---or  the $\rho_{\lambda}$-equivariant minimal Lagrangian immersions $f_{\lambda}$---is traditionally called the {\em associated family}. A point $\lambda_0\in \S^1$ at which $\rho_{\lambda_0}=c\, \Id$ with $c^3=1$ is called a {\em Sym point}. 
The condition $\rho_{\lambda_0}=c\, \Id$ with $c^3=1$ is, in our context,  frequently referred to as {\em extrinsic closing}, since it implies that the minimal Lagrangian immersion $f_{\lambda_0}$ is defined on $\Sigma$, that is, the  ``periods of $f_{\lambda_0}$ in $\CP^2$ {\em close}''.  If $\rho_{\lambda_0}= \Id$ is trivial, also the periods of the Legendrian lift $\hat{f}_{\lambda_0}$ in $\S^5$ close. 

Finally, the $\Z_3$-symmetry in Lemma~\ref{lem:loop} (iv) is already present \cite{Burstall}  in the case of minimal immersions $f\colon \Sigma\to\CP^2$ reflecting the periodicity of its harmonic sequence \eqref{eq:harmseq}. The additional $\Z_2$-symmetry in Lemma~\ref{lem:loop} (v) comes from the Lagrangian condition. Since the two symmetries commute, they combine to the $\Z_6$-symmetry of the associated family of flat connections
\begin{equation}\label{eq:order6}
d_{\xi\lambda}=d_{\lambda}^*. a\,, \qquad a=\begin{psmallmatrix}0&0&\zeta^2\\0& {1}&0\\\zeta&0&0\end{psmallmatrix}\colon V\to V^*\,,\qquad \xi=\exp(\pi i/3)\,,\quad \xi^2=\zeta
\end{equation}
for a minimal Lagrangian immersion $f\colon \Sigma\to\CP^2$ \cite{McI2, DM}.
These symmetries have been used extensively (see for example \cite{MM, Ha,CMcI}) in the construction of minimal (Lagrangian) tori in $\CP^2$ for which $Q=dz^3$, 
since they allow one to lift into associated twistor spaces and apply the integrable-systems machinery of \cite{BFPP}.

\section{Loop groups and the global DPW construction}\label{sec:loop-groups}
In Section~\ref{sec:ass-family} we saw that constructing minimal Lagrangian immersions into $\CP^2$ is equivalent to constructing flat $\C^\times$-families of connections of the form \eqref{eq:loop} satisfying the reality, symmetry, and asymptotic properties described in Lemma~\ref{lem:loop}; equivalently, one seeks families carrying the $\Z_6$-symmetry \eqref{eq:order6} together with the prescribed behaviour at $\lambda=0$ and $\lambda=\infty$. The local DPW construction \cite{DPW} shows that, on simply connected domains, such families can be generated from meromorphic $\lambda$-connections $D_\lambda$. We next adapt this method to the global setting.

A natural framework for treating $\lambda$-families of connections is that of loop groups: one packages such a family as a single connection with values in a loop algebra, or equivalently as a connection with loop-group structure group. 

\subsection{Loop groups}
Let $W$ be a finite-dimensional Hermitian vector space, $\tau\colon \mathrm{SL}(W)\to \mathrm{SL}(W)$ an order $r$ Lie group automorphism preserving the Hermitian structure,  and $\xi=\exp(2\pi i/r)$ a primitive $r$-th root of unity.  The ($\tau$-twisted) loop group  
\begin{equation}\label{eq:loopgroup}
{\bf\SL}^{\tau}(\S^1)=\{g\colon \S^1 \to \mathrm{SL}(W) \,;\, \text{$g$ real analytic},\,\tau g (\lambda)=g(\xi \lambda)\}
\end{equation}
is a smooth Banach Lie group via the norm
\(
||g||=\sum_{k\in\Z} |g^{(k)}|\,.
\)
Here $g=\sum_{k\in\Z} g^{(k)}\lambda^k$ is the Fourier decomposition of $g$ and $|g^{(k)}|$ is a matrix norm for $g^{(k)}\in\mathfrak{gl}(W)$. The loop group ${\bf\SL}^{\tau}(\S^1)$ has several Banach Lie subgroups relevant to our constructions: the group of {\em unitary loops}
\begin{equation}\label{eq:realloop}
{\bf\SU}^{\tau}(\S^1)=\{g\in {\bf\SL}^{\tau}(\S^1)\,;\, g(\lambda)^{-1}=g(1/{\bar{\lambda}})^{\dagger}\}, 
\end{equation}
where  $g(\lambda)\in \SU(W)$, 
the group of {\em positive loops}
\begin{equation}\label{eq:posloops}
{\bf\SL}^{\tau}(\mathbb{D})=\{g\in {\bf\SL}^{\tau}(\S^1)\,;\, \text{$g$ extends holomorphically to $\mathbb{D}$}\}
\end{equation}
where  $g$ extends holomorphically into the unit disk $\mathbb{D}=\{\lambda\in\C\,;\, |\lambda|<1\}$, 
and the group 
\[{\bf\SL}^{\tau}(\mathbb{D}^{\times})=\{g\in {\bf\SL}^{\tau}(\S^1)\,;\, \text{$g$ extends holomorphically to $\mathbb{D}^\times$}\}\] of loops extending holomorphically into the punctured unit disk $\mathbb{D}^{\times}$. For the trivial automorphism $\tau=\Id$, we write  ${\bf\SL}^{\tau}(\S^1)={\bf\SL}(\S^1)$ and correspondingly so for the various subgroups. 

 The basic ingredient of the DPW method is the loop Iwasawa decomposition \cite{PS,DPW}. To guarantee the uniqueness of this decomposition, we choose a solvable subgroup $B\subset \mathrm{SL}(W)^{\tau}$ in the fixed point subgroup under $\tau$ such that $\mathrm{SL}(W)^{\tau}=\SU(W)^{\tau}B$ and $B\cap \SU(W)^{\tau}=\{\Id\}$. Define
 \begin{equation}\label{eq:basedloops}
 {\bf\SL}^{\tau}_B(\mathbb{D}) :=\{ g\in {\bf\SL}^{\tau}(\mathbb{D})\,;\, g(0)\in B\}\subset {\bf\SL}^{\tau}(\mathbb{D}) 
 \end{equation}
the Banach Lie subgroup of loops $g$ extending holomorphically to the disk $\mathbb{D}$ with $g(0)\in B$. 
\begin{The}\label{thm:Iwasawa}
The multiplication maps
\[
{\bf\SL}^{\tau}_B(\mathbb{D})\times {\bf\SU}^{\tau}(\S^1)\to {\bf\SL}^{\tau}(\S^1)\quad\text{and}\quad
{\bf\SU}^{\tau}(\S^1)\times {\bf\SL}^{\tau}_B(\mathbb{D})\to {\bf\SL}^{\tau}(\S^1)
\]
are smooth diffeomorphisms. We can thus decompose every loop $g\in {\bf\SL}^{\tau}(\S^1)$ uniquely as
\[
g=g_{+} g_{u}
\]
with $g_{+}\in {\bf\SL}^{\tau}_B(\mathbb{D})$ a positive loop with $g(0)\in B$ and $g_{u} \in {\bf\SU}^{\tau}(\S^1)$ a unitary loop, and analogously for the reversed order.
\end{The}
Replacing $\mathrm{SL}(W)$ by its Lie algebra $\mathfrak{sl}(W)$, we obtain the loop Lie algebras of the various loop groups under consideration. For example, 
\[
{\bf sl}^{\tau}(\S^1)=\{x\colon \S^1\to \mathfrak{sl}(W)\,;\, \text{$x$ real analytic}, \, \tau x(\lambda)=x(\xi\lambda)\}
\]
is the Banach loop Lie algebra of ${\bf\SL}^{\tau}(\S^1)$. Let $\mathfrak{g}_j\subset\mathfrak{sl}(W)$ denote  the $\xi^j$-eigenspace of $\tau$ with commutator relations $[\mathfrak{g}_j,\mathfrak{g}_l]\subset \mathfrak{g}_{j+l}$ where  $j,l\in\Z_r$. The condition $ \tau x(\lambda)=x(\xi\lambda)$ then translates to  $x^{(k)}\in \mathfrak{g}_k$ where $k\in\Z_r$ and $x=\sum_{k\in\Z} x^{(k)} \lambda^k$.  Furthermore, the loop $x\in {\bf su}^{\tau}(\S^1)$ is unitary if and only if  
$x^{(-k)}=-(x^{(k)})^{\dagger}\in \mathfrak{g}_{-k}$.

\begin{Exa}\label{ex:order6}
To define the twisted loop groups relevant for minimal Lagrangian surfaces, consider
  $W=\C^3$ and identify $W\cong W^*$ via the standard complex bilinear inner product so that the dual of  $A\in \mathfrak{sl}_3(\C)$ is given by $A^*=-A^t$. 
The order 6 automorphism corresponding to the $\Z_6$-symmetry \eqref{eq:order6} is given by
 $ \tau(g)=\mathfrak a^{-1} (g^{-1})^t \mathfrak a
 $
 with
 $
 \mathfrak a=\begin{psmallmatrix}0&0&\zeta^2\\0& {1}&0\\\zeta&0&0\end{psmallmatrix}\,,\qquad  \zeta=\exp(2\pi i/3)\,.
$
The corresponding Lie algebra automorphism $\tau\in \mathrm{Aut}(\mathfrak{sl}_3(\C))$ is given by 
\(
\tau(A)= -\mathfrak a^{-1}A^t \mathfrak a\,.
\)

For computational purposes in Section \ref{sec:em}, we rearrange the line bundle decomposition \eqref{eq:cyclic-bundle} to
\begin{equation}\label{eq:decoagain}
L\oplus KL\oplus K^{-1}L,
\end{equation}
which amounts to conjugating $\tau$ by the inner automorphism  $\varphi(g)=c^{-1}gc$ with 
$c=\begin{psmallmatrix}0&i&0\\i& 0&0\\0&0&-1\end{psmallmatrix}$. The factors of $i$ are chosen so that the natural Legendrian lift $\hat{f}\colon \Sigma\to \S^5$ of a minimal Lagrangian immersion $f\colon \Sigma\to \CP^2$ is {\em special} Legendrian (rather than just minimal Legendrian). The resulting order $r=6$ automorphism $\hat\tau=\varphi^{-1}\circ \tau\circ \varphi$  is then given by
\begin{equation}\label{def.mTta}
\hat\tau(g)=\mathfrak{\hat a}^{-1}(g^{-1})^t\mathfrak{\hat a}\,,\qquad \mathfrak{\hat a}=
\begin{psmallmatrix}-1&0&0\\0& 0&i\zeta^2\\0&i\zeta&0\end{psmallmatrix}.
\end{equation}
The $\xi^k=\exp(2\pi i k/6)$ eigenspaces ${\mathfrak{\hat g}}_k$  under the action of ${\hat \tau}$ on the Lie algebra $\mathfrak{sl}_3( \C)$ are 
\begin{align}\label{eq:Tau-es}
{\mathfrak{\hat g}}_{-1}&=\{ \begin{psmallmatrix}0&a&0\\0& 0&b\\ia&0&0\end{psmallmatrix}\,;\, a,b\in\C\},
\quad {\mathfrak{\hat g}}_{0}=\{ \begin{psmallmatrix}0&0&0\\0& a&0\\0&0&-a\end{psmallmatrix}\,;\, a\in\C\},
\quad  {\mathfrak{\hat g}}_{1}=\{ \begin{psmallmatrix}0&0&ia\\-a& 0&0\\0&b&0\end{psmallmatrix}\,;\, a,b\in\C\}\\
{\mathfrak{\hat g}}_2&=\{ \begin{psmallmatrix}0&a&0\\0& 0&0\\-i a&0&0\end{psmallmatrix}\,;\, a\in\C\},
\quad {\mathfrak{\hat g} }_{3}=\{ \begin{psmallmatrix}-2a&0&0\\0& a&0\\0&0&a\end{psmallmatrix}\,;\, a\in\C\},
\quad  {\mathfrak{\hat g}}_{4}=\{ \begin{psmallmatrix}0&0&a\\-ia& 0&0\\0&0&0\end{psmallmatrix}\,;\, a\in\C\}.\nonumber
\end{align}
\end{Exa}

\subsection{Loop connections}

\begin{Def}\label{def:loopconnection}
Let $V=N\times W$ be the trivial Hermitian bundle over a Riemann surface $N$. An {\em ${\bf\SL}^{\tau}(\S^1)$-connection} is given by
\begin{equation}\label{eq:loopconnection}
D_{\lambda}=d+\eta_{\lambda}
\end{equation}
with $d$ the trivial connection on $V$  and $\eta_{\lambda}=\sum_{k\in\Z}\eta^{(k)}\lambda^k\in\Omega^1(N,{\bf sl}^{\tau}(\S^1))$, that is, $\eta_{\xi\lambda}=\tau\eta_{\lambda}$ equivalently $\eta^{(k)}\in\Omega^1(N,\mathfrak{g}_k)$.  Note that the connection $D=d+\eta^{(0)}$ is an $\mathrm{SL}(W)^{\tau}$-connection.

A smooth map $g_{\lambda}\colon N\to {\bf\SL}^{\tau}(\S^1)$ is called an {\em ${\bf\SL}^{\tau}(\S^1)$-gauge.} 

The corresponding notations are used for the various subgroups of ${\bf\SL}^{\tau}(\S^1)$ introduced above. In particular,  an ${\bf\SU}^{\tau}(\S^1)$-connection $D_{\lambda}$ has $D=d+\eta^{(0)}$ as an $\SU(W)^{\tau}$-connection and $\eta^{(-k)}=-(\eta^{(k)})^{\dagger}\in \Omega^1(N,\mathfrak{g}_{-k})$ for all $k\neq0$. We often refer to $\eta_{\lambda}$ as the {\em potential} of the connection $D_{\lambda}$.

An ${\bf\SL}^{\tau}(\S^1)$-connection $D_{\lambda}=d+\eta_{\lambda}$ is called a ($\tau$-twisted) {\em DPW-connection} if $D_{\lambda}$ is a flat ${\bf\SL}^{\tau}(\mathbb{D}^{\times})$-connection with a simple pole at $\lambda=0$. In other words, the potential $\eta_{\lambda}=\sum_{k\geq -1}\eta^{(k)}\lambda^k\in\Omega^1(N,{\bf sl}^{\tau}(\mathbb{D}^{\times}))$  with residue $\eta^{(-1)}\in \Omega^{(1,0)}(\mathfrak{g}_{-1})$ and $d\eta_{\lambda}+\eta_{\lambda}\wedge \eta_{\lambda}=0$.
\end{Def}
\begin{Rem}\label{rem:tauparallel}
Since $\tau$ is parallel with respect to the trivial connection $d$, an ${\bf\SL}^{\tau}(\S^1)$-connection $D_{\lambda}=d+\eta_{\lambda}$ may also be characterized by the relation
$
D_{\xi\lambda}=\tau D_{\lambda}\,.
$
where $\tau D_{\lambda}: =d+\tau(\eta_{\lambda})$.
\end{Rem}
\begin{Rem}\label{rem:DPW-connection}
DPW-connections $D_{\lambda}=d+\sum_{k\geq -1}\eta^{(k)}\lambda^k$ are preserved under ${\bf\SL}^{\tau}(\mathbb{D})$-gauges $b_{\lambda}$:
\[
D_{\lambda}. b_{\lambda}=d+b_{\lambda}^{-1} d\,b_{\lambda}+\Ad b^{-1}_{\lambda}\eta_{\lambda}=d+\sum_{k\geq -1}\tilde{\eta}^{(k)}\lambda^k
\]
with $\tilde{\eta}^{(-1)}=\Ad (b^{(0)})^{-1} \eta^{(-1)}$.
In particular, a DPW connection $D_{\lambda}$ is an ${\bf\SU}^{\tau}(\S^1)$-connection if and only if 
\[
D_{\lambda}=D+\lambda^{-1}\eta^{(-1)}-\lambda (\eta^{(-1)})^{\dagger}
\]
where $D=d+\eta^{(0)}$ is an $\SU(W)^{\tau}$-connection. 
\end{Rem}

Since the families of flat connections \eqref{eq:loop} that we seek are unitary along $\S^1$, we first need to understand when an ${\bf SL}^{\tau}(\S^1)$-connection can be unitarized, that is, when is there an ${\bf\SL}^{\tau}(\S^1)$-gauge $g_{\lambda}$ such that $D_{\lambda}. g_{\lambda}$ is an ${\bf\SU}^{\tau}(\S^1)$-connection.

\begin{Lem}\label{lem:tau-unitarize}
Let $D_{\lambda}$ be a flat ${\bf\SL}^{\tau}(\S^1)$-connection on $V=N\times W$ and assume that its monodromy representation $\rho_{\lambda}\colon \pi_1(N,q_0)\to \mathrm{SL}(W)$ is irreducible and unitarizable for every $\lambda\in \S^1$ where $q_0\in N$ is some base point. 
Additionally, assume that there exists a subgroup $\hat B\subset\SL(W)$ satisfying
 $\tau(\hat B)\subset\hat B,$ 
$
\mathrm{SL}(W)=\SU(W)\hat B$ and $\SU(W)\cap \hat B=\{\Id\}.
$
Then there exists an ${\bf\SL}^{\tau}(\mathbb{D})$-gauge $b_{\lambda}$ such that $D_{\lambda}. b_{\lambda}$ is a flat  ${\bf\SU}^{\tau}(\S^1)$-connection.

In particular, if $D_{\lambda}=d+\sum_{k\geq -1}\eta^{(k)}\lambda^k$ is a DPW-connection, then 
\[
D_{\lambda}. b_{\lambda}=D+\lambda^{-1} \Psi-\lambda \Psi^{\dagger}
\] 
with $D$ an $\SU(W)^{\tau}$-connection and $\Psi= \Ad (b^{(0)})^{-1}\, \eta^{(-1)} \in  \Omega^{(1,0)}(\mathfrak{g}_{-1} )$ where $\mathfrak{sl}(W)=\sum_{j\in\Z_r}\mathfrak{g_j}$ is the $\tau$-eigenspace decomposition.
\end{Lem}
\begin{proof}
As $\rho_{\lambda}$ is unitarizable for each $\lambda\in \S^1$, there exists $a_{\lambda}\in \mathrm{SL}(W)$ such that
\[
a_{\lambda}\rho_{\lambda}a_{\lambda}^{-1}\colon \pi_1(N,q_0)\to \SU(W)
\]
is a unitary representation. Since $\rho_{\lambda}$ is irreducible for all $\lambda\in \S^1$, the unitarizer $a_{\lambda}$ is unique up to left multiplication by $\SU(W)$. Consider the real-analytic finite-dimensional Iwasawa splitting $
\mathrm{SL}(W)=\SU(W)\hat B$.
Thus, demanding that $a_{\lambda}\in \hat B$ for all $\lambda\in \S^1$ provides a unique choice of unitarizer.  
Since $\rho_\lambda$ depends real-analytically on $\lambda$, and the normalization $a_\lambda\in \hat B$ makes the unitarizer unique, the map $\lambda\mapsto a_\lambda$ is real analytic; hence $a_\lambda\in {\bf\SL}(\S^1)$. To verify that $a_{\lambda} \in {\bf\SL}^{\tau}(\S^1)$ is twisted, we use that $D_{\lambda}$ is an ${\bf\SL}^{\tau}(\S^1)$-connection and thus $\rho_{\xi\lambda}=\tau\rho_{\lambda}$. Therefore 
\[
a_{\xi\lambda}^{-1}\tau a_{\lambda}\rho_{\xi\lambda}(a_{\xi\lambda}^{-1}\tau a_{\lambda})^{-1}=\rho_{\xi\lambda}
\]
and, by irreducibility of $\rho_{\xi\lambda}$, we obtain
\[
a_{\xi\lambda}^{-1}\tau a_{\lambda}=c\,\Id\,,\quad c^{\dim W}=1\,.
\]
Since $\tau(\hat B)\subset \hat B$ the left hand side lies in $\hat B$ whereas the right hand side lies in $\SU(W)$, hence 
$c=1$. Hence $a_{\xi\lambda}=\tau a_{\lambda}$, i.e. $a\in{\bf\SL}^{\tau}(\S^1)$.

Let $H\colon V\times V\to\C$ be the Hermitian bundle metric on $V=N\times W$ given by the Hermitian structure on $W$. 
We may assume --- after gauging by the constant-in-$N$ twisted loop $a_\lambda$ --- that $\rho_\lambda$ is unitary.
Let $\hat{H}_{\lambda}\colon V\times V\to\C$ be the Hermitian bundle metric obtained via parallel transport of $H$ by the connection $D_{\lambda}$ for $\lambda\in \S^1$. Since $\rho_{\lambda}$ is unitary, $\hat{H}_{\lambda}$ is well-defined on $N$, that is, has trivial monodromy. Thus there is a smooth map $g_{\lambda}\colon N\to \mathrm{SL}(W)$ with $g_{\lambda}^{*}H=\hat{H}_{\lambda}$ for $\lambda \in \S^1$. Utilizing the unique splitting $\mathrm{SL}(W)=\SU(W)\hat B$, we also have $(g_{\lambda})_{\hat B}^{*}H=\hat{H}_{\lambda}$. This implies that we can choose $g_{\lambda}$ uniquely by requiring it to take values in $\hat B$ 
for all $\lambda\in \mathbb S^1.$
Then, $g_{\lambda}$ again depends real-analytically on $\lambda$.  In other words, $g_{\lambda}$ is an ${\bf\SL}(\S^1)$-gauge and, by construction,  $D_{\lambda}. g_{\lambda}$ is an ${\bf\SU}(\S^1)$-connection. 

It remains to show that $g_{\xi\lambda}=\tau g_{\lambda}$ is twisted. By construction, $g_{\lambda}$ is the unique real analytic $\hat B$-valued map such that  $D_{\lambda}. g_{\lambda}$ is an ${\bf\SU}(\S^1)$-connection. Using the irreducibility of the 
${\bf\SL}^{\tau}(\S^1)$-connection
$D_{\lambda}$ once more, we obtain that
\[D_{\xi\lambda}.g_{\xi\lambda}=(\tau D_\lambda).g_{\xi\lambda}\quad\text{and}\quad(\tau D_\lambda).(\tau g_\lambda)\] are both ${\bf\SU}(\S^1)$-connections. By irreducibility, this implies
 $\tau g_\lambda=g_{\xi\lambda}\, c\,\Id\,,$ $ c^{\dim W}=1$. As before, we obtain $c=1$ and therefore $\tau g_\lambda=g_{\xi\lambda},$
i.e. $g_{\lambda}$ is twisted.

We can now apply Theorem~\ref{thm:Iwasawa} to factorize 
$
g_{\lambda}=b_{\lambda}u_{\lambda}
$
pointwise on $N$ into a positive $\tau$-twisted and unitary $\tau$-twisted loop. 
Then also $D_{\lambda}. b_{\lambda}$ is an ${\bf\SU}^{\tau}(\S^1)$-connection via the ${\bf\SL}^{\tau}_{B}(\mathbb{D})$-gauge $b_{\lambda}$.
Finally, if $D_{\lambda}=d+\eta_{\lambda}$ is a DPW-connection, then  Remark~\ref{rem:DPW-connection} implies that 
$
D_{\lambda}. b_{\lambda}=D+\lambda^{-1}\Psi-\lambda \Psi^{\dagger}
$
is a flat ${\bf\SU}^{\tau}(\S^1)$-connection with $\Psi= \Ad (b^{(0)})^{-1}\, \eta^{(-1)}$.
\end{proof}
\begin{Rem}
In the minimal Lagrangian case of Example~\ref{ex:order6}, one may take $\hat B\subset \SL_3(\C)$ to be the subgroup of upper triangular matrices with positive real diagonal entries, and $B\subset \SL_3(\C)^\tau$ the corresponding solvable subgroup appearing in Theorem~\ref{thm:Iwasawa}. Then
\[
\SL_3(\C)=\SU_3\,\hat B,\qquad \SU_3\cap \hat B=\{\Id\},\qquad \tau(\hat B)=\hat B.
\]
Hence the hypotheses of Lemma~\ref{lem:tau-unitarize} are satisfied. The same is true for the conjugated automorphism $\hat\tau$ of Example~\ref{ex:order6}, after conjugating $\hat B$ by the matrix $c$ in \eqref{def.mTta}.
\end{Rem}

\begin{Cor}\label{cor:unitarize}
Let $\tilde{D}_{\lambda}$ be a  flat, not necessarily $\tau$-twisted, ${\bf\SL}(\S^1)$-connection on $V=N\times W$ and assume that its monodromy representation $\tilde{\rho}_{\lambda}\colon \pi_1(N,q_0)\to \mathrm{SL}(W)$ is irreducible and unitarizable for every $\lambda\in \S^1$. Assume that there exists an ${\bf\SL}(\S^1)$-gauge
$g_{\lambda}$ such that 
\(
D_{\lambda}:=\tilde{D}_{\lambda}. g_{\lambda}=d+\eta_{\lambda}
\)
is a $\tau$-twisted DPW-connection. Then there exists  
an ${\bf\SL}^{\tau}(\mathbb{D})$-gauge $b_{\lambda}$ so that  
\[
D_{\lambda}. b_{\lambda}=D+\lambda^{-1} \Psi-\lambda \Psi^{\dagger}
\] 
is a flat ${\bf\SU}^{\tau}(\S^1)$ connection with $\Psi= \Ad (b^{(0)})^{-1}\, \eta^{(-1)} \in  \Omega^{(1,0)}(N,\mathfrak{g}_{-1} )$ and $D$ an $\SU(W)^{\tau}$-connection. 
\end{Cor}
\begin{proof}
Because $D_{\lambda}$ is ${\bf\SL}(\S^1)$-gauge equivalent to $\tilde{D}_{\lambda}$, the monodromy  
of $D_{\lambda}$ is irreducible and unitarizable for every $\lambda\in \S^1$. Applying the  previous Lemma proves the claim.
\end{proof}

\subsection{$\tau$-structures}
As observed in Remark~\ref{rem:tauparallel}, a $\tau$-twisted connection
$
D_\lambda=d+\eta_\lambda
$
on a trivial bundle $N\times W$ is defined using a fixed automorphism $\tau$ that is parallel with respect to the trivial connection $d$. This formulation is inherently local. On a compact Riemann surface, the associated family of a harmonic map typically lives on a nontrivial bundle, and the existence of a global trivialization in which the twisting automorphism is constant would force the underlying $\SU(W)^\tau$-bundle
 to be trivial.

This is already too restrictive in the minimal Lagrangian case. There the natural bundle is
$
V=K\oplus \underline{\C}\oplus K^{-1}.
$
Although $V$ is topologically trivial as a complex rank-three bundle, its $\SU(\C^3)^\tau$-reduction involves the canonical bundle $K$ and is therefore nontrivial unless $K$ itself is trivial. On a compact Riemann surface of genus $g\neq 1$ one 
cannot
expect a global $\tau$-twisted DPW potential on the surface itself. This is the basic reason for introducing $\tau$-structures: they provide the intrinsic global replacement for a fixed twisted trivialization.

The passage from local twisted potentials to global objects has been carried out in special cases, such as minimal and CMC surfaces in the $3$-sphere (see for example \cite{HeC}), but a general framework for more complicated target spaces does not seem to be available in the literature. 

\begin{Def}
Let $W$ be a finite-dimensional Hermitian vector space, $\tau\colon \mathrm{SL}(W)\to \mathrm{SL}(W)$ an order $r$ Lie group automorphism preserving the Hermitian structure,  and $\xi=\exp(2\pi i/r)$ a primitive $r$-th root of unity. A Hermitian vector bundle $V\to M$ is an $\SU(W)^{\tau}$-bundle (or equivalently admits a $\tau$-structure)
 provided there exist isometric trivializations whose transition functions take values in $\SU(W)^{\tau}$. An $\S^1$-family $\nabla_{\lambda}$ on an $\SU(W)^{\tau}$-bundle $V$ is called an ${\bf\SU}^{\tau}(\S^1)$-connection if $\nabla_{\lambda}$ is an  ${\bf\SU}^{\tau}(\S^1)$-connection in every $\tau$-trivialization. 
 The same definition applies with ${\bf SU}^{\tau}(\S^1)$ replaced by any of the twisted loop groups introduced in Section~\ref{sec:loop-groups}.
\end{Def}

\begin{Rem}
An $\SU(W)^\tau$-bundle is the global replacement for a trivial bundle endowed with a fixed $d$-parallel twisting automorphism $\tau$. In the minimal Lagrangian case, the corresponding $\SU(\C^3)^\tau$-bundle is $K\oplus\underline{\C}\oplus K^{-1}$; it is trivial as an $\SU(\C^3)^\tau$-bundle only if $K$ is trivial.
\end{Rem}

\begin{Exa}
We recall from Lemma~\ref{lem:loop2mLag} that given a Riemann surface $M$ every minimal Lagrangian immersion $f\colon M\to\CP^2$ 
is obtained from a $\C^{\times}$-family of flat connections $d_{\lambda}$ on the Hermitian rank three bundle $V=K\oplus \underline{\C}\oplus K^{-1}$ satisfying the properties (i)-(v) in Lemma~\ref{lem:loop}.

The Hermitian bundle $V=K\oplus \underline{\C}\oplus K^{-1}$ with diagonal metric gives rise to an $\SU(W)^{\tau}$-structure with $W=\C^3$ and $\tau$ as in Example~\ref{ex:order6}. Choosing a local trivialization $\psi$ of $K$ and  corresponding trivialization $\psi^{-1}$ of $K^{-1}$, we locally trivialize $V$ by 
\begin{equation}\label{eq:tautriv}
(\psi,1,\psi^{-1})\mapsto (e_1,e_2,e_3)
\end{equation}
where $e_k$ denote the standard basis in $\C^3$. These trivializations are isometric and their transition functions take values in $\SU(\C^3)^{\tau}$.
Moreover, the family of flat connections $d_{\lambda}=D+\lambda^{-1}\Phi-\lambda\Phi^{\dagger}$  is an ${\bf\SU}^{\tau}(\S^1)$-connection on $V$.
The local trivializations  \eqref{eq:tautriv} map $d_{\lambda}$ to $d+\omega^{(0)}+\lambda^{-1}\omega-\lambda \omega^{\dagger}$ and $D$ to $d+\omega^{(0)}$ with $\omega^{(0)}\in \mathfrak{g}_0$, $\omega\in \mathfrak{g}_{-1}$,  $\omega^{\dagger}\in \mathfrak{g}_{1}$ as required for an ${\bf\SU}^{\tau}(\S^1)$-connection on $W$.  
\end{Exa}

In fact, every  $\SU(\C^3)^{\tau}$-bundle,
with $\tau$ as in Example~\ref{ex:order6},
 splits orthogonally as $V=E\oplus\underline{\C}\oplus E^{-1}$ for a unique line bundle $E\to M$. 
For a $\tau$-trivialization $F\colon V\to \underline{\C}^3$, we locally define an orthogonal decomposition $V=E_1\oplus E_2\oplus E_{3}$ via $F(E_k)=\C e_k$.  Since the transition functions of such trivializations take values in $\SU(\C^3)^{\tau}$ they are of the form $\diag(a,1,a^{-1})$. Therefore, the sub-bundles $E_k\subset V$ are globally well-defined. Moreover, $E_3=E_1^{-1}$  and $E_2=\underline{\C}$.

These observations provide a useful variant of Lemma~\ref{lem:loop2mLag}.
\begin{Lem}\label{lem:loop2mLag2}
Let $M$ be a Riemann surface, $V\to M$ a rank three Hermitian $\SU(\C^3)^{\tau}$-bundle with trivial determinant bundle, and let $d_{\lambda}=D+\lambda^{-1}\Phi-\lambda\Phi^{\dagger}$ be a flat ${\bf\SU}^{\tau}(\S^1)$-connection with $\Phi^2\in\Gamma(M,K^2\mathfrak{sl}(V))$ nowhere vanishing.
Then $V=K\oplus \underline{\C}\oplus K^{-1}$, $D$ is the diagonal connection, the Higgs field $\Phi\in\Gamma(K\mathfrak{sl}(V))$ is of the form \eqref{eq:Higgsfield} 
with respect to the induced splitting.
Moreover, $d_{\lambda}$ describes an $\S^1$-family of minimal Lagrangian immersions $f_{\lambda}\colon \tilde{M}\to \CP^2$ equivariant \eqref{eq:equiv} with respect to the monodromy $\rho_{\lambda}$ of $d_{\lambda}$. If  for some $\lambda_0\in \S^1$  the representation $\rho_{\lambda_0}=\Id$ is trivial, then $f_{\lambda_0}\colon M\to \CP^2$ and its special Legendrian lift $\hat{f}_{\lambda_0}\colon M\to \S^5$ descend to $M$. 
\end{Lem}
\begin{proof}
We have already seen in the discussion above  that the $\SU(\C^3)^{\tau}$-bundle splits orthogonally as $V=E\oplus \underline{\C}\oplus E^{-1}$ for a line bundle $E\to M$. If  we can verify the properties (i)-(v) of Lemma~\ref{lem:loop} then Lemma~\ref{lem:loop2mLag} proves our claims. Now (i), (ii), and (iii) hold by assumption. The discussion above also shows that 
in any local $\tau$-trivialization $F\colon V\to \underline{\C}^3$
the ${\bf\SU}^{\tau}(\S^1)$-connection $d_{\lambda}$ takes the form $d+\omega^{(0)}+\lambda^{-1}\omega-\lambda \omega^{\dagger}$ with $\omega_0\in \mathfrak{g}_0$, $\omega\in \mathfrak{g}_{-1}$ and $\omega^{\dagger}\in \mathfrak{g}_{1}$. In particular, $\omega^{(0)}$ preserves the orthogonal decomposition $\C^3=\C e_1\oplus \C e_2\oplus \C e_3$ and thus $D$, whose local expression is $d+\omega^{(0)}$, preserves the decomposition $V=E\oplus \underline{\C}\oplus E^{-1}$. 
Define the $D$-parallel gauges $g=\diag(\zeta^2,1,\zeta)$ for the cube root  $\zeta=\exp(2\pi i/3)$  and $b=\begin{psmallmatrix} 0&0&1\\0&1&0\\1&0&0\end{psmallmatrix}\colon V\to V^*$
which satisfy properties (iv) and (v) of Lemma~\ref{lem:loop} since they do so in the $\tau$-trivializations. We can now apply Lemma~\ref{lem:loop2mLag}.
\end{proof}

\begin{Rem}\label{Rem:sllift0}
The  minimal Lagrangian immersion $f_{\lambda_0}\colon M\to \CP^2$ from the previous Lemma can be viewed as the line subbundle $\underline{\C}\subset V$ in the decomposition $V=K\oplus \underline{\C}\oplus K^{-1}$ with respect to  the trivial background connection $d_{\lambda_0}$ as explained in \eqref{eq:f-recon}. Thus, in every  $\tau$-trivialization $\Upsilon\colon V\to \underline{\C}^3$, the minimal Lagrangian immersion is given as the preimage $f_{\lambda_0}=\Upsilon(\C e_2)$ of the line $\C e_2\subset \C^3$. Moreover, $\hat{f}_{\lambda_0}=\Upsilon(e_2)$ is a minimal Legendrian lift into $\S^5$.
 \end{Rem}
\begin{Rem}\label{Rem:sllift}
In later parts of the paper, we shall work with the conjugated order six symmetry $\hat{\tau}$ of Example~\ref{ex:order6}. In that case, the minimal Lagrangian immersion is given as $f_{\lambda_0}=\Upsilon(\C e_1)$ and the minimal Legendrian lift $$\hat{f}_{\lambda_0} = 
\Upsilon(e_1)=F^{-1}(e_1)$$ into $\S^5$ is  in fact {\em special Legendrian:} 
with $\omega_{-1}=\left(\begin{smallmatrix}0&a&0\\0&0&b\\ia&0&0\end{smallmatrix}\right)dz$ and $\omega_{1}=\left(\begin{smallmatrix}0&0&i\bar a\\-\bar a&0&0\\0&-\bar b&0\end{smallmatrix}\right)d\bar z$ for a holomorphic coordinate $z=x+iy$ and local functions $a,b$,  
the inverse of the $\hat \tau$-trivialization $F=\Upsilon^{-1}$ satisfies
\[
-dF F^{-1}=\omega_{\lambda_0}=\lambda_0^{-1}\omega^{(-1)}+\omega^{(0)}+\lambda_0 \omega^{(1)}\,.
\]
 Thus, the holomorphic volume form evaluates to 
 \[
 \Omega(\tfrac12F^{-1}e_1,\tfrac{\partial}{\partial z}F^{-1}e_1,\tfrac{\partial}{\partial \bar{z}}F^{-1}e_1)dz\wedge d\bar z=\tfrac12\det(e_1,ia e_3,-\bar ae_2)dz\wedge d\bar z=  |a|^2dx\wedge dy >0.
 \] 
 Thus, the phase $\varphi$ in \eqref{eq:sLeg} is a real positive constant. As $\varphi$ is by construction of length $1$, we obtain 
 $\varphi=1$
and the surface $F^{-1}(e_1)$ is special Legendrian.
 \end{Rem}

\begin{Rem*}
Our focus lies on the construction of minimal Lagrangian surfaces, and we do not aim to discuss further the general theory for  harmonic maps into compact symmetric spaces. For the case of negatively curved symmetric spaces, a general theory has been developed, see 
e.g. \cite{GP2}. 
\end{Rem*}

\subsection{Removable singularities}
We now pass to compact Riemann surfaces with punctures. The aim is to remove 
the singularities of meromorphic DPW-connections in a way compatible with the underlying $\tau$-structure.
\begin{Def}
An ${\bf\SL}^{\tau}(\S^1)$-connection $D_{\lambda}=d+\eta_{\lambda}$ is {\em meromorphic} if $\eta_{\lambda}$ is an 
${\bf sl}^{\tau}(\S^1)$-valued meromorphic 1-form on $N$ with poles contained in the punctures $\{q_1,\dots, q_n\}$. Thus, 
if $\eta_{\lambda}=\sum_{k\in\Z} \eta^{(k)} \lambda^k$ then $\eta^{(k)}$ are meromorphic $\mathfrak{g}_k$-valued 1-forms on $N$ whose poles are contained in the punctures. Since each coefficient $\eta^{(k)}$ is meromorphic, hence of type $(1,0)$, such a connection is automatically flat on $\mathring N$.

An ${\bf\SL}^{\tau}(\S^1)$-gauge $g_{\lambda}$  is {\em meromorphic} if $g_{\lambda}\colon N\to {\bf\SL}^{\tau}(\S^1)$ is meromorphic on $N$ with poles contained in the punctures $\{q_1,\dots, q_n\}$. In other words, if $g_{\lambda}=\sum_{k\in\Z} g^{(k)}\lambda^k$ then $g^{(k)}\colon N\to \mathfrak{gl}(W)$ are meromorphic with poles contained in the punctures.
\end{Def}

\begin{Def}
Let $\mathring{M}=M\setminus\{p_0\}$ be a punctured surface and $D_{\lambda}$ an ${\bf\SL}^{\tau}(\S^1)$-connection on $\mathring{M}\times W$.
We call $p_0\in M$  an ${\bf\SL}^{\tau}(\D)$-{\em apparent singularity} if there exists a gauge $g_{\lambda}\colon \mathring{M}\to {\bf\SL}^{\tau}(\D)$ such that the  ${\bf\SL}^{\tau}(\S^1)$-connection 
$D_{\lambda}. g_{\lambda}$ extends smoothly into $p_0\in M$. 
\end{Def}
The following lemma links the  removal of singularities with $\SU(W)^{\tau}$-bundles. 
\begin{Lem}\label{lem:apparent}
Let $\mathring{M}=M\setminus\{p_0\}$, and let $\mathring{\nabla}_{\lambda}$ be an ${\bf\SU}^{\tau}(\S^1)$-connection on $\mathring{M}\times W$, and $\tilde{\nabla}_{\lambda}$ an ${\bf\SU}^{\tau}(\S^1)$-connection on $M\times W$. Assume there exists an ${\bf\SL}^{\tau}(\mathbb{D})$-gauge $h_{\lambda}\colon \mathring{M}\to {\bf\SL}^{\tau}(\mathbb{D})$ with $h_{\lambda}(p_1)\in\SU(W)^{\tau}$ for $p_1\in \mathring{M}$ such that $\mathring{\nabla}_{\lambda}=\tilde{\nabla}_{\lambda}. h_{\lambda}$ are gauge equivalent over $\mathring{M}$.  
Then 
\begin{enumerate}[(i)]
\item
$h_{\lambda}=h_0\colon \mathring{M}\to\SU(W)^{\tau}$. 
\item
There exists an $\SU(W)^{\tau}$-bundle $V\to M$ with an ${\bf\SU}^{\tau}(\S^1)$-connection $\nabla_{\lambda}$ for which 
\[
\nabla_{\lambda}=\mathring{\nabla}_{\lambda}. \Upsilon
\]
in a $\tau$-trivialization $\Upsilon\colon V_{| \mathring{M}}\to \mathring{M}\times W$. 
\end{enumerate}
\end{Lem}
\begin{proof}
Since $\mathring{\nabla}_{\lambda}=d+\mathring{\omega}_{\lambda}$ with $\mathring{\omega}_{\lambda}\in\Omega^1(\mathring{M}, {\bf su}^{\tau}(\S^1))$ and  $\tilde{\nabla}_{\lambda}=d+\tilde{\omega}_{\lambda}$ with $\tilde{\omega}_{\lambda}\in\Omega^1(M, {\bf su}^{\tau}(\S^1))$, the gauge $h_{\lambda}$ satisfies
\[
dh_{\lambda}=h_{\lambda}\mathring{\omega}_{\lambda}-\tilde{\omega}_{\lambda}h_{\lambda}.
\]
Using $\mathring{\omega}_{\lambda}^{\dagger}=-\mathring{\omega}_{1/\bar{\lambda}}$ and likewise for $\tilde{\omega}_{\lambda}$, we see that  $h_{\lambda}^{-1}$ and $h_{1/\bar{\lambda}}^{\dagger}$ satisfy the same linear differential equation with the same initial condition at $p_1\in \mathring{M}$. Therefore, $h_{\lambda}^{-1}=h_{1/\bar{\lambda}}^{\dagger}$ and since $h_{\lambda}$ is an ${\bf\SL}^{\tau}(\mathbb{D})$-gauge, we conclude $h_{\lambda}=h_0$ is independent of $\lambda$ and thus takes values in $\SU(W)^\tau$. 
We now glue the trivial $W$-bundles over $\mathring{M}$ and $M$ using $h_0$ as the transition function to obtain the $\SU(W)^{\tau}$-bundle $V\to M$ with ${\bf\SU}^{\tau}(\S^1)$-connection $\nabla_{\lambda}$ as stated.
\end{proof}
With this at hand, we have the following  removable singularity result:
\begin{Lem}\label{lem:apparent1}
Let $\mathring{M}=M\setminus\{p_0\}$ with $M$ simply connected and $D_{\lambda}$ a flat ${\bf\SL}^{\tau}(\S^1)$-connection on $\mathring{M}\times W$ for which $p_0\in M$ is an ${\bf\SL}^{\tau}(\D)$-apparent singularity. Further assume there exists an  ${\bf\SL}^{\tau}(\D)$ gauge $b_{\lambda}\colon \mathring{M}\to {\bf\SL}^{\tau}(\D)$ so that $D_{\lambda}. b_{\lambda}$ is a flat ${\bf\SU}^{\tau}(\S^1)$-connection on $\mathring{M}\times W$. Then there exists an $\SU(W)^{\tau}$-bundle $V\to M$ and a  flat ${\bf\SU}^{\tau}(\S^1)$-connection $\nabla_{\lambda}$ on $V\to M$ such that 
\[
\nabla_{\lambda}=(D_{\lambda}. b_{\lambda}). \Upsilon
\]
for a $\tau$-trivialization $\Upsilon\colon V_{|\mathring{M}}\to \mathring{M}\times W$.
\end{Lem}
\begin{proof}
Since $p_0\in M$ is an ${\bf\SL}^{\tau}(\D)$-apparent singularity, there is a gauge $g_{\lambda}\colon \mathring{M}\to {\bf\SL}^{\tau}(\D)$ with
$D_{\lambda}. g_{\lambda}$ defined on $M$. Now $M$ is simply connected and thus we have  a gauge $k_{\lambda}\colon M\to {\bf\SL}^{\tau}(\S^1)$ such that 
$
D_{\lambda}. g_{\lambda}=d. k_{\lambda}
$
with $d$ the trivial connection on $M\times W$. 
Thus,
\[
D_{\lambda}. b_{\lambda}=(D_{\lambda}. g_{\lambda}). (g_{\lambda}^{-1}b_{\lambda})= d. (k_{\lambda}g_{\lambda}^{-1}b_{\lambda})\,.
\]
Applying the Iwasawa decomposition in Theorem~\ref{thm:Iwasawa}, we factorize $k_{\lambda}=(k_{\lambda})_u (k_{\lambda})_{+}$ with 
$(k_{\lambda})_{+}\colon M\to {\bf\SL}_B^{\tau}(\D)$ and $(k_{\lambda})_u\colon M\to {\bf\SU}^{\tau}(\S^1)$. Then
\[
D_{\lambda}. b_{\lambda}=(d. (k_{\lambda})_u). h_{\lambda}
\]
with $h_{\lambda}:=(k_{\lambda})_{+} g_{\lambda}^{-1}b_{\lambda}\colon \mathring{M}\to {\bf\SL}^{\tau}(\D)$. Since the gauge $k_{\lambda}$ can be chosen to satisfy $k_{\lambda}(p_1)=b_{\lambda}^{-1}(p_1)g_{\lambda}(p_1)$ at some  fixed $p_1\in\mathring{M}$, we have 
\[
1=k_{\lambda}(p_1)g_{\lambda}^{-1}(p_1)b_{\lambda}(p_1)=(k_{\lambda})_u(p_1) (k_{\lambda})_{+}(p_1) g_{\lambda}^{-1}(p_1)b_{\lambda}(p_1)
\]
with $(k_{\lambda})_u(p_1)\in {\bf\SL}^{\tau}(\D)\cap  {\bf\SU}^{\tau}(\S^1)=SU(W)^{\tau}$. Thus $h_{\lambda}(p_1)\in\SU(W)^{\tau}$ and we can apply Lemma~\ref{lem:apparent} to $\mathring{\nabla}_{\lambda}=D_{\lambda}. b_{\lambda}$ and $\tilde{\nabla}_{\lambda}=d. (k_{\lambda})_u$ to obtain the desired result.
\end{proof}
In the setting of this paper, we will have to pass to a suitable covering in order to be able to remove singularities. 
\begin{The}\label{the:sing-cover}
Let $\mathring{N}=N\setminus\{q_1,\dots,q_n\}$ be a punctured Riemann surface and $D_{\lambda}=d+\eta_{\lambda}$ a DPW-connection on $N\times W$ with poles $q_j\in N$. Assume there exists an  ${\bf\SL}^{\tau}(\D)$-gauge $b_{\lambda}\colon \mathring{N}\to {\bf\SL}^{\tau}(\D)$ so that $D_{\lambda}. b_{\lambda}$ is a flat ${\bf\SU}^{\tau}(\S^1)$-connection on $\mathring{N}\times W$ necessarily of the form
\[
D_{\lambda}. b_{\lambda}=D+\lambda^{-1}\Psi-\lambda\Psi^{\dagger}
\]
with $\Psi=\Ad b_0^{-1} \eta_{-1}\in\Omega^{1,0}(\mathring{N},\mathfrak{g}_{-1})$.  

If there exists a Riemann surface $M$ and a holomorphic covering $\pi\colon M\to N$ branched over  $\{q_1,\dots,q_n\}$ all of whose branch points are ${\bf\SL}^{\tau}(\D)$-apparent for $\pi^{*}D_{\lambda}$, then there exists an $\SU(W)^{\tau}$-bundle $V\to M$ and a flat ${\bf\SU}^{\tau}(\S^1)$-connection 
\[
\nabla_{\lambda}=\nabla +\lambda^{-1}\Phi-\lambda\Phi^{\dagger}
\qquad\text{
with}
\qquad
\nabla_{\lambda}=\pi^*(D_{\lambda}. b_{\lambda}). \Upsilon
\]
for a $\tau$-trivialization $\Upsilon\colon V_{|\mathring{M}}\to \mathring{M}\times W$.
\end{The}
\begin{proof}
Apply Lemma~\ref{lem:apparent1} on punctured disks around the branch points of $\pi$, and glue the resulting local $\SU(W)^\tau$-bundles by the induced $\SU(W)^\tau$-valued transition functions to obtain $V$ and $\nabla_{\lambda}$ on $M$.
\end{proof}

The following definition specifies under which condition we can apply Theorem \ref{the:sing-cover} on a covering.
\begin{Def}
Let $D_{\lambda}=d+\eta_{\lambda}$ be a meromorphic DPW-connection on $V=N\times W$ with singularity at $q.$ Assume there exists
a local branched covering $\pi\colon U\to N$ of degree $k$ totally branched over $q$. Assume that there exists an 
${\bf\SL}^{\tau}(\mathbb{D})$-gauge $G$ on $U\setminus\{\pi^{-1}(q)\}$
such that $(\pi^*D_{\lambda}).G$ extends smoothly to $\pi^{-1}(q)$. We then call  $q$ a $k$-{\em removable} singularity.
\end{Def}

\subsection{Minimal Lagrangian surfaces from meromorphic potentials}
From now on we restrict to the minimal Lagrangian case, so that $W=\C^3$ and $\tau$ is one of the two order-six automorphisms from Example~\ref{ex:order6}. 
The following definition specifies the conditions to get an immersion from a meromorphic DPW-connection with removable singularity.
\begin{Def}
Let $M$ be a Riemann surface. Then $\Phi\in H^0(M, K\mathfrak{sl}(3,\C))$ is called {\em regular} at $p\in M$ if $\Phi_p^2\neq0.$ It is called {\em regular} if it is regular at all points $p\in M.$

Let $D_{\lambda}=d+\eta_{\lambda}$ be a meromorphic DPW-connection on $V=N\times W$ with singular points at $q_1,\dots,q_n.$
Assume that for all $j=1,\dots,n$ there is $k_j\in\mathbb N^{>0}$
such that $q_j$ is a $k_j$-removable singularity. Assume that $\mathrm{Res}_{\lambda=0}\eta_{\lambda}\in H^0(\mathring N,K\mathfrak{sl}(3,\C))$ is regular on $\mathring N$.
Further, assume that  $\mathrm{Res}_{\lambda=0} (\mathrm{Ad}(G_j^{-1})\eta_{\lambda})$ is regular at $\pi_j^{-1}(q_j)$, 
where $G_j$ and the local covering $\pi_j$ are as in the definition of a $k_j$-removable singularity.
Then, $\eta_{\lambda}$ is called {\em immersive}.
\end{Def}

A priori, being immersive (at the points $\pi_j^{-1}(q_j)$) depends on the choice of  $G_j$. That this property is actually independent  of the choice 
can be shown similarly to the proof of Lemma \ref{lem:apparent1}.

To construct compact minimal Lagrangian immersions, we will use the following theorem.
The statement holds literally in the same way when replacing $\tau$ by $\hat\tau.$
\begin{The}\label{THM:MLrecon}
Let $D_{\lambda}=d+\eta_{\lambda}$ be a $\tau$-twisted meromorphic DPW-connection on $V=N\times W$.
Assume that 
\begin{enumerate}
\item each singular point $q_j$ is a $k_j$-removable singularity for some $k_j\in\N^{>0};$ 
\item for all $\lambda\in \S^1$ the monodromy $\rho_\lambda$ of $D_{\lambda}$ is unitary up to conjugation and irreducible;
\item there is $\lambda_0\in\S^1$ such that the monodromy $\rho_{\lambda_0}$
 has finite image;
\item the local monodromies around the singular points $q_j$ are of order 
$k_j;$
\item $\eta_{\lambda}$ is immersive.
\end{enumerate}

Then, there
 exists a finite degree holomorphic covering $\pi\colon M\to N$ such that $\pi^*D_{\lambda}$ is gauge equivalent on $\pi^{-1}(\mathring N)$ to the associated family $d_\lambda$ of a  minimal Lagrangian immersion $f\colon M\to \CP^2.$
Furthermore, $\Gamma:=\pi_1(\mathring N,z_0)/\ker(\rho_{\lambda_0})$ acts on $M$ by deck transformations, and $f$ is $\Gamma$-equivariant, i.e. $\rho_{\lambda_0}(\gamma^{-1})(f)=\gamma^*f$ for all $\gamma\in\Gamma.$ 
\end{The}
\begin{proof}
By (2), we can apply Corollary~\ref{cor:unitarize} on the punctured surface $\mathring N$. We thus obtain a ${\bf SL}^{\tau}(\mathbb D)$-gauge $b_\lambda$ such that
\[
\widetilde D_\lambda:=D_\lambda. b_\lambda = D+\lambda^{-1}\Psi-\lambda\Psi^\dagger
\]
is a flat ${\bf SU}^{\tau}(\S^1)$-connection on $\mathring N\times W$; since $D_\lambda$ is meromorphic DPW, the form $\eta^{(-1)}$ is holomorphic on $\mathring N$, hence $\Psi=\operatorname{Ad}(b^{(0)})^{-1})\eta^{(-1)}\in \Omega^{1,0}(\mathring N,\mathfrak g_{-1})$ is smooth.

Consider the branched holomorphic covering $\pi\colon M\to N$ specified by $\ker(\rho_{\lambda_0})\subset \pi_1(\mathring N,z_0)$ for some fixed $z_0\in\mathring N.$ In particular, by point (3), the branched covering $\pi$ has finite degree. In fact, the degree is the order of the finite group $\pi_1(\mathring N,z_0)/\ker(\rho_{\lambda_0})$. 
The map $\pi$
is branched exactly over the singular points $q_j$. By point (4),  the branch order of points lying over $q_j$ is $k_j-1$. 
Set $\mathring M=\pi^{-1}(\mathring N)$.
By the definition of the covering, the monodromy of the ${\bf\SU}^{\tau}(\S^1)$ connection  $\pi^*\widetilde D_\lambda$  on $\mathring M$ is trivial for $\lambda=\lambda_0.$ 

By Conditions (1) and (4), each branch point of $\pi$ is ${\bf SL}^{\tau}(\mathbb D)$-apparent for $\pi^*D_\lambda$, so Theorem~\ref{the:sing-cover} applies.
We thus obtain 
the existence of a
${\bf\SU}^{\tau}(\S^1)$-connection $\nabla_{\lambda}$ on a $\SU(W)^{\tau}$-bundle $V\to M$ and
 $\tau$-trivialization $\Upsilon\colon V_{|\mathring{M}}\to \mathring{M}\times W$ such that
 \begin{equation}\label{eq:Fgage}\nabla_\lambda=\pi^*\widetilde D_\lambda.\Upsilon.\end{equation}
 
 Then $\nabla_{\lambda_0}$ has trivial monodromy as well.
By point (5), the Higgs field $\Phi$ of $\nabla_\lambda=\lambda^{-1}\Phi+D-\lambda\Phi^\dagger$ satisfies $\Phi_p^2\neq0$ for all $p\in M.$
Thus, by Lemma \ref{lem:loop2mLag2}  we obtain a well-defined minimal Lagrangian immersion $f_{\lambda_0}\colon M\to \CP^2$
which admits a special Legendrian  lift $\hat f \colon M\to\S^5.$

 It remains to show that $f=f_{\lambda_0}$ is $\Gamma$-equivariant. 
To show this, apply  Lemma \ref{lem:loop2mLag2} directly to $\widetilde D_\lambda$ on $\mathring N$ to
obtain a $\tilde\rho_{\lambda_0}$-equivariant minimal Lagrangian immersion $\tilde f\colon \widetilde{\mathring N}\to\mathbb CP^2$.
Because of \eqref{eq:Fgage}, we obtain,  
up to an $\mathrm\SU_3$-transformation, $\tilde f=f\circ \tilde \pi,$ where 
\begin{equation}\label{uni-topo-cov}\tilde \pi\colon  \widetilde{\mathring N}\to \mathring M=\widetilde{\mathring N}/ \ker(\tilde\rho_{\lambda_0})\subset M
\end{equation}
is the induced covering. Since $\ker(\tilde\rho_{\lambda_0})\subset \pi_1(\mathring N,z_0)$ is normal, the result follows.
\end{proof}
\begin{Rem}\label{rem:general-tau-recon}
Theorem~\ref{THM:MLrecon} is largely independent of the specific finite order automorphism $\tau$ or $\hat\tau$ relevant for minimal Lagrangian surfaces. More generally, the same bundle-theoretic argument applies to meromorphic DPW-connections twisted by an arbitrary finite-order automorphism $\tau$, provided one has a reconstruction result identifying the resulting ${\bf SU}^{\tau}(\S^1)$-family with a primitive harmonic map into the corresponding $r$-symmetric space.

In geometric situations where this primitive harmonic map encodes a distinguished surface class -- such as minimal Lagrangian surfaces in the present paper, or hyperbolic affine spheres as in \cite{HOP} -- one must in addition impose the appropriate nondegeneracy condition ensuring that the resulting map is an immersion, equivalently that no branch points occur. In the present setting this is precisely the regularity condition $\Phi^2\neq 0$.
\end{Rem}

The following corollary gives a convenient reconstruction formula for later use.

\begin{Cor}\label{recon-formula}
Let $D_{\lambda}=d+\eta_{\lambda}$ be a $\tau$-twisted meromorphic DPW-connection on $N\times\C^{3}$ satisfying the hypotheses of Theorem~\ref{THM:MLrecon}, and set
$
\mathring N:=N\setminus\{q_1,\dots,q_n\}.
$
Let
$
b_{\lambda}\colon \mathring N\to {\bf SL}^{\tau}(\mathbb D)
$
be the positive gauge given by Corollary~\ref{cor:unitarize}, so that $D_{\lambda}.b_{\lambda}$ is a flat ${\bf SU}^{\tau}(\S^1)$-connection on $\mathring N$.
Fix a lift $\hat z_0\in \widetilde{\mathring N}$ of $z_0\in\mathring N$, and let
$
\Omega_{\lambda}\colon \widetilde{\mathring N}\to {\bf SL}^{\tau}(\mathbb D^{\times})
$
be the $\tau$-twisted parallel frame of $D_{\lambda}$ normalized by
$
\Omega_{\lambda}(\hat z_0)=b_{\lambda}(z_0).
$
Write its Iwasawa decomposition as
\[
\Omega_{\lambda}=B_{\lambda}F_{\lambda},
\qquad
B_{\lambda}\colon \widetilde{\mathring N}\to {\bf SL}^{\tau}_{B}(\mathbb D),\quad
F_{\lambda}\colon \widetilde{\mathring N}\to {\bf SU}^{\tau}(\S^1).
\]
Then $B_{\lambda}$ is single-valued on $\mathring N$ (more precisely, after pullback to $\widetilde{\mathring N}$ one has $B_{\lambda}=b_{\lambda}$).
Moreover, if $\pi\colon M\to N$ is the branched covering from Theorem~\ref{THM:MLrecon} and $\mathring M:=\pi^{-1}(\mathring N)$, then $F_{\lambda_0}$ has trivial monodromy on $\mathring M$. The reconstructed minimal Lagrangian immersion is given, up to an ambient $\SU_3$-transformation, by
$
(f_{\lambda_0})_{|\mathring M}=F_{\lambda_0}^{-1}(\C e_2).$
If one works with the conjugated symmetry $\hat\tau$, then $\C e_2$ is replaced by $\C e_1$; in particular, the special Legendrian lift is
\begin{equation}\label{eq:recon-formula}
\hat f_{\lambda_0}=F_{\lambda_0}^{-1}(e_1).
\end{equation}
\end{Cor}

\begin{proof}
We identify $b_{\lambda}$ with its pullback to $\widetilde{\mathring N}$, and set
\[
U_{\lambda}:=b_{\lambda}^{-1}\Omega_{\lambda}.
\]
Then $U_{\lambda}$ is a parallel ${\bf SL}^{\tau}(\D^\times)$-frame of the flat ${\bf SU}^{\tau}(\S^1)$-connection
$
D_{\lambda}.b_{\lambda}
$
on $\widetilde{\mathring N}$, and by the normalization of $\Omega_{\lambda}$ we have
$
U_{\lambda}(\hat z_0)=\Id.
$
Since $D_{\lambda}.b_{\lambda}$ is unitary, it follows that $U_{\lambda}\in {\bf SU}^{\tau}(\S^1)$ pointwise on $\widetilde{\mathring N}$. Hence
\[
\Omega_{\lambda}=b_{\lambda}U_{\lambda}
\]
is already an Iwasawa decomposition of $\Omega_{\lambda}$. By uniqueness of the Iwasawa decomposition,
i.e.
$B_{\lambda}=b_{\lambda},$ and $
F_{\lambda}=U_{\lambda}.
$
In particular, $B_{\lambda}$ is single-valued on $\mathring N$.
Since $B_{\lambda}$ is single-valued, $F_{\lambda}$ has the same monodromy as $\Omega_{\lambda}$, namely the monodromy of $D_{\lambda}$. By construction of the covering $\pi\colon M\to N$ in Theorem~\ref{THM:MLrecon}, this monodromy is trivial at $\lambda=\lambda_0$ on $\mathring M$. Thus $F_{\lambda_0}$ descends to $\mathring M$.

The reconstruction formula now follows from Remark~\ref{Rem:sllift0}; in the $\hat\tau$-twisted normalization one uses Remark~\ref{Rem:sllift}. This gives 
$
(f_{\lambda_0})_{|\mathring M}=F_{\lambda_0}^{-1}(\C e_2),$
and in the $\hat\tau$-case the special Legendrian lift is $F_{\lambda_0}^{-1}(e_1)$.
\end{proof}

\subsection{The energy of harmonic maps in terms of meromorphic potentials}\label{sec:energyresformula}
We conclude the general discussion by deriving a residue formula for the energy of a harmonic map from a Riemann surface into a symmetric space which is obtained from a meromorphic DPW-connection. 
Similar formulas have been obtained for  minimal CMC surfaces in $\S^3$, see \cite[Theorem 8]{H} and \cite[Corollary 4.3]{HHT0}, and for harmonic maps to
hyperbolic 3-space  \cite{BeHR,HHT1}. 
In the special case of the primitive harmonic map associated with a minimal Lagrangian immersion, this combines with \eqref{eq:areaML} to give an area formula in terms of the residues of the potential.

Let $\Sigma$ be a compact Riemann surface, and let 
\[
D_\lambda=d+\xi_\lambda,\qquad \xi_\lambda=\sum_{j\geq -1}\lambda^j\xi^{(j)}
\]
be a meromorphic DPW-connection on the trivial bundle $\Sigma\times W$ with singularities in $p_1,\dots,p_n\in \Sigma$. 
Set
$
\mathring\Sigma:=\Sigma\setminus\{p_1,\dots,p_n\}.
$
Assume that there exists a  gauge
$
B_\lambda\colon \mathring\Sigma\to {\bf SL}^{\tau}(\mathbb D)
$
such that
\[
\nabla_\lambda:=D_\lambda. B_\lambda
\]
is the associated family of flat connections of a (possibly equivariant) harmonic map $\psi$.  
In a $\tau$-trivialization of the corresponding $\SU(W)^\tau$-bundle, we may write
\[
\nabla_\lambda=d+\lambda^{-1}\omega^{(-1)}+\omega^{(0)}+\lambda\omega^{(1)},
\]
and obtain
\begin{equation}\label{eq:energy-density-general}
E(\psi)=i\int_\Sigma \tr(\omega^{(-1)}\wedge\omega^{(1)}).
\end{equation}

\begin{The}\label{thm:area-sing}
Assume that for every singular point $p_k$ there exists a gauge
\[
g_k=g_k^{(0)}+\lambda g_k^{(1)}+O(\lambda^2)\colon U_k\to {\bf SL}^{\tau}(\mathbb D)
\]
defined on a punctured neighbourhood $U_k$ of $p_k$ such that
$
D_\lambda. g_k
$
extends smoothly across $p_k$. Then
\begin{equation}\label{eq:en-form}
E(\psi)=2\pi\sum_{k=1}^n \Res_{p_k}\tr\bigl((g_k^{(1)}(g_k^{(0)})^{-1})\,\xi^{(-1)}\bigr).
\end{equation}
\end{The}

\begin{proof}
Choose disjoint coordinate discs $U_k\subset\Sigma$ around the poles $p_k$. After shrinking  $U_k$ if necessary, choose a smooth gauge
$
g\colon \mathring\Sigma\to {\bf SL}^{\tau}(\mathbb D)
$
which agrees with $g_k$ near $p_k$ for every $k$ and is smooth on $\mathring\Sigma$ away from the poles. Write
$
B_\lambda=g_\lambda\,b_\lambda,
$
where $b_\lambda\colon \Sigma\to {\bf SL}^{\tau}(\mathbb D)$. Then
\[
D_\lambda. g_\lambda=\nabla_\lambda. b_\lambda^{-1}
\]
is  a DPW-connection on $\Sigma$.
Write
\[
g_\lambda=g^{(0)}+\lambda g^{(1)}+\lambda^2 g^{(2)}+O(\lambda^3)
\qquad\text{and}\qquad
g_\lambda=(\Id+\lambda X+\lambda^2Y+O(\lambda^3))\,g^{(0)},
\]
with
$
X=g^{(1)}(g^{(0)})^{-1}$ and $ Y=g^{(2)}(g^{(0)})^{-1}.
$
Then
\[
D_\lambda. g_\lambda=
\Bigl(
\lambda^{-1}\xi^{(-1)}
+d+\xi^{(0)}+[\xi^{(-1)},X]
+\lambda\bigl(\xi^{(1)}+dX+[\xi^{(0)},X]+[\xi^{(-1)},Y]\bigr)
+O(\lambda^2)
\Bigr). g^{(0)}
\]
from which we obtain
\[
\omega^{(-1)}=\Ad_{(g^{(0)})^{-1}}\xi^{(-1)}
\qquad\text{
and}
\qquad
\omega^{(1)}=\Ad_{(g^{(0)})^{-1}}\bigl(\xi^{(1)}+dX+[\xi^{(0)},X]+[\xi^{(-1)},Y]\bigr).
\]
Using invariance of the trace, the fact that $\xi^{(-1)}$ is a meromorphic $(1,0)$-form, and the coefficient of $\lambda^{-1}$ in the flatness equation
$
d\xi_\lambda+\xi_\lambda\wedge\xi_\lambda=0,
$
namely
$
d\xi^{(-1)}+[\xi^{(0)}\wedge\xi^{(-1)}]=0,
$
we obtain
\[
\begin{split}
\tr(\omega^{(-1)}\wedge\omega^{(1)})
&=\tr\bigl(\xi^{(-1)}\wedge(\xi^{(1)}+dX+[\xi^{(0)},X]+[\xi^{(-1)},Y])\bigr)=\tr(\xi^{(-1)}\wedge dX)=-\,d\,\tr(X\,\xi^{(-1)}).
\end{split}
\]
Therefore \eqref{eq:energy-density-general} gives
\[
 E(\psi)
=
-i\int_{\mathring\Sigma} d\,\tr\bigl(g^{(1)}(g^{(0)})^{-1}\xi^{(-1)}\bigr).
\]
Applying Stokes' theorem on $\Sigma\setminus\bigcup_k U_k$ and the residue theorem yields \eqref{eq:en-form}
as claimed.
\end{proof}

\section{The relative character variety}\label{sec:chava}

\subsection{Motivation}
In the previous section, we explained how to construct minimal Lagrangian immersions $f\colon \Sigma\to \CP^2$ from meromorphic ${\bf\SL}^{\tau}(\S^1)$-connections $d_{\lambda}= d+\eta_{\lambda}$ which are required to be 
unitarizable for each $\lambda\in\S^1$. To achieve this, we start with  highly symmetric Riemann surfaces, namely the Fermat curves $\Sigma_{k}$ given by the smooth plane algebraic curve  
\[\Sigma_k:=\{[x,y,z]\in\CP^2\mid x^k+\zeta y^k+\zeta^2 z^k=0\}\]
of genus $\frac{1}{2} (k-1)(k-2)$, where  $\zeta=\exp(2\pi i/3)$. Consider the automorphisms
 \begin{equation}\label{eq:autos}
 \begin{split}
\varphi_1&\colon [x,y,z]\mapsto [\exp(\tfrac{2\pi i}{k}) x,y,z]\\
\varphi_2&\colon  [x,y,z]\mapsto [x,\exp(\tfrac{2\pi i}{k}) y,z]\\
\sigma&\colon  [x,y,z]\mapsto [y,z,x].\\
\end{split}
\end{equation}
A standard Riemann--Hurwitz calculation shows the quotient map $\Sigma_{k} \to \Sigma_{k}/ \langle\varphi_{1}, \varphi_{2}\rangle \cong \CP^1$ is a $k^{2}$-fold covering with three branch values.
The automorphism $\sigma$ naturally acts on the quotient $\Sigma_k/\langle\varphi_1,\varphi_2\rangle$ and permutes the three branch values.
Thus, the full quotient $\Sigma_{k} \to  \CP^1$ is a $3k^{2}$-fold covering with three branch values at $0,1,\infty$ of branch orders $2,k-1,2$ respectively.

Let $t\in(0,\tfrac13)$. Consider the subgroup $\Gamma_t\subset \SU_3$ generated by 
\begin{equation}\label{eq:finite-mon}
\begin{split}
M_0&:=\begin{psmallmatrix}
 0 & 0 & 1 \\
 1 & 0 & 0 \\
 0 & 1& 0 \\
\end{psmallmatrix},\quad
M_1:=\begin{psmallmatrix}
 1 & 0 & 0 \\
 0 & e^{-2 i \pi  t} & 0 \\
 0 & 0 & e^{2 i \pi  t} \\
\end{psmallmatrix},\quad
M_\infty:=\begin{psmallmatrix}
 0 & e^{2 i \pi  t} &0\\
0& 0 &  e^{-2 i \pi  t}  \\
 1& 0 & 0 \\
\end{psmallmatrix}
\end{split}
\end{equation}

satisfying $M_\infty M_1M_0=\Id.$ For $t = 1/k$ their orders are $3, k,$ and $3$, respectively, and the subgroup $\Gamma=\Gamma_{1/k}$ is finite (see for example \cite{SSTYY}).
Let
$\Lambda$ be the kernel of the
representation
\begin{equation} \label{finite-group-rep1}
\rho\colon \pi_1(S,z_0)\to \SU_3\subset\SL_3(\C);\quad \gamma_0\mapsto M_0,\, \gamma_1\mapsto M_1
\end{equation}

of the fundamental group of the thrice-punctured sphere $S=\CP^1\setminus\{0,1,\infty\}$ with natural generators $\gamma_0, \gamma_1$ (see \eqref{eq:fggen} below).
Then \[\Gamma_{deck}\,\cong\,\pi_1(S,z_0)/\Lambda\cong\langle\varphi_1,\varphi_2,\sigma\rangle\] acts on  $\Sigma_k$ 
from the right
by holomorphic automorphisms such that \begin{equation}\label{quotientidentity}\pi\colon \Sigma_k\to\Sigma_k/\Gamma_{deck}=\CP^1\end{equation} is a branched covering 
with branch points at $0,1$ and $\infty$ and branch orders $2,k-1,2$, respectively.

By the Riemann existence theorem (see, e.g., \cite[Chapter 4]{Donald})
 the compact Riemann surface is uniquely specified by the fact that the (topological) covering 
\begin{equation*}
\pi\colon \mathring\Sigma_k:=\Sigma_k\setminus\pi^{-1}(\{0,1,\infty\})\to S 
\end{equation*}
has monodromy $\rho.$ Moreover, the fundamental group of $ \mathring\Sigma_k$ is isomorphic to $\Lambda$. 
Let $\hat z_0$ be a lift of $z_0,$ i.e. $\pi(\hat z_0)=z_0.$
There is a natural surjective homomorphism $i: \Lambda \to \pi_{1}(\Sigma_k, \hat z_0)$ induced by the inclusion $\mathring\Sigma_k \to \Sigma_k$. Its kernel is generated by the simple loops around the branch  points $\Sigma_k \setminus \mathring\Sigma_k$. We have the following classical fact, which we state as a lemma. 
\begin{Lem}
Let $\tilde{\rho}\colon \pi_1(S,z_0)\to \SL_3(\C)$ be a representation of the thrice-punctured sphere into $\SL_3(\C)$ such that the local conjugacy classes of  $\tilde{\rho}$ are the same as those of $\rho$. Then $\tilde{\rho}$ yields, up to conjugation, a unique representation
$
\hat{\rho} : \pi_{1}(\Sigma_k,\hat z_0) \to \SL_3(\C)
$
such that
\begin{equation}\label{eq:monext}
\hat{\rho} \circ i =\tilde \rho|_{\Lambda}.
\end{equation}
The representation of $\pi_{1}(\Sigma_k)$ corresponding to $\rho$ in \eqref{finite-group-rep1} is the trivial representation.
\end{Lem}
\begin{proof}
Since $\tilde\rho|_{\Lambda}$ vanishes on the simple curves around the removed points $p \in \Sigma_k \setminus \mathring\Sigma_k$, the existence of $\hat{\rho}: \pi_{1}(\Sigma_k,\hat z_0) \to \SL_3(\C)$ satisfying \eqref{eq:monext} is clear.
The final statement, namely that the lift of the representation $\rho$ to $\pi_1(\Sigma_k)$ is trivial, follows directly from \eqref{quotientidentity}.
\end{proof}

\subsection{Relative character varieties}
Let $\Sigma_{g, n}$ be a surface of genus $g$ with $n$ punctures. If $n$ is at least one, then its fundamental group $\pi_{1} (\Sigma_{g,n})$ is a free group on $2g+n-1$ generators. Let $G =\SL_3(\C)$, and consider the representation variety Hom$(\pi_{1}( \Sigma_{g,n}), G)$ consisting of group homomorphisms of the fundamental group $\pi_{1}( \Sigma_{g,n})$ to the group $G$. 
This is identified with the product $G^{2g+n-1}$. 
For a fixed element $g \in G$, define $\mathcal C(g)$ to consist of all elements $h \in G$ so that the closures of the conjugacy orbits of $g$ and $h$ have nonempty intersection. If $\gamma_{i}$ denotes a peripheral curve oriented counterclockwise around the puncture $p_{i}$ and if $h_{i}$ is an element of $G$, then the relative representation variety is the subset  
\[\mathrm{Hom}^{rel}(\pi_{1}( \Sigma_{g,n}), G) := \{\rho \in \text{Hom} (\pi_{1}( \Sigma_{g,n}), G): \rho(\gamma_{i}) \in \mathcal C(h_{i}) \}.\] There is a natural action by conjugation on the space, whose quotient is not necessarily Hausdorff. The relative character variety then is
\[\chi^{rel}(\pi_{1}( \Sigma_{g,n}), G) : = \mathrm{Hom}^{rel}(\pi_{1}( \Sigma_{g,n}), G) \sslash G,\]
where $\sslash G$ is the so-called GIT quotient that identifies two representations $\rho$ and $\rho'$ whenever the closures of their orbits under conjugation intersect. The resulting space is thus the largest Hausdorff quotient of the relative representation variety. In the situations relevant here, the representations are generically irreducible, and hence their conjugation orbits are closed.

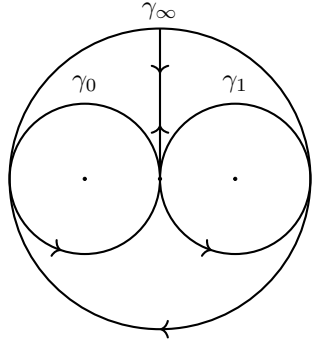
\begin{figure}[t]
  \centering
\begin{tikzpicture}[scale=2.]

\draw[thick,
  postaction={decorate},
  decoration={markings, mark=at position 0.7 with {\arrow{>}}}
] (0,0) circle (0.5);
\node at (0,0.63) {$\gamma_0$};
\filldraw (0,0) circle (0.3pt);

\draw[thick,
  postaction={decorate},
  decoration={markings, mark=at position 0.7 with {\arrow{>}}}
] (1,0) circle (0.5);
\node at (1,0.63) {$\gamma_1$};
\filldraw (1,0) circle (0.3pt);

\draw[thick,
  postaction={decorate},
  decoration={markings, mark=at position 0.25 with {\arrow{>}}}
] (1.5,0) arc (0:-360:1);
\node at (0.5,1.12) {$\gamma_\infty$};

\draw[thick,
  postaction={decorate},
  decoration={markings,
    mark=at position 0.3 with {\arrow{>}},
    mark=at position 0.7 with {\arrow{<}}
  }
] (0.5,1) -- (0.5,0);

\filldraw (0.5,0) circle (0.3pt);
\end{tikzpicture}
\caption{Curves representing the homotopy classes $\gamma_0$, $\gamma_1,$ and $\gamma_\infty.$}
\label{alphabetagamma}
\end{figure}

We take our underlying Riemann surface to be the thrice-punctured sphere $S = \CP^1 \setminus\{0, 1, \infty \} = \C \setminus \{ 0,1\}$. Choosing the base point $z_{0} = \tfrac{1}{2}$, 
we choose representatives $\gamma_j$ of the homotopy classes in $S$ by the following closed curves.

\begin{equation}\label{eq:fggen}
\begin{split}
&\gamma_0\colon [0,1)\mapsto \tfrac{1}{2}e^{2\pi is}\,, \qquad \gamma_1\colon [0,1)\mapsto 1-\tfrac{1}{2}e^{2\pi is}\,\\
&\gamma_\infty=\begin{cases}\tfrac{1}{2}+4 i s: \quad s\in[0,\tfrac{1}{4}] 
\\ \tfrac12+  e^{-4\pi is-\tfrac{\pi i}{2}}: \quad s\in[\tfrac{1}{4}, \tfrac{3}{4}]
\\ \tfrac{1}{2}+4i (1-s): \quad s\in[\tfrac{3}{4},1] \end{cases},\\
\end{split}
\end{equation}

so that $\gamma_\infty*\gamma_1*\gamma_0=\Id$ 
is the trivial loop, as shown in Figure \ref{alphabetagamma}.
We will use the convention that concatenation $(\beta*\alpha)(s)$ is given by $\alpha(2s)$ for $s\leq\tfrac{1}{2}$ and $\beta(2s-1)$ if $s\in[\tfrac{1}{2},1]$, so that the fundamental group acts from the right on the universal covering, and the monodromy of a connection is a representation of the fundamental group.

Motivated by the fact that \eqref{eq:finite-mon} describes the covering monodromy of the Fermat curves,
we fix the following three conjugacy classes for the local monodromies $M_0,M_1,M_\infty$ of $\gamma_{0}, \gamma_{1}, \gamma_{\infty}$, respectively, satisfying
\begin{equation}\label{local_conjugacy_classes}
\begin{split}
&M_1\in \mathcal C(\text{diag}(1, \exp(2\pi i t),\exp(-2 \pi i t))),\quad
 M_0,M_\infty\in \mathcal C(\text{diag}(\zeta,1,1/\zeta) )\\
 &M_\infty M_1M_0=\Id
 \end{split}
\end{equation}
 where $t\in(0,\tfrac{1}{3}).$ As each local monodromy has distinct eigenvalues, the conjugacy orbits are themselves closed.

Fixing the conjugacy classes of $M_0$, $M_1$, and $M_\infty$ --
which are determined by their traces, determinants, and the traces of their squares -- we consider their {\em global trace coordinates}
\begin{equation}\label{XYZdef}
\begin{split}
X&:=\tr(M_1 M_{0}^{-1}),\qquad
Y=\tr(M_1^{-1}M_{0}),\qquad\text{and}\qquad
Z=\tr(M_1 M_{0}M_{1}^{-1}M_{0}^{-1}).\\
\end{split}
\end{equation}
By applying the formulas of Lawton \cite{Law07,Law} with the conjugacy classes as in \eqref{local_conjugacy_classes}, the global traces 
satisfy $P(X,Y,Z)=0$, where $P$ is
 Lawton's polynomial
\begin{equation}\label{Def-characterpolynomial}
\begin{split}
P:= &
\,\, 2\, (2 \cos (2 \pi  t)+\cos (4 \pi  t)) (X Y-Z)+4 \cos (6 \pi  t)\\
&
+X^3-X Y (Z+3)+Y^3+Z^2+5.
\end{split}
\end{equation}
This equation can also be obtained by a case-by-case analysis using elementary linear algebra; see, e.g., \cite{HOP} for a related case on a three-punctured sphere with different local conjugacy classes. Lawton shows these global traces determine points on the relative character variety.
Thus
\[F=\{(X,Y,Z)\in\C^3\mid \;P(X,Y,Z)=0\}\]
is naturally identified with the relative character variety. In what follows, we need to determine the irreducible representations that are unitarizable. A first step is the following lemma.

\begin{Lem}\label{lem:tildeZ}
For $M_0,M_1,M_\infty$ and $X, Y$ as above, consider 
 the quadratic polynomial \eqref{Def-characterpolynomial}
in the variable $Z$. Its two  zeros are given by
\[Z=\tr(M_1M_0M_1^{-1}M_0^{-1})\quad\mathrm{and}\quad
\tilde Z=\tr(M_0M_1M_0^{-1}M_1^{-1}).\]
\end{Lem}
\begin{proof}
This follows by explicit  inspection of all possible cases, or by a direct computation using \cite[Equation 3.1.1]{Law}.
\end{proof}

\begin{Exa}\label{exa:finitemon}
Let $t=1/k$ for $k\in\N^{>3}.$
The representation $\rho$ given by \eqref{finite-group-rep1}
 with finite image $\Gamma$ has trace coordinates
\[X=Y=0\quad\text{and}\quad Z=2\exp(-2\pi i t)+\exp(4\pi it).\]
Note that the contragredient representation $\tilde\rho=(\rho^{-1})^T$ has trace coordinates $X=Y=0$ and $\tilde Z=2\exp(2\pi i t)+\exp(-4\pi it).$
It generates a finite group $\tilde\Gamma\leq\SU_3$ which is isomorphic to $\Gamma$ even though $\rho$ and $\tilde\rho$
are not conjugate. Furthermore
\[\ker(\rho\colon\pi_1(S,z_0)\to\SU_3)=\ker((\rho^{-1})^T\colon\pi_1(S,z_0)\to\SU_3),\]
and both representations become trivial on the same covering surface \eqref{quotientidentity}.
\end{Exa}

\subsection{Reducible and irreducible representations}\label{sec:redrep}
 Let  
$\rho$ be a reducible representation with
$M_0=\rho(\gamma_0)$ and $M_1=\rho(\gamma_1).$ 
We want to determine its trace coordinates.

The first case is that of a common eigenline of $M_0$ and $M_1$. Let  $0<t<\tfrac{1}{3}$.
Since $M_\infty M_1 M_0=\Id$, the product of the three eigenvalues along a common eigenline is 1.
This implies that the  eigenvalues on this eigenline are either all 1, or the eigenvalue of $M_{0}$ is $\zeta^{\pm 1}$, of $M_{1}$ is 1,  and of $M_{\infty}$ is $\zeta^{\mp 1},$ where $\zeta=\exp(\tfrac{2\pi i}{3}).$

We first diagonalize $M_{0}$. By a direct computation, up to conjugation, this leads in the first case to
\begin{equation*}
\begin{split}
M_{0}&=
\begin{psmallmatrix}
 1 & 0 & 0 \\
 0 & \zeta & 0 \\
 0 & 0 & \zeta^{-1} \\
\end{psmallmatrix}
\qquad\text{and}\quad
M_{1}=\begin{psmallmatrix}
 1 &b_{11}& b_{12} \\
 0 & \left(1-\frac{i}{\sqrt{3}}\right) \cos (2 \pi  t)+\frac{i}{\sqrt{3}} & \frac{1}{3} (-4) \sin ^2(\pi  t) (2 \cos (2 \pi  t)+1) \\
 0 & 1 & \left(1+\frac{i}{\sqrt{3}}\right) \cos (2 \pi  t)-\frac{i}{\sqrt{3}} \\
\end{psmallmatrix}
\end{split}
\end{equation*}
for some $b_{11},b_{12}\in\C$.
The trace coordinates are
\begin{equation}\label{singXYZ1}
\begin{split}
X_1&=4 \sin ^2(\pi  t),\qquad
Y_1=4 \sin ^2(\pi  t),\qquad\text{and}\qquad
Z_1=-2 \cos (2 \pi  t)+2 \cos (4 \pi  t)+3=
\tilde Z_1.
\\
\end{split}
\end{equation}

The other two cases yield the trace coordinates  $(X_2,Y_2,Z_2)$ and $(X_3,Y_3,Z_3)$ with
\begin{equation}\label{singXYZ2}
\begin{split}
X_2&=-2 i \left(\sqrt{3}-i\right) \sin ^2(\pi  t) =Y_3\\
Y_2&=2 i \left(\sqrt{3}+i\right) \sin ^2(\pi  t)=X_3\\
Z_2=\tilde Z_2&=-2 \cos (2 \pi  t)+2 \cos (4 \pi  t)+3=Z_3=\tilde Z_3.\\
\end{split}
\end{equation}

In the case of a two-dimensional invariant subspace of the representation generated by $M_0$ and $M_1$, we obtain analogously (or by looking at the inverse transpose representation)
the same possible traces as in the case of a 1-dimensional invariant subspace.

Note that the three points given by  \eqref{singXYZ1} and \eqref{singXYZ2} 
are precisely the singular points of the relative character variety $F.$ 
We will denote by \[F^*=F\setminus\{(X_1,Y_1,Z_1),(X_2,Y_2,Z_2),(X_3,Y_3,Z_3)\} \] the set of nonsingular points, so that $F^*$ corresponds to irreducible representations only.

The following  theorem is a special case of  \cite[Theorem 8]{Law07} (see also \cite{Law}).
It can be proved by elementary methods, e.g. as in \cite[Theorem 4 (2)]{HOP} for a different relative character variety for the three-punctured sphere.
Note that the ``only if'' part of the first part of the following theorem was just proven above. 
\begin{The}\cite[Theorem 8]{Law07}\label{irreps}
A representation with $M_{0},M_{1},M_{\infty}$ satisfying \eqref{local_conjugacy_classes} 
is reducible if and only if the trace coordinates are given by \eqref{singXYZ1} or \eqref{singXYZ2}.

For every $p=(X,Y,Z)\in F$, there exists a representation of the fundamental group of the thrice-punctured sphere
with prescribed local conjugacy classes \eqref{local_conjugacy_classes} and global traces $X,$ $Y,$ and $Z.$
Moreover, two irreducible representations have the same traces $X,Y,Z$
if and only if they are conjugate to each other.
\end{The}

As a consequence of the preceding Theorem, we do not distinguish between points in $F^*$ and conjugacy classes of irreducible representations.

\subsection{The unitary component}\label{sec:unirep}
To solve the intrinsic closing conditions for minimal Lagrangian surfaces, we derive conditions on $(X,Y,Z)\in F$ corresponding to unitarizable representations.

For any $A\in\mathrm{SL}(3,\C)$ with 1 as an eigenvalue, we have
$\tr(A)=\tr(A^{-1}).$

Thus, a  unitarizable representation given by $M_{0},M_{1},M_{\infty}\in\SU_3$ 
satisfying \eqref{local_conjugacy_classes}
also satisfies for all $k=1,2$ and $j = 0, 1, \infty$
\begin{equation*}
\begin{split}
\tr(M_{j}^k) =\tr(\bar M_{j}^k)=\tr((\bar M_{j}^T)^{-k}).\\
\end{split}
\end{equation*}
In particular, all the {\em local traces} are real.
Additionally, we have
\begin{Lem}\label{nec_un_con}
If a representation is unitarizable then its trace coordinates satisfy
\begin{equation}\label{eq:necun}\bar X=Y \quad\text{and} \quad \tilde Z=\bar Z\end{equation}
\end{Lem}
\begin{proof}
Without loss of generality, we assume that all matrices are unitary. Then 
\begin{equation*}
\begin{split}
X=&\tr(M_{1}M_{0}^{-1})=\tr((\bar M^T_{1})^{-1}\bar M^T_{0})
=\tr(\bar M_{0}(\bar M_{1})^{-1})=\bar Y,
\end{split}
\end{equation*}
and analogously for $Z$ and $\tilde Z.$
\end{proof}

Note that the polynomial $P(X,Y,Z)$ \eqref{Def-characterpolynomial} in the variable $Z$ has real coefficients if $\bar X=Y$. Therefore,
Lemma \ref{nec_un_con} implies that the discriminant
\begin{equation}\label{eq:discrimant-complex}
\begin{split}
D:=\,\, &(4 \cos (2 \pi  t)+2 \cos (4 \pi  t)+X Y)^2\\
&-4 \left(2 X Y (2 \cos (2 \pi  t)+\cos (4 \pi  t))+4 \cos (6 \pi  t)+X^3-3 X Y+Y^3+5\right)\\
\end{split}
\end{equation}
of the quadratic polynomial $P$ in $Z$
must be nonpositive in order to have a unitarizable representation.
We
denote by $F^\R$ the set of points satisfying \eqref{eq:necun}, i.e.,
\[F^\R=\{(X,Y,Z)\in F\mid \bar X=Y, \bar Z=\tilde Z\}\]
where $\tilde Z$ is uniquely determined by $(X,Y,Z)$ by Lemma \ref{lem:tildeZ}.

\begin{Lem}\label{suf_un_con}
For a connected component of the set of nonsingular real points $(X,Y,Z)\in F^\R\cap F^*$, 
either all or none of the elements in that component correspond to unitarizable representations.
\end{Lem}
\begin{proof}
Using Lemma \ref{lem:tildeZ}, we obtain from Theorem \ref{irreps} that
the representation $(\bar \rho^T)^{-1}$
determined by
$\tilde M_j:=(\bar M_j^T)^{-1}$, $j=0,1,\infty$,
is conjugate to $\rho.$ Thus,  there is a (unique up to scale by a third root of unity) conjugator $C\in\mathrm{SL}(3,\C)$ satisfying
\[C^{-1}\rho C=(\bar \rho^T)^{-1}.\]
Applying this relation twice we get
\[C(\bar C^T)^{-1}=\hat\zeta\,\Id\quad \text{ with} \quad \hat\zeta^3=1.\]
After scaling with a third root of unity, we can assume without loss of generality that
$C=\bar C^T$
is Hermitian. The signature of the associated quadratic form is either $(3,0)$ or $(1,2).$ In the first case,
$C$ is positive definite and can be written as $D\bar D^T.$ Then, conjugating yields a unitary representation $D^{-1}\rho D.$

The Hermitian choice $C$ is unique, hence  it depends continuously along paths in a real component of irreducible representations, and the signature does not change. For a unitarizable  representation,  the signature is $(3,0)$.
\end{proof}

To analyze the different connected components of $F^\R$ we substitute
 \begin{equation}\label{XYxy}X=A+i B,\quad Y=A-iB\qquad\text{with}\quad A,B\in\R.\end{equation}
Then, the discriminant
 \eqref{eq:discrimant-complex} is a real polynomial of degree 4
\begin{equation}
\begin{split}
D^r=&
-8 \cos (2 \pi  t) \left(A^2+B^2-1\right)-4 \cos (4 \pi  t) \left(A^2+B^2-2\right)-8 \cos (6 \pi  t)+2 \cos (8 \pi  t)\\&+A^4-8 A^3+12 A^2+2 (A (A+12)+6) B^2+B^4-10.
\end{split}
\end{equation}

For a given real $A$, there are at most four real solutions of $D^r=0$ in $B$, counted with multiplicity. They are given by
\begin{equation*}
\begin{split}B_1=&\sqrt{-4 H+4 \cos (2 \pi  t)+2 \cos (4 \pi  t)-A^2-12 A-6}\\
B_2=&\sqrt{4 H+4 \cos (2 \pi  t)+2 \cos (4 \pi  t)-A^2-12 A-6}\\
B_3=&-\sqrt{-4 H+4 \cos (2 \pi  t)+2 \cos (4 \pi  t)-A^2-12 A-6}\\
B_4=&-\sqrt{4 H+4 \cos (2 \pi  t)+2 \cos (4 \pi  t)-A^2-12 A-6},\\
\end{split}
\end{equation*}
where $H:=\sqrt{(-\cos (2 \pi  t)+A+1)^2 (4 \cos (2 \pi  t)+2 A+5)}.$

From this, one sees that
there are four different connected components inside $F^\R\cap F^*$, see Figure \ref{Fig:4components}.
Their closures are obtained by adding the three singular points of $F$. Explicitly, these are given by the following four closed and connected sets
\begin{equation}\label{eq-components0}
\begin{split}
\mathcal C_{-1}^+&:=\{(A+iB,A-iB,Z)\in F\mid (A,B)\in\R^2, -2 \cos (2 \pi  t)-\frac{5}{2}\leq A\leq\cos (2 \pi  t)-1,\; B_1\leq B\leq B_2\}\\
\mathcal  C_{-1}^-&:=\{(A+iB,A-iB,Z)\in F\mid (A,B)\in\R^2, -2 \cos (2 \pi  t)-\frac{5}{2}\leq A\leq\cos (2 \pi  t)-1,\; B_3\leq B\leq B_4\}\\
\mathcal  C_+&:=\{(A+iB,A-iB,Z)\in F\mid (A,B)\in\R^2, 4 \sin ^2(\pi  t)\leq A\leq 8 \cos ^2\left(\frac{\pi  t}{2}\right) \cos (\pi  t),\; B_4\leq B\leq B_2\}.
\end{split}
\end{equation}
and
\begin{equation}\label{eq-components}
\begin{split}
\mathcal  C_u&:=\{(A+iB,A-iB,Z)\in F\mid (A,B)\in\R^2,\cos (2 \pi  t)-1\leq A\leq -8 \sin ^2\left(\frac{\pi  t}{2}\right) \cos (\pi  t),\; B_1\leq B\leq B_2\}\\
&\cup\{(A+iB,A-iB,Z)\in F\mid (A,B)\in\R^2,\cos (2 \pi  t)-1\leq A\leq -8 \sin ^2\left(\frac{\pi  t}{2}\right) \cos (\pi  t),\; B_4\leq B\leq B_3\}\\
&\cup\{(A+iB,A-iB,Z)\in F\mid (A,B)\in\R^2, -8 \sin ^2\left(\frac{\pi  t}{2}\right) \cos (\pi  t)\leq A\leq 4 \sin ^2(\pi  t),\; B_4\leq B\leq B_2\}.\\
\end{split}
\end{equation}
\begin{figure}[t]
\centering
\includegraphics[width=0.3825\textwidth]{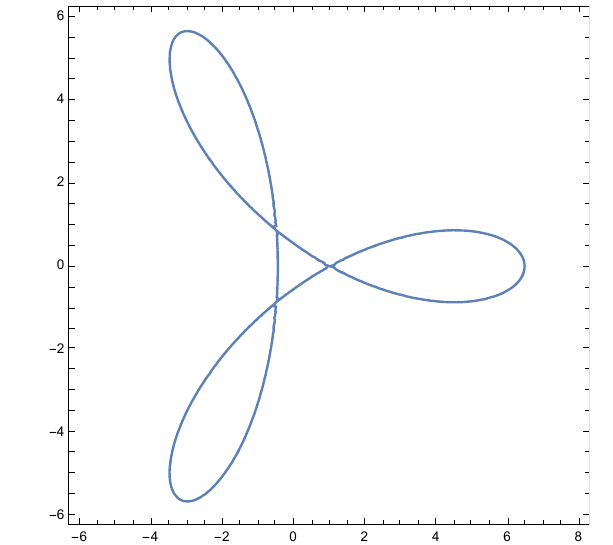}
\includegraphics[width=0.3975\textwidth]{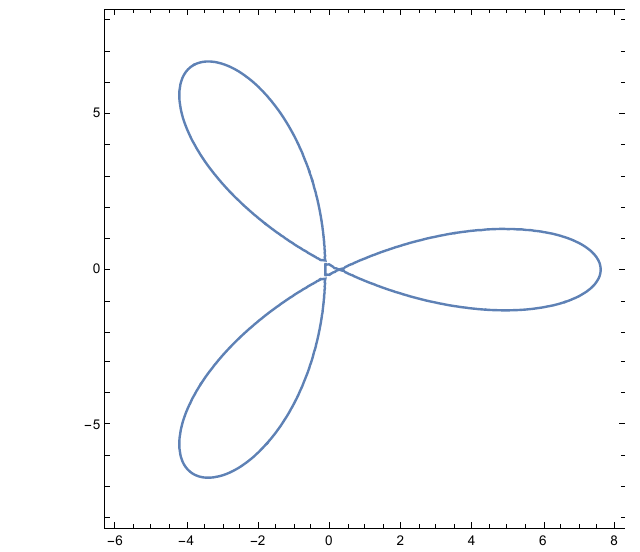}
\caption{The vanishing locus of $D^r$, for $t=1/6$ and $t=1/12$. The unitary component $\mathcal  C_u$ is the small component in the center.}
\label{Fig:4components}
\end{figure}

The intersections of these components consist of singular points of $F$ (i.e., corresponding to reducible representations) only.
Furthermore, if
$(X,Y,Z)\in \mathcal C\subset F^\R$ is contained in some real component $\mathcal C$, then automatically $(X,Y,\tilde Z)\in \mathcal  C\subset F^\R$.

We are now able to characterize  the relevant unitarizable representations through their trace coordinates.

\begin{The}\label{The:unitary-component}
Let $0<t<\tfrac{1}{3}, $ and let $\mathcal C_u\subset F^\mathbb R$ be as in \eqref{eq-components}.
Then $\mathcal C_u$
is a connected component
 which consists of unitarizable
representations only.
\end{The}
\begin{proof}
The representation
given by \eqref{finite-group-rep1}
is obviously unitary. It has trace coordinates $X=0=Y$ and $Z=2 e^{-2 i \pi  t}+e^{4 i \pi  t}$ is therefore contained in $\mathcal C_u$.  By Theorem \ref{irreps}, $\rho$ is irreducible.
By Lemma  \ref{suf_un_con}, $\mathcal  C_u$ entirely consists of unitarizable representations. 
\end{proof}

\subsection{Derivative of trace coordinates at $t=0$}
We will make use of the following lemma.

\begin{Lem}\label{derivativeofmonodromy-singularID}
Let $\epsilon>0$, and, for $j=0,1,\infty$, let
$M_j\colon [0,\epsilon)\to\mathrm{SL}(3,\C)$ 
 be smooth such that
$M_\infty M_1 M_0=\Id.$
Assume that for all $t$, $M_0(t)$ and $M_\infty(t)$ are both conjugate to
$\text{diag}(1,\zeta,\zeta^{-1}).$
Moreover, assume \begin{equation}\label{eqMM010}M_0(0)=\text{diag}(1,\zeta, \zeta^{-1})\quad\text{ and}\quad M_1(0)=\Id.\end{equation} Let
$X(t)=\tr(M_1(t)M_0(t)^{-1})$, $Y(t)=\tr(M_0(t)M_1(t)^{-1})$
and
$M'_1(0)=
\begin{pmatrix}
 p_{i\,j} 
\end{pmatrix}_{i,j}
.$
Then we have
\begin{equation}\label{secordtr}
\begin{split}
X(0)&=Y(0)=0,\quad
\quad X'(0)=Y'(0)=0\\
X''(0)&=-\zeta\, p_{12}\, p_{21}- \zeta^{-1}\, p_{13}\, p_{31}- p_{23}\, p_{32}\\
 Y''(0)
&=-\zeta^{-1}\, p_{12}\, p_{21}-\zeta \, p_{13}\, p_{31}- p_{23}\, p_{32}.\\
\end{split}
\end{equation}
\end{Lem}
\begin{proof}
That $X(0)=0=Y(0)$ directly follows from \eqref{eqMM010}.
For the remaining formulas, we first consider the case where $M_0(t)=M_0(0)$ is constant in time. 
From the conditions on the conjugacy class of $M_\infty(t)=(M_1(t)M_0(t))^{-1}$, we obtain
the equations 
\[ \tr(M_1 M_0)'(0)=0,\quad \tr(M_1 M_0 M_1 M_0)'(0)=0, \quad \tr(M_1 M_0 M_1 M_0M_1 M_0)'(0)=0.\]
This is a nondegenerate homogeneous linear system on $p_{11},p_{22},p_{33}$, i.e.,
\begin{equation}\label{psvanish}0=p_{11}=p_{22}=p_{33}.\end{equation}
By direct computation, one then obtains
$X'(0)=0=Y'(0).$
Let $M''_1(0)=
\begin{pmatrix}
 q_{i\,j} 
\end{pmatrix}_{i,j}.$
From
\[\tr(M_1 M_0)''=0=\tr(M_1 M_0 M_1 M_0)''\]
we obtain (using \eqref{psvanish})
\begin{equation}\label{eq:qqpp}
\begin{split}q_{11}+\zeta \,q_{22}+\zeta^{-1}\,q_{33}&=0\\
q_{11}+\zeta^{-1}\,q_{22}+\zeta\,q_{33}&=-\zeta\, p_{12}p_{21}-\zeta^{-1}\, p_{13}p_{31}-p_{23}p_{32}.\end{split}\end{equation}
Plugging \eqref{eq:qqpp} into
 \[X''(0)=q_{11}+\zeta^{-1}\,q_{22}+\zeta \,q_{33}\]
 one obtains \eqref{secordtr}. Similarly one can derive the formula for $Y''(0).$
 
 The general case \[M_0(t)=C(t)M_0(0)C^{-1}(t)\]
 is simply reduced to the special case by conjugating with $C(t)$: Using $M_1(0)=\Id$ and the Leibniz rule we obtain
 \[\frac{d}{dt}_{\mid t=0}(C^{-1}(t) M_1(t) C(t))=M_1'(0).\]
 Therefore, \eqref{secordtr} is true in general.
\end{proof}

\section{Construction of unitarizable DPW potentials}\label{sec:solvingIFT}

In this section we construct, for sufficiently small $t\ge0$, a family of Fuchsian potentials on the thrice-punctured sphere whose gauged DPW potentials satisfy the hypotheses of Theorem \ref{THM:MLrecon}. 
The monodromy will be unitarizable for every $\lambda\in\S^1$; the local conjugacy classes are prescribed so that, after pullback to the relevant covering, the singularities become removable; and at one Sym point $\lambda_0\in\S^1$ the pulled-back monodromy becomes trivial.

\subsection{Basic definitions}
Fix $r>1$ and let $\xi=\exp(\pi i/3)$.
Denote by
\[
\mathcal{W}:=\{h\colon \mathbb D_r\to\mathbb C\mid h \text{ is holomorphic} \}\qquad\text{and}\qquad  \mathring{\mathcal{W}}:=\{h\colon \mathbb A_r\to\mathbb C\mid h \text{ is holomorphic} \}
\]
the Banach spaces of bounded holomorphic functions on the disk $\mathbb D_r:=\{\lambda\in\C\mid |\lambda|<r\}$ 
and on the annulus $\mathbb A_r:=\{\lambda\in\C\mid \tfrac 1r<|\lambda|<r\}$, respectively.
 Elements of $\mathcal W$ and $\mathring{\mathcal{W}}$ restrict to  real-analytic functions on $\S^1\subset\mathbb C^\times$.
Define the bounded (real) linear maps
\begin{equation}\begin{split}
(.)^+&\colon \mathring{\mathcal{W}}\to \mathcal{W}; \quad f=\sum_{k\in\Z}f^{(k)}\lambda^k\mapsto  \sum_{k\in\N^{>0}}f^{(k)}\lambda^k.\\
(.)^*&\colon \mathring{\mathcal{W}}\to \mathring{\mathcal{W}}; \quad f=\sum_{k\in\Z}f^{(k)}\lambda^k\mapsto  \sum_{k\in\Z}\bar f^{(k)}\lambda^{-k}.
\end{split}
\end{equation}
Clearly, $f\in\mathring{\mathcal W}$ is real along $\S^1$ if and only if $f^*=f.$ Set
\begin{equation*}
\begin{split}
&\mathring{\mathcal{W}}^{\tau}:=\{f\in\mathring{\mathcal{W}}\mid f(\xi\lambda)=f(\lambda)\ \text{for all}\ \lambda\in\mathbb A_r\},
\qquad
\quad
\mathcal{W}^{\tau}:=\mathcal W\cap \mathring{\mathcal{W}}^{\tau},\\
&\mathring{\mathcal{W}}^{\tau}_{\R}:=\{f\in\mathring{\mathcal{W}}^{\tau}\mid \overline{f(\bar\lambda)}=f(\lambda)\ \text{for all}\ \lambda\in\mathbb A_r\},
\qquad
\quad
\mathcal{W}^{\tau}_{\R}:=\mathcal{W}\cap\mathring{\mathcal{W}}^{\tau}_{\R}.
\end{split}
\end{equation*}
These spaces are (real) closed subspaces of $\mathring{\mathcal W}$ respectively $\mathcal W$, hence Banach spaces. Elements of 
$\mathring{\mathcal{W}}^{\tau}_{\R}$ and of $\mathcal W^\tau_\R$ are real-valued along the real axis.
Since \(\mathbb A_r\) contains \(\S^1\), all these functions restrict to
real-analytic loops on $\S^1.$ Thus the Laurent
expressions in \(\lambda\) used in the potentials below, with coefficients
in \(\mathcal W^\tau\), define loop-algebra-valued meromorphic
potentials as in Section~\ref{sec:loop-groups}.

For a smooth Banach space $\mathcal E$-valued map $f\colon [0,\epsilon)\to\mathcal E$ with $f(0)=f'(0)=0$, define
\begin{equation}\label{wh-function}
\widehat f(t):=
\begin{cases}
t^{-2}f(t),& t>0,\\[4pt]
\tfrac12 f''(0),& t=0.
\end{cases}
\end{equation}

\subsection{The potential}
Let $t\in[0,\tfrac13)\subset\mathbb R$. We consider Fuchsian potentials of the form
\begin{equation}\label{Lawson-type-potential}
\eta=\eta_{t,a,b,c}=R_0\frac{dz}{z}+R_1\frac{dz}{z-1},
\end{equation}
where
\begin{equation}\label{eqR1}
\begin{split}
R_0&=
\left(
\begin{array}{ccc}
 0 & -\tfrac{t \left(a+\lambda ^3 b\right)}{\lambda ^3} & 0 \\
 0 & -\tfrac{1}{3} & 0 \\
 t \left(i \lambda ^3 b-i a\right) &i c t & \tfrac{1}{3} \\
\end{array}
\right),\\
R_1&=
\left(
\begin{array}{ccc}
 0 & \tfrac{t \left(a+\lambda ^3 b\right)}{\lambda ^3} & \tfrac{t \left(i \lambda ^3 b-i a\right)}{\lambda ^3} \\
 -t \left(a+\lambda ^3 b\right) & 0 & i c t \\
 -t \left(i \lambda ^3 b-i a\right) & -i c t & 0 \\
\end{array}
\right),\quad R_{\infty}=-R_0-R_1
\end{split}
\end{equation}
are the residues at $0$, $1$, and $\infty$, respectively. 
 The connection $d+\eta$ 
 is not yet in DPW form, since its coefficients may have higher order poles in $\lambda$. The next lemma shows that, under mild conditions, $d+\eta$ can be gauged into a $\hat{\tau}$-twisted DPW potential.

\begin{Lem}\label{harmonic-family-form}
Let $t\in(0,\tfrac13)$ and $r>1$. Suppose $a,b,c\in\mathcal W^\tau$ satisfy the following:
\begin{enumerate}[(i)]
\item $4ab-c^2=-1$;
\item $a(\lambda)\neq 0$ for all $\lambda\in \D_r$. 
\end{enumerate}
Then there exists a ${\bf\SL}(\D^\times)$-gauge $g$ such that the gauged potential $(d+\eta). g$ is a $\hat{\tau}$-twisted DPW potential and has the local monodromies \eqref{local_conjugacy_classes}.

The poles $z=0$ and $z=\infty$ of $(d+\eta). g$ are $3$-removable, and if $t=\tfrac1k$ for some $k\in\N^{>3}$, then pole $z=1$ is $k$-removable. Moreover, the DPW potential $(d+\eta). g$ is immersive.
\end{Lem}
\begin{proof}
For a non-resonant Fuchsian system, the conjugacy classes of the local monodromies are determined by the eigenvalues of the residues. The eigenvalues of $R_0$ and $R_\infty$ are $\tfrac13$, $0$, and $-\tfrac13$. By condition (i), the eigenvalues of $R_1$ are $t$, $0$, $-t$. Thus the local monodromies satisfy \eqref{local_conjugacy_classes}.

Consider the non-positive gauge
\begin{equation}\label{gauge-Lawson-DPW}
g:=
\begin{psmallmatrix}
 \frac{1}{\lambda } & 0 & 0 \\
 0 & \lambda  & iz \\
 0 & 0 & 1 \\
\end{psmallmatrix}.
\end{equation}
After gauging, one obtains
\begin{equation}\label{DPWpotentialomega}
d+\omega:=(d+\eta). g=d+\left(
\begin{array}{ccc}
 0 & \frac{t \left(a+\lambda ^3 b\right)}{\lambda  (z-1) z} & i\tfrac{2  \lambda  b t}{z-1} \\
 -\frac{2 \lambda  b t}{z-1} & \frac{1-z (1+3 c t)}{3 (z-1) z} & i\frac{-3 c t+1}{3 \lambda } \\
i \frac{ t \left(a-\lambda ^3 b\right)}{\lambda  (z-1) z} & i\frac{c \lambda  t}{z-z^2} & -\frac{1-z (1+3 c t)}{3 (z-1) z} \\
\end{array}
\right) dz.
\end{equation}
Because $a,b,c\in\mathcal W^\tau$ are $\Z_6$-invariant and due to \eqref{eq:Tau-es}, the connection $d+\omega$ is a $\hat{\tau}$-twisted DPW potential.

It remains to desingularize $d+\omega$ on suitable coverings by $\hat{\tau}$-twisted positive gauges. For $z=0$ and $z=\infty$, we use the $3$-fold covering
\[
\pi_3\colon z\mapsto z^3
\]
totally branched over $0$ and $\infty$. Pulling back \eqref{DPWpotentialomega} gives
\begin{equation}\label{pi3om}
\pi_3^*(d+\omega)=d+\left(
\begin{array}{ccc}
 0 & \frac{3 t \left(a+\lambda ^3 b\right)}{\lambda  z \left(z^3-1\right)} &i \frac{6  \lambda  b t z^2}{z^3-1} \\
 -\frac{6 \lambda  b t z^2}{z^3-1} & \frac{1+z^3 (-1-3  c t)}{z \left(z^3-1\right)} & i\frac{z^2 (-3  c t+1)}{\lambda } \\
 i\frac{3 t \left(a-\lambda ^3 b\right)}{\lambda  z \left(z^3-1\right)} & i\frac{3 c \lambda  t}{z-z^4} & - \frac{1+z^3 (-1-3  c t)}{z \left(z^3-1\right)} \\
\end{array}
\right)\,dz.
\end{equation}

For $z=0$, consider the $\hat{\tau}$-twisted positive gauge
\begin{equation}\label{eq:desgo}
g_0:=
\mathrm{diag}
 (1, 
 z , \tfrac{1}{z}) .
\end{equation}
Then
\begin{equation}\label{eq:desgoga}
(\pi_3^*(d+\omega)). g_0=d+\left(
\begin{array}{ccc}
 0 & \frac{3 t \left(a+\lambda ^3 b\right)}{\lambda  \left(z^3-1\right)} & i\frac{6 \lambda  b t z}{z^3-1} \\
 -\frac{6 \lambda  b t z}{z^3-1} & \frac{-3  c t z^2}{z^3-1} & i\frac{-3 c t+1}{\lambda } \\
i \frac{3  t \left(a-\lambda ^3 b\right)}{\lambda  \left(z^3-1\right)} & -i\frac{3 c \lambda  t z}{z^3-1} & \frac{3  c t z^2}{z^3-1} \\
\end{array}
\right)dz
\end{equation}
is nonsingular at $z=0$. The $\lambda^{-1}$-part $\Phi_0$ of \eqref{eq:desgoga} satisfies $\Phi_0^2\neq0$ since $a(0)\neq0$
by condition (ii). Hence $d+\omega$ has a $3$-removable singularity at $z=0$, and the resulting DPW potential is immersive there.

For $z=\infty$, consider the $\hat{\tau}$-twisted positive gauge
\begin{equation}\label{eq:ginfdes}
g_\infty:=
\begin{psmallmatrix}
 1 & 0 & 0 \\
 0 & -i z^2 & 0 \\
 0 & \frac{\lambda }{z} & \frac{i}{z^2} \\
\end{psmallmatrix}.
\end{equation}
We obtain
\[
(\pi_3^*(d+\omega)). g_\infty=d+\left(
\begin{array}{ccc}
 0 & \frac{3 i t z \left(\lambda ^3 b-a\right)}{\lambda  \left(z^3-1\right)} & -\frac{6 \lambda  b t}{z^3-1} \\
 -\frac{6 i \lambda  b t}{z^3-1} & -\frac{3 i c t}{z-z^4} & \frac{3 c t-i}{\lambda  z^2} \\
 \frac{3 t z \left(a+\lambda ^3 b\right)}{\lambda  \left(z^3-1\right)} & -\frac{3 c \lambda  t}{z^3-1} & \frac{3 i c t}{z-z^4} \\
\end{array}
\right)dz,
\]
which is smooth at $z=\infty$, i.e. $z=\infty$ is a $3$-removable singularity.
Moreover, we obtain as in the case of $z=0$ that the $\lambda^{-1}$-part $\Phi_\infty$ satisfies  $\Phi_\infty^2\neq0,$ i.e.
$z=\infty$ is also immersive.

For $z=1$ and $t=\tfrac1k$ with $k\in\N^{>3}$, consider the $k$-fold covering
\[
\pi_k(w)=w^k+1.
\]
Then
\[
\pi_k^*(d+\omega)=d+\left(
\begin{array}{ccc}
 0 & \frac{a+\lambda ^3 b}{\lambda  w^{k+1}+\lambda  w} & i\frac{2  \lambda  b}{w} \\
 -\frac{2 \lambda  b}{w} & \frac{-\frac{k w^k}{w^k+1}-3 c}{3 w} &i \frac{(-3 c+ k) w^{k-1}}{3 \lambda } \\
 i\frac{ \left(a-\lambda ^3 b\right)}{\lambda  w \left(w^k+1\right)} & -i\frac{c \lambda }{w^{k+1}+w} & \frac{\frac{k w^k}{w^k+1}+3  c}{3 w} \\
\end{array}
\right)dw.
\]
Define
\begin{equation}\label{eq:gkdesinG}
g_k:=\left(
\begin{array}{ccc}
 1 & \frac{(1- c) \lambda ^2 (w-1)}{2 a w} & -i\frac{(c-1) \lambda  (w-1)}{a w} \\
 \frac{( c-1) \lambda  (w-1)}{a} & \frac{4 a^2 w^2-(c-1)^2 \lambda ^3 (w-1)^2}{4 a^2 w} & -i\frac{ ( c-1)^2 \lambda ^2 (w-1)^2}{2 a^2 w} \\
 i\frac{(c-1) \lambda ^2 (w-1)}{2 a} & \frac{i \lambda  (w-1) \left(4 a^2 (w+1)-( c-1)^2 \lambda ^3 (w-1)\right)}{8 a^2 w} & \frac{4 a^2+(c-1)^2 \lambda ^3 (w-1)^2}{4 a^2 w} \\
\end{array}
\right).
\end{equation}
Then $\det(g_k)=1$, and $g_k$ is 
a ${\bf\SL}^{\hat\tau}(\D)$-gauge since
$a(\lambda)\neq0$ for all $\lambda\in\D_r$, and
because $a,b,c\in\mathcal W^\tau$. A direct computation using condition (i) shows that
\[
\pi_k^*(d+\omega). g_k
\]
is holomorphic at $w=0$. Moreover, the Higgs field is
\[
\Phi_1=\left(
\begin{array}{ccc}
 0 & \frac{a}{w^k+1} & 0 \\
 0 & 0 & \frac{i}{3} (-3  c+ k) w^{k-3} \\
 \frac{i a}{w^k+1} & 0 & 0 \\
\end{array}
\right).
\]
At $w=0$ the Higgs field $\Phi_1$ is $3$-step nilpotent, i.e $\Phi_1^2\neq0$ but $\Phi_1^3=0$, provided $k\geq 4$. Therefore, for $t=\tfrac1k$ with $k\in\N^{>3}$, the point $z=1$ is an immersive, $k$-removable singularity of $d+\omega$.
\end{proof}

\subsection{The monodromy functions}
Let $\eta=\eta_{t,a,b,c}$ with $a,b,c\in\mathcal W^\tau_\R$ as above.
Let $\Psi=\Psi_{t,a,b,c}$ be a solution of
\begin{equation}\label{eq:odes}
d\Psi_{t,a,b,c}+\eta_{t,a,b,c}\Psi_{t,a,b,c}=0,
\qquad
\Psi_{t,a,b,c}\bigl(\tfrac12\bigr)=\Id.
\end{equation}
Then $\Psi_{t,a,b,c}$ depends real-analytically on $(t,a,b,c).$
We denote by $M_0=M_0(t,a,b,c)$, $M_1=M_1(t,a,b,c)$, and $M_\infty=M_\infty(t,a,b,c)$ the monodromies around $0$, $1$, and $\infty$, respectively. We also write
\[X=X_{t,a,b,c}\quad\text{and}\quad Y=Y_{t,a,b,c}\]
for the corresponding global trace coordinates,
but we usually omit the dependence on $(t,a,b,c)$, or on $(a,b,c)$, to keep the notation light.

\begin{Lem}\label{Lem:twist-monodromy}
Assume that $a,b,c\in\mathcal{W}^{\tau}$. Then, for all $t\in[0,\tfrac13)$ and for all $\lambda\in\mathbb D_r^\times$,
 \[
 X(\zeta \lambda)=X(\lambda), \qquad  Y(\zeta \lambda)=Y(\lambda), \qquad  X(-\lambda)=Y(\lambda).
 \]
\end{Lem}
\begin{proof}
The first two identities follow directly from $f(\zeta\lambda)=f(\lambda)$ for all $f\in\mathcal W^\tau$ and all $\lambda\in\D_r.$
 For the last one, we use $f(-\lambda)=f(\lambda)$ for all $f\in\mathcal W^\tau$ and all $\lambda\in\D_r,$ and consider
\[
h=\lambda
\begin{psmallmatrix}
 -\frac{1}{\lambda ^3} & 0 & 0 \\
 0 & 0 & i \\
 0 & -i & 0 \\
\end{psmallmatrix}.
\]
Then, using $\eta_\lambda=\eta_{t,a(\lambda),b(\lambda),c(\lambda)}$ for $\lambda\in\mathbb D_r^\times$,
\[
(d+\eta_\lambda). h=d-\eta_{-\lambda}^T.
\]
Since gauge-equivalent connections have conjugate monodromy, and the monodromy of $d-\eta_{-\lambda}^T$ along $\delta$ is the transpose inverse of the monodromy of $d+\eta_{-\lambda}$ along $\delta^{-1}$, the trace coordinates satisfy $X(-\lambda)=Y(\lambda).$

\end{proof}

Motivated by \eqref{XYxy}, we define for $t\in[0,\tfrac13)$,
\[
A=A_{t,a,b,c}:=\frac12(X+Y),
\qquad
B=B_{t,a,b,c}:=\frac i2(-X+Y).
\]

\begin{Lem}\label{lem:ABprop}
Assume that $a,b,c\in\mathcal W^\tau_\R$ and $t\in[0,\tfrac13).$
Then, $A\in \mathring{\mathcal W}^\tau_\R$ and $\lambda^3 B\in \mathring{\mathcal W}^\tau_\R.$
\end{Lem}
\begin{proof}
The Riemann-Hilbert map from the space of Fuchsian systems to its monodromy is a complex analytic map (to its image).
Taking traces, and composition with the map $(a,b,c)\mapsto \eta_{t,a,b,c}$ thus shows that
$A,B\in\mathring{\mathcal W}.$ In particular, $A$ and $B$ are bounded on the annulus $\mathbb A_r.$

It then follows directly from Lemma \ref{Lem:twist-monodromy} and their definition that $A,\lambda^3B\in\mathring{\mathcal W}^\tau.$ To show the 
reality properties,
consider
\[
\tilde g=\mathrm{diag}(i,i,-1)
\qquad\text{and}\qquad
d+\tilde\eta:=(d+\eta).\tilde g.
\]
Let $\kappa(z)=\bar z$. A direct computation using $a,b,c\in\mathcal W^\tau_\R$ gives
\[
\kappa^*\overline{(d+\tilde\eta)}=d+\tilde\eta.
\]
Since complex conjugation sends the standard loops $\gamma_0,\gamma_1$ to their inverses, the corresponding monodromies satisfy 
$\overline{M_j}\sim M_j^{-1}.$ It follows from the definition of $X$ and $Y$ that $\overline{X}=Y,$ hence $A$ and $B$ are real-valued on the real axis.

Since the paths defining $X$ and $Y$ are exchanged by $\kappa$, it follows from the definitions that $A,B\in\mathring{\mathcal W}_\R$, i.e. $A, \lambda^3 B\in\mathring{\mathcal W}^\tau_\R$ as claimed.
\end{proof}
\begin{Rem}
The same antiholomorphic symmetry  $\kappa$ used in the proof of Lemma~\ref{lem:ABprop} also implies
$
\overline{Z(\lambda)}=Z(\bar\lambda),
$
hence $Z\in \mathring{\mathcal W}_{\R}$ and, in particular, $Z$ is real along the real axis.
This reality condition should not be confused with unitarity: along the real axis it comes from the involution $z\mapsto \bar z$, whereas the unitary condition relevant on $\S^1$ is $\bar X=Y$ and $\bar Z=\tilde Z$.
\end{Rem}

Set $\lambda_0=\exp(\tfrac{2\pi i}{12})$, and  define
\begin{equation}\label{eq:int-extr}
\begin{split}
\mathcal  F=\mathcal  F_{t,a,b,c}&:=( A-A^*)^+\\
\mathcal  G=\mathcal  G_{t,a,b,c}&:=( B- B^*)^+\\
\mathcal H=\mathcal  H_{t,a,b,c}&:=4ab-c^2+1\\
\mathcal S=\mathcal  S_{t,a,b,c}&:=A(\lambda_0).\\
\end{split}
\end{equation}

\begin{Pro}\label{pro:int-extcond}
Let $t\in(0,\tfrac13)$ be fixed. 
Let $a,b,c\in\mathcal W^\tau_\R$.
Assume that $a(\lambda)\neq 0$ and $c(\lambda)\neq 0$ for all $|\lambda|\leq 1$.
Assume that
\begin{equation}\label{eq:mondrom}\mathcal F=\mathcal G=\mathcal H=0 \quad \text{and} \quad \mathcal S=0.\end{equation}
Then,  the monodromy of
$d+\eta_{t,a,b,c}$ is unitarizable for every $\lambda\in \S^1,$ and the 
monodromy is irreducible for every $\lambda$ with $0<|\lambda|\leq 1$.
Moreover, $B(\lambda_0)=0,$ and 
if $t$ is rational, the monodromy at $\lambda_0$ is $\rho$ or $(\rho^{-1})^T$ defined in \eqref{finite-group-rep1},
and generates a  finite subgroup isomorphic to $\Gamma$ or $\tilde \Gamma$ 
as in Example \ref{exa:finitemon}.
\end{Pro}
\begin{proof}
We first show that $A$ and $B$ are real along the unit circle $\lambda\in \S^1.$ This is equivalent to $A^*=A$ and $B^*=B.$
Let $A=\sum_kA^{(k)}\lambda^{6k}.$
By reality $A\in \mathring{\mathcal W}_\R$, $\bar A^{(k)}=A^{(k)}$ for all $k.$
In particular, $A^{(0)}\in\R$, and 
therefore $A-A^*$ has no constant part.
Since $(A-A^*)^+=\mathcal F=0$, its negative part also vanishes by applying $(\cdot)^*$, hence $A-A^*=0.$
Analogously, one obtains $B=B^*.$

Hence, $A$ and $B$ are real along the unit circle $\lambda\in \S^1.$ 
With $X=A+iB$ and $Y=A-iB$, this
 implies the necessary 
unitarity condition that 
$X(\lambda)=\overline{Y(\bar\lambda^{-1})}$ for all $\lambda\in\S^1$.
We have to show that for all $\lambda\in \S^1$ the monodromy $(X(\lambda),Y(\lambda),Z(\lambda))$
is unitarizable.
By Theorem \ref{The:unitary-component} it is enough to show that
$(X(\lambda),Y(\lambda), Z(\lambda))\in \mathcal C_u$ for all $\lambda\in \S^1$.

We first show
that $B(\lambda_0)=0$. 
In fact, $\lambda^3B \in\mathring{\mathcal W}^\tau_\R$ implies
$
\overline{B(\lambda)}=B(\bar\lambda)
$
for all $\lambda\in\mathbb A_r.$
Using $\bar\lambda_0=-\zeta^{-2}\lambda_0$ and $\lambda^3B \in\mathring{\mathcal W}^\tau$, in particular $B$ is an odd function in $\lambda$, we therefore obtain
\begin{equation}\label{eq:Boddim}
\overline{B(\lambda_0)}
=
B(\bar\lambda_0)
=
B(\zeta^{-2}(-\lambda_0))
=
B(-\lambda_0)
=
-B(\lambda_0).
\end{equation}
Since $B$ is real along the unit circle, $B(\lambda_0)=0$ as claimed.
We obtain that $X(\lambda_0)=0=Y(\lambda_0),$ and
this implies
$(X(\lambda_0),Y(\lambda_0),Z(\lambda_0))\in\mathcal C_u$. 
For rational $t$, the image of the monodromy representation is therefore conjugate to $\Gamma$ or $\tilde\Gamma,$
and in particular finite.

We next show that $\left(X(\lambda),Y(\lambda),Z(\lambda)\right)$ remains in the unitary component $\mathcal C_u$ for all $\lambda\in\S^1.$
From Section \ref{sec:unirep} this can happen only if the discriminant \eqref{eq:discrimant-complex} 
has an odd order zero along the unit circle, or if we pass through one of the three reducible boundary points.
Since the representation depends holomorphically on $\lambda$, 
also the functions $Z=Z(\lambda)$  and $\tilde Z=\tilde Z(\lambda)$ depend holomorphically on $\lambda$.
The discriminant \eqref{eq:discrimant-complex} of the quadratic polynomial \eqref{Def-characterpolynomial} 
is given by $D(\lambda)=(Z(\lambda)-\tilde Z(\lambda))^2,$
hence its  zeros are even.

It thus remains  to establish that the monodromy cannot be reducible.
Assume the monodromy is reducible at $\lambda_1$ with $0<|\lambda_1|\leq1$. By the proof of Lemma \ref{Lem:twist-monodromy}, the monodromy at $-\lambda_1$
is conjugate to the contragredient of the monodromy at
 $\lambda_1.$ By the classification of reducible representations in Section \ref{sec:redrep}, we can therefore assume without loss of generality that the monodromy at $\lambda_1$ preserves a line.
Using parallel transport, this generates a holomorphic line subbundle $L$ of the trivial holomorphic vector bundle  $\C^3\to\CP^1$ which passes through eigenlines of the residues
of $d+\eta_{t,a,b,c}(\lambda_1).$
The corresponding eigenvalue at $z=1$ is then $0$, while the eigenvalues at $z=0$ and $z=\infty$ satisfy  $\alpha_0 +\alpha_\infty=0.$
Furthermore, the Fuchsian system $d+\eta_{t,a,b,c}(\lambda_1)$ induces a logarithmic connection on $L$, i.e. a meromorphic connection with first-order poles
at $z=0,1,\infty$. By the generalized residue formula for logarithmic connections on a line bundle (see \cite[Lemma 3]{HeHe}), $\mathrm{deg}(L)=0.$
Hence, $L$ is a holomorphic line subbundle of degree 0, i.e. a constant line in $\C^3.$ By construction, $L$ is an eigenline for $R_0$, $R_1$ and $R_\infty.$ A direct algebraic computation shows that
this is only possible if $c(\lambda_1)=0$ or $a(\lambda_1)=0,$ which is excluded by assumption. Thus, the monodromy cannot be reducible for $\lambda$ with $0<|\lambda|\leq 1$. Consequently, the monodromy along $\S^1$ stays in the unitary component $\mathcal C_u.$ 
\end{proof}

\begin{Lem}\label{lem:technicallemmarealtwist}
Let $a,b,c\in\mathcal W^\tau_\R$, $t\in[0,\tfrac13)$. Then
\[\lambda^{-6}{\mathcal F},\lambda^{-3}{\mathcal G},{\mathcal H}\in\mathcal W^\tau_\R\quad \text{and} \quad \mathcal S\in\R.\]
\end{Lem}
\begin{proof}
That ${\mathcal H}\in\mathcal W^\tau_\R$ follows as $\mathcal W^\tau_\R$ is a Banach algebra.
By Lemma \ref{lem:ABprop},  $A,\lambda^3 B\in\mathring{\mathcal W}^\tau_\R.$ 
Thus, also $A-A^*,\lambda^{-3}(B-B^*)\in\mathring{\mathcal W}^\tau_\R$.
Taking the positive parts $(.)^+$ of $A-A^*$ and $B-B^*$ therefore yields $\lambda^{-6}\mathcal F,\lambda^{-3}\mathcal G\in\mathcal W^\tau_\R$.
It remains to show $A(\lambda_0)=\mathcal S\in\R.$
This is shown similarly to \eqref{eq:Boddim}: since A is even and $\zeta$-invariant, one obtains $\overline{A(\lambda_0)}=A(\lambda_0).$
\end{proof}

\subsection{Setting up the implicit function theorem}
We seek $\epsilon>0$ and $a,b,c\colon[0,\epsilon)\to\mathcal W^\tau_\R$  which solve
\eqref{eq:mondrom}. We use the implicit function theorem in the Banach spaces of holomorphic functions.
To do so, we first expand the monodromy of $d+\eta_{t,a,b,c}$ given by \eqref{Lawson-type-potential} at $t=0$.  In the following, primes always denote differentiation with respect to $t$, even for functions depending on additional variables.

Let $a,b,c\in\mathcal W^\tau_\R$.
Consider a solution $\Psi_t=\Psi_{t,a,b,c}$ of \eqref{eq:odes}.
At $t=0$ the potential simplifies to
\[
d+\eta_{0,a,b,c}=d+\mathrm{diag}(0,-\tfrac{1}{3},\tfrac{1}{3})\frac{dz}{z},
\]
and a parallel frame with $\Psi_0(\tfrac12)=\Id$ is given by
\begin{equation}
\label{eqPsi0}
\Psi_0(z)=\mathrm{diag}(1,(2z)^{\tfrac{1}{3}},(2z)^{-\tfrac{1}{3}}).
\end{equation}
The branch is chosen holomorphically on the disk of radius $3/4$ centered at $1$, normalized by $\Psi_0(\tfrac12)=\Id$.
Thus $M_1(0)=\Id$,  while the  monodromy of $\Psi_0$ along $\gamma_0$ is
$M_0(0)=\mathrm{diag}(1,\zeta,\zeta^{-1}).$

\begin{Lem}\label{lem10:derivmono}
We have
\begin{equation}\label{eq.m1prime}
M_1'(0)=-2\pi i \,\Psi_0(1)^{-1}\,
\left(
\begin{array}{ccc}
 0 & \frac{a}{\lambda ^3}+b & i \left(b-\frac{a}{\lambda ^3}\right) \\
 -\lambda ^3 b-a & 0 & ci \\
 i \left(a-\lambda ^3 b\right) & -ci & 0 \\
\end{array}
\right)
 \Psi_0(1).
 \end{equation}
\end{Lem}
\begin{proof}
Let $Y_t(s):=\Psi_{t,a,b,c}(\gamma_1(s))$. Then
$
\frac{d}{ds}Y_t+\gamma_1^*(\eta_{t,a,b,c})\,Y_t=0$ with
$ Y_t(0)=\Id,
$
and the monodromy along $\gamma_1$ is $M_1(t)=Y_t(1)$.
Differentiating at $t=0$ gives
\[
\frac{d}{ds}(Y_0^{-1}Y_0')
=
-\gamma_1^*(\Psi_0^{-1}\eta'_{0,a,b,c}\Psi_0).
\]
Integrating from $0$ to $1$, and using $Y_0(0)=Y_0(1)=\Id$ and $Y_0'(0)=0$, yields
\[
M_1'(0)=Y_0'(1)
=
-\int_{\gamma_1}\Psi_0^{-1}\eta'_{0,a,b,c}\Psi_0.
\]
Thus, \eqref{eq.m1prime} follows from the residue theorem.
\end{proof}

\begin{Lem}\label{lem:prep}
Let $a,b,c\in \mathcal W^\tau_\R$. Then
\begin{equation}\label{eq:X0Y0Abl0}
X_{0,a,b,c}=Y_{0,a,b,c}=0,
\qquad
 X'_{0,a,b,c}=Y'_{0,a,b,c}=0.
\end{equation}
Using the definition \eqref{wh-function}, we have at $t=0$
\begin{equation}\label{XYdoubleprimeabc}
\begin{split}
\widehat{X}_{0,a,b,c}
&=\tfrac{2 \pi ^2}{\lambda^3} \left(-i \sqrt{3} a^2+2 a b \lambda ^3-i \sqrt{3} b^2 \lambda ^6+c^2 \lambda ^3\right)\\
\widehat{ Y}_{0,a,b,c}
&=\tfrac{2 \pi ^2}{\lambda^3} \left(i \sqrt{3} a^2+2 a b \lambda ^3+i \sqrt{3} b^2 \lambda ^6+c^2 \lambda ^3\right).
\end{split}
\end{equation}
\end{Lem}
\begin{proof}
The formulas are pointwise in $\lambda$, so it suffices to regard $a$, $b$, and $c$ as complex numbers.
Consider the monodromies $M_0(t),M_1(t),M_\infty(t)$, which satisfy $M_\infty(t) M_1(t) M_0(t)=\mathrm{Id}$ for all $t$.
Furthermore, the conjugacy class of $M_0(t)$ and $M_\infty(t)$ are independent of $t$ and given by $\mathrm{diag}(1,\zeta,\zeta^{-1})$.
We have $M_0(0)=\mathrm{diag}(1,\zeta,\zeta^{-1})$ and $M_1(0)=\mathrm{Id}$. 
Hence, the hypotheses of Lemma~\ref{derivativeofmonodromy-singularID} are satisfied, and we obtain \eqref{eq:X0Y0Abl0}.
 Denote $M_1'(0)=(p_{ij})_{i,j}$. 
Using Lemma \ref{lem10:derivmono} together with $\Psi_0(1)=\mathrm{diag}(1,2^{1/3},2^{-1/3})$, we compute 
\[
p_{12}p_{21}=4\pi^2\tfrac{(a+\lambda^3 b)^2}{\lambda^3},
\qquad
p_{13}p_{31}=-4\pi^2\tfrac{(a-\lambda^3 b)^2}{\lambda^3},
\qquad
p_{23}p_{32}=-4\pi^2 c^2.
\]
Together with Lemma \ref{derivativeofmonodromy-singularID} this directly 
gives \eqref{XYdoubleprimeabc}.
\end{proof}

\subsection{The initial data}
We next choose the initial direction for the implicit-function argument. 
This should be viewed as choosing a tangent direction at the singular point corresponding
 to $t=0$ in the moduli space of Fuchsian systems. Geometrically, this choice determines the branch of solutions, and hence the resulting family of surfaces.

\begin{Lem}\label{Lem:initial-data}
The constant functions
\begin{equation}\label{abc}
\begin{split}
a_0:=\tfrac{1}{\sqrt{6}},\quad
b_0:=-\tfrac{1}{\sqrt{6}},\quad
c_0:=\pm\tfrac{1}{\sqrt{3}}
\end{split}
\end{equation}
satisfy $a_0,b_0,c_0\in\mathcal W^\tau_\R,$ and Condition (i) and (ii) of Lemma \ref{harmonic-family-form}. The corresponding Fuchsian potential 
$d+\eta_{t,a_0,b_0,c_0}$ satisfies the monodromy conditions \eqref{eq:mondrom} up to second order at $t=0:$
\begin{equation}\label{eq:initialeqsatisfied}
\begin{split}
\mathcal F_{0,a_0,b_0,c_0}&=0,\quad\mathcal G_{0,a_0,b_0,c_0}=0,\quad \mathcal H_{t,a_0,b_0,c_0}=0,\quad\mathcal S_{0,a_0,b_0,c_0}=0,\\
\mathcal F'_{0,a_0,b_0,c_0}&=0,\quad \mathcal G'_{0,a_0,b_0,c_0}=0, \quad\quad\quad\qquad\quad\qquad\;\;  \mathcal S'_{0,a_0,b_0,c_0}=0\\
\widehat{\mathcal F}_{0,a_0,b_0,c_0}&=0,\quad \widehat{\mathcal G}_{0,a_0,b_0,c_0}=0,\quad \quad\quad\qquad\quad\qquad\;\; \widehat{\mathcal S}_{0,a_0,b_0,c_0}=0\\
\end{split}
\end{equation}
\end{Lem}
\begin{proof} Clearly, $4a_0b_0-c_0^2=-1,$ independently of $t$ and $\lambda.$
Using Lemma \ref{lem:prep}, we compute $X_{0,a_0,b_0,c_0}=Y_{0,a_0,b_0,c_0}=0$ and  $X'_{0,a_0,b_0,c_0}=Y'_{0,a_0,b_0,c_0}=0.$
Furthermore,
inserting \eqref{abc} into \eqref{XYdoubleprimeabc} gives
\[
\begin{split}
\widehat X_{0,a_0,b_0,c_0}=&-\tfrac{ i \pi ^2 \left(\lambda ^6+1\right)}{\sqrt{3} \lambda ^3},\qquad
\widehat Y_{0,a_0,b_0,c_0}=\tfrac{ i \pi ^2 \left(\lambda ^6+1\right)}{\sqrt{3} \lambda ^3}.
\end{split}
\]
By inspection, \eqref{eq:initialeqsatisfied} is satisfied.
\end{proof}
The two sign choices for $c_0$ lead to the two branches of solutions produced later by the implicit function theorem.

\subsection{The solution}
We next prove the existence of potentials satisfying the intrinsic and extrinsic closing conditions.

In the following, let $a,b,c\in\mathcal W^\tau_\R$. By Lemma~\ref{lem:technicallemmarealtwist},
$
\lambda^{-6}\mathcal F,\ \lambda^{-3}\mathcal G,\ \mathcal H\in\mathcal W^\tau_\R,$ and 
$
\mathcal S\in\R.
$
Moreover, by \eqref{eq:X0Y0Abl0}, the functions
$
\widehat{\mathcal F}, \widehat{\mathcal G},$ and $\widehat{\mathcal S}
$
extend real-analytically across $t=0$. Hence, after restricting to a sufficiently small open interval $I\subset\R$ containing $0$, we obtain a well-defined real analytic map
\begin{equation}\label{eq_mopro}
\begin{split}
\mathcal M&\colon I\times\mathcal W^\tau_\R\times \mathcal W^\tau_\R\times \mathcal W^\tau_\R\to\mathcal W^\tau_\R\times \mathcal W^\tau_\R\times \mathcal W^\tau_\R\times\R;\\
&(t,a,b,c)\mapsto (\lambda^{-6}\widehat{\mathcal F}_{t,a,b,c}\,,\,\lambda^{-3}\widehat{\mathcal G}_{t,a,b,c}\,,\,\mathcal H_{t,a,b,c}\,,\, \widehat{\mathcal S}_{t,a,b,c})\,,\,
\end{split}
\end{equation}
where $X,Y$ and thus $\mathcal F,\mathcal G,\mathcal S$ are computed with respect to the connection $d+\eta_{t,a,b,c}$ as before.
Since the Riemann--Hilbert map assigning to a Fuchsian system its monodromy is real analytic,
and the involved projections $(.)^+$ and involutions $(.)^*$ are bounded linear, 
 the map $\mathcal M$ is real analytic
as a composition of real analytic maps.

\begin{Pro}\label{ProIFT} Let $r>1,$ and
let $a_0,b_0,c_0$ be as in \eqref{abc}.
 There exists $\epsilon>0$ and unique real analytic maps
\begin{equation}\label{map:abc}
a,b,c\colon [0,\epsilon)\to \mathcal W^\tau_\R
\end{equation}
with $a(0)=a_0$, $b(0)=b_0$, and $c(0)=c_0$, satisfying 
\[\mathcal F_{t,a(t),b(t),c(t)}=\mathcal G_{t,a(t),b(t),c(t)}=\mathcal H_{t,a(t),b(t),c(t)}=0\quad\text{and}\quad \mathcal S_{t,a(t),b(t),c(t)}=0\]
for all $t\in[0,\epsilon).$ Furthermore, $a(t,\lambda)\neq 0$ and $c(t,\lambda)\neq 0$ for all $|\lambda|\leq r$ and all $t\in[0,\epsilon).$
For \(t=1/k<\epsilon\) and $c_0=\tfrac{1}{\sqrt{3}}$, the monodromy $\rho_{\lambda_0}$ is conjugate to $\rho$ in \eqref{finite-group-rep1}, while $\rho_{\lambda_0}$ is conjugate to $(\rho^{-1})^T$ for 
$c_0=-\tfrac{1}{\sqrt{3}}$.
\end{Pro}
\begin{proof}
By the definition \eqref{wh-function}, together with Lemma~\ref{Lem:initial-data}, the vanishing of
$
\widehat{\mathcal F}, \widehat{\mathcal G},$ and $ \widehat{\mathcal S}
$
near $t=0$ is equivalent to the vanishing of
$
\mathcal F,\mathcal G,$ and $ \mathcal S.
$
To prove existence,  it is therefore enough  to find $a=a(t),b=b(t),c=c(t)$ such that
\[\widehat{\mathcal F}_{t,a(t),b(t),c(t)}=\widehat{\mathcal G}_{t,a(t),b(t),c(t)}={\mathcal H}_{t,a(t),b(t),c(t)}=0\quad\text{and}\quad \widehat{\mathcal S}_{t,a(t),b(t),c(t)}=0\]
for sufficiently small $t\ge0.$ By \eqref{eq:initialeqsatisfied}, we have
\[\widehat{\mathcal F}_{0,a_0,b_0,c_0}=\widehat{\mathcal G}_{0,a_0,b_0,c_0}={\mathcal H}_{0,a_0,b_0,c_0}=0\quad\text{and}\quad \widehat{\mathcal S}_{0,a_0,b_0,c_0}=0.\]
We apply the implicit function theorem for the map $\mathcal M$ in \eqref{eq_mopro} between the corresponding Banach spaces.
As observed above, $\mathcal M$ is well-defined,  Banach-space real analytic, and satisfies
\[\mathcal M(0,a_0,b_0,c_0)=0.\]
We claim that the differential at the initial point
\[d_{(0,a_0,b_0,c_0)}\mathcal M\colon\{0\}\times\mathcal W^\tau_\R\times \mathcal W^\tau_\R\times \mathcal W^\tau_\R
\to\mathcal W^\tau_\R\times \mathcal W^\tau_\R\times \mathcal W^\tau_\R\times\R\]
restricted to the $t=0$-slice is an isomorphism. We compute, 
using \eqref{XYdoubleprimeabc} and $A=\tfrac12(X+Y)$, $B=\tfrac{i}{2}(-X+Y)$
\begin{equation}\label{eq:dAdB}
\begin{split}
\tfrac{1}{4 \pi ^2}d_{(0,a_0,b_0,c_0)}\widehat{A}&=b_0 \,da +a_0\,db +c_0\,dc \\
-\tfrac{1}{4\sqrt{3} \pi ^2}d_{(0,a_0,b_0,c_0)}\widehat{B}&=\frac{1}{\lambda^3} a_0da+ \lambda ^3 b_0db. 
\end{split}
\end{equation}
Consider the Banach space isomorphisms
\begin{equation}
\begin{split}
&\mathcal I=(\mathcal I_{0},\mathcal I_{+})\colon
\mathcal W^\tau_\R\to \R\times\mathcal W^\tau_\R;\;  x+\lambda^6 f \mapsto (x,f)\\
\mathcal K=(\mathcal K_c,\mathcal K_p)&\colon \mathcal W^\tau_\R\to \R\times\mathcal W^\tau_\R;\;
g=x+(\lambda^6-1) f\mapsto (x,f).
\end{split}
\end{equation}
Here $x=g(1)$ and $f=(g-g(1))/(\lambda^6-1)$; this is well-defined
and bounded because functions in $\mathcal W^\tau$ depend on $\lambda^6$.
Decompose (whenever necessary) 
\[da=da^{(0)}+\lambda^6da^+,\quad db=db^{(0)}+\lambda^6db^+,\quad dc=dc^{(0)}+\lambda^6dc^+ .\]
via $\mathcal I(da)=(da^{(0)},da^+),\dots$. Then, using $a_0=\tfrac{1}{\sqrt{6}}=-b_0,c_0=\pm\tfrac{1}{\sqrt{3}}\in\R$ and setting $\sigma:=\pm\sqrt{2}$, we obtain
\begin{equation}\label{eq:dFdG}
\begin{split}
\delta F:=\tfrac{\sqrt{6}}{4 \pi ^2}d_{(0,a_0,b_0,c_0)}\lambda^{-6}\widehat{\mathcal F}
=&- \,da^+ +\,db^+ + \sigma\,dc^+ \\
\delta G:=-\tfrac{\sqrt{6}}{4\sqrt{3} \pi ^2}d_{(0,a_0,b_0,c_0)}\lambda^{-3}\widehat{\mathcal G}
=&da^+-da^{(0)}-db\\
\delta H:=\tfrac12\sqrt{6}d_{(0,a_0,b_0,c_0)}\mathcal H
=&2  db-2 da-\sigma dc\\
\delta S:=\tfrac{\sqrt{6}}{4 \pi ^2}d_{(0,a_0,b_0,c_0)}\widehat{\mathcal S}
=&
-\,da(\lambda_0) +\,db(\lambda_0) +\sigma\,dc(\lambda_0). 
\end{split}
\end{equation}
We split $\delta H=\delta H^{(0)}+\lambda^6 \delta H^+$ and apply $\mathcal K$ to obtain an infinite-dimensional
\begin{equation}\label{eq:dFdG2}
\begin{split}
dc^+&=-\tfrac{1}{3\sigma}(\delta H^+-2\delta F)\\
db^+
&=\mathcal K_p(-\delta G-\delta F+ \sigma\,dc^+) \\
da^+&=- \delta F+db^+ + \sigma dc^+ \\
\end{split}
\end{equation}
and a finite-dimensional part 
\begin{equation}\label{eq:dFdG3}
\begin{split}
\delta H^{(0)}=&-2 da^{(0)}+2  db^{(0)}-\sigma dc^{(0)}\\
\mathcal K_c(\delta G+\delta F-\sigma\,dc^+ )
=& -da^{(0)}-db^{(0)}\\
\delta S-da^+(\lambda_0)+db^+(\lambda_0)+\sigma dc^+(\lambda_0)
=&
-da^{(0)}+db^{(0)} +\sigma dc^{(0)},
\end{split}
\end{equation}
where we used  $\lambda_0^6=-1$ in the last equation.
Obviously, \eqref{eq:dFdG2} has a unique solution. The coefficient matrix
of the finite-dimensional system \eqref{eq:dFdG3} has determinant $6\sigma\neq0,$
 so \eqref{eq:dFdG3} has a unique solution for any $da^+,db^+,dc^+$ determined by \eqref{eq:dFdG2}. 
This shows that the differential
\[d_{(0,a_0,b_0,c_0)}\mathcal M\colon \{0\}\times\mathcal W^\tau_\R\times \mathcal W^\tau_\R\times \mathcal W^\tau_\R\to\mathcal W^\tau_\R\times \mathcal W^\tau_\R\times \mathcal W^\tau_\R\times\R\]
is invertible (as a bounded linear map).
By applying the implicit function theorem, the existence and uniqueness result for the maps \eqref{map:abc} follows.

Since $a_0$ and $c_0$ are nonzero constants, continuity of
$
t\mapsto a(t)$ and $ t\mapsto c(t)
$
in $\mathcal W^\tau_\R$ implies, after shrinking $\epsilon$ if necessary, that
$
a(t,\lambda)\neq0,
$ and $c(t,\lambda)\neq0
$
for all $|\lambda|\le1$ and all $t\in[0,\epsilon)$.

It remains to determine the $Z$-coordinate of $\rho_{\lambda_0}$. Analogously to Lemma \ref{lem:prep} and Lemma \ref{Lem:initial-data}, one computes
for the initial data $a_0=\tfrac{1}{\sqrt{6}},b_0=-\tfrac{1}{\sqrt{6}},c_0=\tfrac{1}{\sqrt{3}}$ the Taylor expansion
\begin{equation}\label{Zexpan}Z_{t,a(t),b(t),c(t)}=3-12 \pi^2t^2-8 \pi^3 i t^3+ O(t^4).\end{equation}
Since the only two solutions of the Lawton equation $P(0,0,Z)=0$ are given by $2 e^{-2 i \pi  t}+e^{4 i \pi  t}$ and
 $2 e^{2 i \pi  t}+e^{-4 i \pi  t}$  corresponding to
the finite representations $\rho$ respectively $(\rho^{-1})^T$, 
and \eqref{Zexpan} picks the first one. Hence the claim for $c_0=\tfrac{1}{\sqrt{3}}$ follows. The case
$c_0=-\tfrac{1}{\sqrt{3}}$ works analogously.
\end{proof}

\section{New minimal Lagrangian surfaces}\label{sec:construction}

In this section we combine Theorem \ref{THM:MLrecon} with the monodromy construction of Section \ref{sec:solvingIFT} to obtain new
compact minimal Lagrangian immersions of Fermat curves. We then discuss their symmetry properties and compute their areas.

\subsection{Construction}

Consider the real analytic maps
$
a,b,c\colon [0,\epsilon)\to \mathcal W^\tau_\R
$
given by Proposition \ref{ProIFT}, corresponding to one of the two choices
$
c_0=c(0)=\pm \tfrac{1}{\sqrt{3}}
$
in Lemma \ref{Lem:initial-data}. For $t\in(0,\epsilon)$ set
\[
\eta_t:=\eta_{t,a(t),b(t),c(t)}.
\]

By Proposition \ref{pro:int-extcond}, 
the monodromy representation $\rho_\lambda$ of the family $d+\eta_t$ is unitarizable for every $\lambda\in\S^1$ and irreducible for $|\lambda|\le1$.

Now let $k\in\N$ with  $\frac{1}{k}=t<\epsilon$. Proposition \ref{ProIFT} implies that 
the monodromy representation
at the Sym point
$
\lambda_0=\exp\left(\tfrac{2\pi i}{12}\right)
$
  of the Fuchsian system $d+\eta_t(\lambda_0)$ is $\rho$ in \eqref{finite-group-rep1} or $(\rho^{-1})^T$, and 
has finite image isomorphic to $\Gamma_{deck}$.
Moreover, Lemma \ref{harmonic-family-form} shows that after the gauge transformation \eqref{gauge-Lawson-DPW} the connection 
$d+\eta_t$ becomes a meromorphic $\hat\tau$-twisted DPW-potential $d+\omega_t$ whose singularities are removable on the branched cover 
corresponding to the covering monodromy $\rho_{\lambda_0}$, and whose Higgs field is immersive.
Furthermore, the monodromies of $d+\omega_t$  are conjugate to the monodromies of
$d+\eta_t(\lambda_0)$, and hence unitarizable for all $\lambda\in\S^1.$

 Therefore, all hypotheses of Theorem \ref{THM:MLrecon} are satisfied (where we use $\hat\tau$ instead of $\tau$). 
 By 
  Proposition \ref{ProIFT}, the monodromy at $\lambda_0$ is given $\rho$ defined in \eqref{finite-group-rep1} 
  or $(\rho^{-1})^T$.
  The representations generate a group 
 $\Gamma\cong \Gamma_{deck}\subset \mathrm{Aut}(\Sigma_k)$ and both have the same kernel $\Lambda\subset \pi_1(S,z_0)$.
By
 \eqref{quotientidentity},  
 the corresponding branched covering is precisely the Fermat curve covering
\[\Sigma_k\to\Sigma_k/\Gamma_{deck}\cong\CP^1.\]  
Theorem \ref{THM:MLrecon} then yields a minimal Lagrangian immersion
$
f_k^\pm\colon \Sigma_k\to \CP^2.
$
Furthermore, the immersion is equivariant with respect to the 
action of the deck group $\Gamma_{deck}$ of the Fermat covering.

\begin{The}\label{mainT}
There exists $k_0\in\N$ such that for every $k\geq k_0$ and each choice $c(0)=c_0=\pm \frac{1}{\sqrt{3}}$ in Lemma \ref{Lem:initial-data}, there exists a minimal Lagrangian immersion
\[
f_k^\pm\colon \Sigma_k\to \CP^2
\]
from the Fermat curve $\Sigma_k$ of genus $\frac12(k-1)(k-2)$. The immersion $f_k^\pm$ is $\Gamma_{deck}$-equivariant.

If $3\nmid k$, then $f_k^\pm$ 
does not factor through a nontrivial lower-genus quotient of
$\Sigma_k$. If $k=3n$, then $f_{3n}^\pm$ factors through the unbranched quotient
$
\Sigma_{3n}\to X_n:=\Sigma_{3n}/\Z_3.
$
\end{The}

\begin{proof}
The existence part follows from the preceding discussion.
To prove the factorization statement, consider the induced projective representation
\[
\underline\rho_{\lambda_0}\colon \pi_1(\CP^1\setminus\{0,1,\infty\})\to \SU_3\to \PSU_3.
\]
If $3\nmid k$, then the kernel of $\underline\rho_{\lambda_0}$ agrees with the kernel of $\rho_{\lambda_0}$, and therefore the immersion does not descend to a nontrivial quotient of $\Sigma_k$.

If $k=3n$, then the central subgroup $\Z_3\subset \SU_3$ is contained in the image of $\rho_{\lambda_0}$ and becomes trivial in $\PSU_3$. Hence the kernel of $\underline\rho_{\lambda_0}$ is strictly larger than the kernel of $\rho_{\lambda_0}$. The corresponding quotient is a fixed-point-free $\Z_3$-quotient
\begin{equation}\label{eq_S3nXn}
\Sigma_{3n}\to X_n:=\Sigma_{3n}/\Z_3.
\end{equation}
By construction, the projective extrinsic closing condition is already satisfied on $X_n$, so the immersion factors through $X_n$.
\end{proof}

\begin{Cor}\label{cor:SLag}
For every $k\geq k_0$ and each sign choice $c_0=\pm \frac{1}{\sqrt{3}}$, the immersion $f_k^\pm$ admits a special Legendrian lift
$
\hat f_k^\pm\colon \Sigma_k\to \S^5.
$
\end{Cor}
\begin{proof}
The corollary directly follows from Theorem \ref{mainT} and Theorem \ref{thm:reconstruct1}, because the projective monodromy $c$ is trivial.
\end{proof}

\begin{Rem}
The surfaces obtained in Theorem \ref{mainT} are constructed using the implicit function theorem in the parameter $t=\tfrac{1}{k}.$ If the deformation in $t$ exists until $t=\tfrac{1}{3}$, we would obtain a minimal Lagrangian immersion from $X_1=\Sigma_3/\Z_3$ of genus 1 that exhibits the same symmetries as the minimal Lagrangian Clifford torus
in $\CP^2$. The surfaces in Theorem \ref{mainT} can be interpreted as the minimal Lagrangian counterparts of the minimal Lawson surfaces $\xi_{k,k}$ in the 3-sphere \cite{La}. Furthermore, it is reasonable to anticipate that the methods employed in \cite{CHHT} can be applied to establish that the maximal interval $(0,T)$ of existence for the implicit function theorem is relatively large, implying that we might expect 
the existence of minimal Lagrangian surfaces $f_k\colon\Sigma_k\to\CP^2$ for relatively low values of $k$. 
\end{Rem}

\subsection{Symmetries}
In this subsection we write $f_k$ for either of the two families $f_k^\pm$ whenever the statement is independent of the sign choice.
The group
$\Gamma_{deck}$ of deck transformations
acts on the minimal Lagrangian immersion $f\colon \Sigma_k\to\CP^2$ via the $\SU_3$-representation $\rho^+=\rho_{\lambda_0}=\rho$ in
\eqref{finite-group-rep1} for $c_0=\tfrac{1}{\sqrt{3}}$ or by $\rho^-=(\rho^{-1})^T$  for $c_0=-\tfrac{1}{\sqrt{3}}$. 
 Specifically, for every
 $\gamma\in\Gamma_{deck},$  
\[\gamma^*f_k^\pm=\rho^\pm(\gamma^{-1})\circ f_k^\pm.\]
For $k\geq4$, the full automorphism group of the Riemann surface $\Sigma_k$ has order $6k^2,$ and is given by the semi-direct product $\text{Aut}(\Sigma_k)=(\Z_k\times \Z_k)\rtimes S_3$. 
The additional $\mathbb Z_2$-factor  corresponds to the involution $\delta\colon \Sigma_k\to\Sigma_k$ defined by \[[x:y:z]\mapsto [\exp(-\tfrac{2\pi i}{3k})z:y:\exp(\tfrac{2\pi i}{3k})x].\] 
On $\CP^1=\Sigma_k/\Gamma_{deck}$, $\delta$ corresponds to the involution $z\mapsto z^{-1}$.
Note that $\delta$ transforms the cubic differential $Q$ of the minimal Lagrangian into its negative. Consequently, it cannot be lifted to a $\SU_3$-symmetry. On the other hand, it follows from the uniqueness part of the implicit function theorem  that $\delta$ lifts to an $\mathrm{U}_3$-symmetry $\hat\delta$ of the immersion. 
 Since $\delta^2=\Id$ on $\Sigma_k$, any ambient lift $\hat\delta\in \mathrm U_3$ induces a projective transformation of order two. Hence
$
\hat\delta^2=\zeta^m\Id
$
for some $m\in\{0,1,2\}$, and therefore
$
(\det\hat\delta)^2=\det(\hat\delta^2)=1.
$
Thus $\det\hat\delta=\pm1$. The case $\det\hat\delta=1$ would imply $\hat\delta\in \SU_3$, which is impossible since 
$\delta^*Q=-Q$. Hence $\det\hat\delta=-1$.
Since a lift is unique up to multiplication by a cube root of unity, and this does not change the determinant, the determinant of the ambient lift defines a well-defined character
\begin{equation}\label{def:dete}
\varepsilon\colon \Aut(\Sigma_k)\to \{\pm1\}\cong \Z_2.
\end{equation}

\begin{The}\label{thm:pmsym}
For all sufficiently large $k$, the full holomorphic automorphism group $\Aut(\Sigma_k)$ acts on the minimal Lagrangian immersion $f_k\colon \Sigma_k\to \CP^2$, and the associated character
$
\varepsilon\colon \Aut(\Sigma_k)\to \{\pm1\}
$
satisfies
$
\ker(\varepsilon)=(\Z_k\times \Z_k)\rtimes\Z_3=\Gamma_{deck}^{op}.
$
\end{The}

\begin{proof}
By Theorem~\ref{mainT}, the immersion $f_k$ is $\Gamma_{deck}$-equivariant, hence $\Gamma_{deck}^{op}\subset \ker(\varepsilon)$. The involution $\delta$ lifts to a $\mathrm{U}_3$-symmetry with determinant $-1$, so $\delta\notin\ker(\varepsilon)$. Since
$
\Aut(\Sigma_k)=\Gamma_{deck}^{op}\sqcup \delta\Gamma_{deck}^{op},
$
the claim follows.
\end{proof}

 \subsubsection{Real involutions}
 Anti-holomorphic involutive symmetries of a minimal Lagrangian immersion $f\colon \Sigma\to \CP^2$ arise from reflections either across a totally real projective plane $\RP^2\subset \CP^2$ or across a complex projective line $\CP^1\subset \CP^2$. Such a symmetry induces an anti-holomorphic involution $\mu$ of $\Sigma$ satisfying
\[
\hat\mu\circ f=f\circ \mu.
\]
If $\gamma$ parametrizes a component of the fixed point set $\mathrm{Fix}(\mu)$, then $\gamma$ is a geodesic for the induced metric. In the totally real case the cubic differential $Q$ is real along $\gamma$, whereas in the complex-line case it is purely imaginary along
$\gamma$. Conversely, by Schwarz reflection, a geodesic arc along which $Q$ is real, respectively purely imaginary, extends to the fixed-point set of a symmetry of the corresponding type.

Consider the real involution
\[\mu\colon \Sigma_k\to\Sigma_k;\,  [x:y:z]\mapsto[\bar x:\exp(\tfrac{2\pi i}{3k})\bar y:\exp(\tfrac{4\pi i}{3k})\bar z],\]
which induces the real involution $\underline\mu$ given by
$[x,y]\mapsto [\bar y,\bar x]$
on $\CP^1=\Sigma_k/(\Z_k\times\Z_k).$
Note that the cubic differential $Q$ of the immersions constructed in Theorem \ref{mainT} is real along the fixed-point set of $\underline\mu$ because $\lambda_0^3=i$.
By the reality conditions of the potentials in Section \ref{sec:solvingIFT}, $\mu$ is induced by an extrinsic symmetry $\hat\mu$ of the minimal Lagrangian immersion $f_k$, necessarily induced by the reflection across a real plane.
Specifically,
$p_1^1=[0:1:\exp(\tfrac{\pi i}{3k})],$ $p_1^2=\sigma^{-1}(p_1^1),$ and $ p_1^3=\sigma(p_1^1)$
are fixed points of $\mu.$
Furthermore, from the holomorphic symmetries of $f_k$ we know that
$f_k(p_1^1)=[1:0:0]$, 
and similarly $f_k(p_1^2)=
[0:1:0]$
and $f_k(p_1^3)
=[0:0:1].$

Moreover,
\[\kappa\colon \Sigma_k\to\Sigma_k;\,  [x:y:z]\mapsto[\bar x:\bar z:\bar y]\]
is another real involution of $\Sigma_k$ that induces  $[x:y]\mapsto [\bar x:\bar y]$ on the quotient $\CP^1$.
It has $p_1^1$ and $q_{1,1}:=[1:1:1]$ as fixed points (both in the same component of the set of fixed points). Furthermore, $Q$ is imaginary along the fixed-point set of $\kappa$, and it lifts to a reflection $\hat\kappa$ across a complex line. By the holomorphic symmetries, $f_k(q_{1,1})=[1:1:1]$.
The two real symmetries $\mu$ and $\kappa$ are related by an element of $\Aut(\Sigma_k)$ whose character \eqref{def:dete} is $-1$.

\subsection{Area formulas}
The area of the minimal Lagrangian surfaces $f_k\colon\Sigma_k\to\CP^2$ constructed in Theorem \ref{mainT} above can be calculated using Theorem \ref{thm:area-sing}. Note that each of the three singular points $0,1,\infty\in\CP^1$ can be desingularized on a specific covering, as shown in Lemma \ref{harmonic-family-form}. We therefore only need to compute the corresponding contributions from these 3 points and multiply them by the number of preimages in $\Sigma_k$.

\subsubsection{Contribution from the poles over $z=0$}

We start with the meromorphic DPW potential \eqref{DPWpotentialomega}
$d+\omega=(d+\eta).g$. 
Here 
$\eta=R_0 \tfrac{dz}{z}+R_1\tfrac{dz}{z-1}$
with  $R_0$ and $R_1$ as in \eqref{eqR1} for $a,b,c\in\,\mathcal W^\tau_\R$ satisfying
$4ab-c^2=-1.$
For the pullback \(\pi_3^*(d+\omega)\) via
$\pi_3(w)=w^3=z,$ we can  remove the singularity at $w=0$  by applying the diagonal gauge $g_0=\text{diag}(1,w,1/w)$, see \eqref{eq:desgo} and \eqref{eq:desgoga}. Since $g_0$ is constant in $\lambda$
we  get zero contribution in \eqref{eq:en-form} from the points lying over $z=0.$
\subsubsection{Contribution from the singularities over $z=\infty$}
As in the proof of Lemma \ref{harmonic-family-form}, we pullback via $w^3=z$, and then  desingularize $d+\hat\omega=\pi_3^*(d+\omega)$ by using the gauge \eqref{eq:ginfdes}.
 A direct computation of the residue term yields
\[\tr\Res_{w=\infty}( g_\infty^{(1)}(g_\infty^{(0)})^{-1}\hat\omega^{({-1})} )=1-3  c^{(0)} t\]
where  $c^{(0)}=c(\lambda=0)$ as above .
Note that there are $k^2$ many points lying over 
$w=\infty.$

\subsubsection{Contribution from the singularities over $z=1$}
As in the proof of Lemma \ref{harmonic-family-form}, we pull back via $z=\pi_k(y)=y^k+1,$
and desingularize $d+\tilde\omega:=\pi_k^*(d+\omega)$ via \eqref{eq:gkdesinG}.

A direct computation shows that each of the
 the $3k$ points on $\Sigma_k$ 
lying over $z=1$ 
 contributes with
\[\tr\Res_{w=0}( g_k^{(1)}(g_k^{(0)})^{-1}\tilde\omega^{(-1)} )=2-2 c^{(0)} .\]

Using these preliminary calculations, we obtain:
\begin{The}\label{Thm:areas}
The area of the minimal Lagrangian immersion $f_k^\pm\colon\Sigma_k\to\CP^2$ of genus 
$g=\tfrac{1}{2}(k-1)(k-2)$ constructed in Theorem \ref{mainT} is 
\[\mathrm{Area}(f_k^\pm)=6\pi (1-c^{(0)})\, k,\]
where $c^{(0)}$ is the constant term in the $\lambda$ expansion of $c=c^{(0)}+\lambda^6 c^{(1)}+\dots$ in 
\eqref{eqR1}.
Furthermore, there is a (convergent) analytic expansion as $k\to\infty$ of the renormalized area  as
\[\tfrac{1}{k}\mathrm{Area}(f_k^\pm)=6 \pi\,(1\mp\sqrt{\tfrac{1}{3}}\,)+O(\tfrac{1}{k}),\]
where the sign depends  on the choice of the initial data $c_0=\pm\tfrac{1}{\sqrt{3}}$ in Lemma \ref{Lem:initial-data}.
\end{The}
\begin{proof}
We first apply the energy formula \eqref{eq:en-form} to the primitive harmonic map $\psi$ for the smooth solutions on
 $\Sigma_k$.
 
 The points over $z=0$ contribute 0.
There are $k^2$ points over $z=\infty$, each contributing $1-3c^{(0)}(t)t$,
and $3k$ points over $z=1$, each contributing $2-2c^{(0)}(t)$.
\begin{equation}
\begin{split}
\tfrac{1}{2\pi}E(\psi)&=k^2(1-3 c^{(0)}t)+3k(2-2 c^{(0)})=k^2+6k-9kc^{(0)}\\
&=2g-2+9k(1-c^{(0)}).
\end{split}
\end{equation}

Using the relation 
\eqref{eq:areaML} between the area of $f$ and the energy of the primitive harmonic map $\psi$
we obtain  
$\mathrm{Area}( f_k^\pm)=6\pi k(1-c^{(0)}).$ 
 
For the second part,  we observe that using $t=\tfrac{1}{k}$
\[\tfrac{1}{k}\mathrm{Area}( f_k^\pm)=6\pi(1-c^{(0)}(t))\] can be interpolated  through a real analytic function in $t\sim0.$
Using the initial data
$c_0=\pm\tfrac{1}{\sqrt{3}}$, the result follows.
\end{proof}
\begin{Cor}\label{cor:two-families}
For all sufficiently large $k$, the two immersions $f_k^+$ and $f_k^-$ are not congruent.
\end{Cor}

\begin{proof}
By Theorem \ref{Thm:areas}, their areas have different asymptotic expansions.
Hence their areas are different for all sufficiently large $k$, and the surfaces cannot be congruent.
\end{proof}

For the lower area choice of initial data $c_0=\tfrac{1}{\sqrt{3}}$, we show in
Theorem \ref{thm:expansionorder2} below
 that the first-order coefficients in the $t$-expansion of $a,b,c$ vanish. Consequently, we obtain an improved expansion for the lower area surfaces:
\[\tfrac{1}{k}\mathrm{Area}(f_k^+)=6 \pi\,(1-\sqrt{\tfrac{1}{3}}\,)+O(\tfrac{1}{k^2}).\]
With more computational effort, our methods can also be used to calculate the quadratic and higher-order terms in the
area expansion in $\tfrac{1}{k};
$  compare with \cite{CHHT}.

\subsection{Willmore functional estimates}

The Willmore functional for an immersion $f\colon \Sigma\to M$ into a Riemannian manifold $(M,g)$ is defined to be
\[\mathrm{W}(f)=\int_\Sigma(|H|^2+\bar K) dA\]
where $dA$ is the induced area form, $H$ is the mean curvature vector, and $\bar K$ is the sectional curvature of the tangent planes $\im(df)$ with respect to $g$. 
The functional plays an  important role in the geometry of surfaces and their applications in general relativity.

The Willmore functional has been studied in detail for Lagrangian surfaces in $\C^2$ (\cite{Mini})  and
 in $\CP^2$ (\cite{MU}).
By analogy with the Willmore conjecture in three-dimensional space, which has been proved by Marques and Neves \cite{MN},
 it is conjectured that the Willmore functional for surfaces in $\CP^2$ of genus $g>0$ is minimized by the Clifford torus
\[\{[x,y,z]\in\CP^2\mid |x|=|y|=|z|=\tfrac{1}{\sqrt{3}}\}.\]
If our existence theorem covered the case $t=\tfrac{1}{3}$,  the surface constructed in Theorem \ref{mainT} for $k=3$ 
would be (the 3-fold covering of) the Clifford torus.

\begin{The}
Let $\Sigma_k$ be the Fermat curve of genus $\tfrac{1}{2}(k-1)(k-2).$ 
Consider the holomorphic minimal immersion $F_k\colon\Sigma_k\to\CP^2$
and the minimal Lagrangian immersion $f_k=f_k^+\colon\Sigma_k\to\CP^2$ provided by Theorem \ref{mainT} for the initial choice $c_0=\tfrac{1}{\sqrt{3}}$. Then
\[\lim_{k\to\infty}\frac{\mathrm{W}(f_k)}{\mathrm{W}(F_k)}=\tfrac{1}{2}(3-\sqrt{3})<1.\]
\end{The}
\begin{proof}
For our normalization of the Fubini--Study metric, a holomorphic curve of degree $k$ has area $\pi k$.
Thus the Willmore energy of a holomorphic curve of degree $k$ is
\[\mathrm{W}(F_k)=4 \mathrm{Area}(F_k)=4\pi \deg(F_k)=4\pi k.\]
 By  \cite{MU}, the Willmore functional for minimal Lagrangian surfaces is also given by the area
$\mathrm{W}(f_k)=\mathrm{Area}(f_k).$
Consequently, the theorem follows from Theorem \ref{Thm:areas}.
\end{proof}

Let $k=3n.$
Consider the unbranched 3-fold covering $\Sigma_{3n}\to X_n$ in \eqref{eq_S3nXn}.
By Riemann--Hurwitz, the genus of $X_n$ is given by 
$g(X_n)=\tfrac{1}{2}(3n^2-3n+2)\leq\tfrac{1}{2}(3n-1)(3n-2)=g(\Sigma_{3n})$
with equality if and only if $n=1$. In that case the genus of both surfaces is 1.
The following lemma shows that there are infinitely many Fermat curves $\Sigma_{\tilde k}$ of the same genus as $X_n=\Sigma_{3n}/\Z_3$ for suitable $\tilde k,n\in \N.$
\begin{Lem}
There are infinitely many pairs
$(\tilde k,n)\in\N$ with
$3n^2-3n+2=(\tilde k-1)(\tilde k-2).$
\end{Lem}
\begin{proof}
The equation 
is equivalent to
$
(2\tilde k-3)^2-3(2n-1)^2=6.
$
This is a Pell-type equation in $\mathbb Z[\sqrt3]$. Starting from the solution
$
(\tilde k,n)=(3,1),
$
and multiplying $2\tilde k-3+(2n-1)\sqrt3$ by powers of $7+4\sqrt3$, one obtains infinitely many integer solutions. 
\end{proof}

We denote by $\underline f_n\colon X_n\to \CP^2$ the map such that $f_{3n}=\underline f_n\circ\pi$ for $\pi\colon\Sigma_{3n}\to\Sigma_{3n}/\mathbb Z_3=X_n.$
We compare the Willmore energies of these surfaces for the same (large) genus.
\begin{The}
Let $\epsilon>0$. 
There exists $N_0\in\N$ such that for all
 $(\tilde k,n)$ with $n>N_0$ and
  $3n^2-3n+2=(\tilde k-1)(\tilde k-2)$ one has
 $\mathrm{W}(f_{\tilde k})> (\sqrt{3}-\epsilon)\mathrm{W}(\underline f_n).  $
\end{The}
\begin{proof}
This follows from Theorem \ref{Thm:areas} together with the  observation that
$3\mathrm{W}(\underline f_n)=\mathrm{W}(f_{3n})$
and the relation
$\tilde k=\frac{1}{2} \left(\sqrt{12 n^2-12 n+13}+3\right)$, which implies
$\lim_{n\to\infty} \tfrac{\tilde k}{n}=\sqrt{3}.$
\end{proof}

\begin{Rem}
Note that for both minimal Lagrangian surfaces and holomorphic curves, the modified Willmore functional  $\mathrm{W}^-$ (see \cite{MU} for details)  coincides with the area. 
For $k$ large, the area of $f_k$ is approximately $6 \left(1-\sqrt{\frac{1}{3}}\right)>1$ times the area of the holomorphic  Fermat curve $F_k$. Of course, the holomorphic immersion and the minimal Lagrangian immersion
 lie in different homotopy classes.
Similarly, when $k=3n$ is large,  the area of $\underline f_n$ exceeds -- by roughly the factor 
$2\sqrt{3}\left(1-\sqrt{\frac{1}{3}}\right)>1$ -- the area of the Fermat curve $F_{\tilde k}$ of the corresponding genus. Our computer experiments suggest that this is true for all $(\tilde k,n)$
 satisfying $3n^2-3n+2=(\tilde k-1)(\tilde k-2)$, with a single exception occurring for the Clifford torus.
\end{Rem}

\section{Embeddedness of the special Legendrian surfaces}\label{sec:em}
In this section, we investigate the special Legendrian lifts (Corollary \ref{cor:SLag}) of the minimal Lagrangian surfaces constructed in Theorem \ref{mainT}. 
Topological obstructions prevent these minimal Lagrangian surfaces from being embedded; see for example \cite{Zhang}. We show, however, that the special Legendrian lifts $\hat f_k^+$ for the initial condition $c_0=\tfrac{1}{\sqrt{3}}$
are embedded minimal surfaces in $\S^5$.
We expect the second family of special Legendrian surfaces $\hat f_k^-$
to have self-intersections.

 As a byproduct of our analysis, we obtain a detailed geometric description of the shape of the surfaces $ \hat f_k:=\hat f_k^+$ for large $k$.
In the limit $k \to \infty,$ the geometry of $\hat f_k$ resembles that of the Lawson minimal surfaces $\xi_{k,k}$ in $\S^3$: these arise as desingularizations of infinitely many totally geodesic discs foliating the interior and exterior of the Clifford torus and meeting orthogonally; compare with Figure \ref{Fig:lim}. 
In our setting, the blow-up of the tangent space at an intersection point yields a minimal Lagrangian surface in $\mathbb C^2,$ analogous to the classical doubly periodic Scherk surface
for $\xi_{k,k}$; see Theorem \ref{the:scherk}.

 \subsection*{Strategy}
 We begin by outlining the overall strategy of the proof. The key observation is that the asymptotic behaviour of the DPW potential enables the computation of the limiting surface via explicit expansion of the loop group factorization in $t$. In order to obtain sufficient geometric control it is necessary to compute the DPW potential to second order in the parameter $t.$ This refinement is achieved by explicitly solving the monodromy problem of Section \ref{sec:solvingIFT}
 to one additional order, with the monodromies themselves expressed via iterated integrals; see Section \ref{ssec:mon-exp}.

In the limit $k \to \infty$ corresponding to $t\to0,$ two fundamentally distinct asymptotic regimes emerge. The first occurs on compact subsets of the thrice-punctured sphere $\Sigma_k/(\Z_k\times\Z_k)$, where each of the three singularities has order $k$. In this region, we show that the surfaces converge to Scherk-type minimal Lagrangian surfaces in the complex plane, identified with a horizontal subspace of the Hopf fibration. These limit surfaces possess infinite topological type, although their embeddedness can be established without difficulty; see Section \ref{ssec:scherk} for details.
The second regime concerns the behaviour near the three singular points on $\Sigma_k/(\Z_k\times\Z_k)$. We show in Theorem \ref{The:limit-surfacerp2}, that, in neighbourhoods of each singularity, the surfaces converge to a spherical cap of radius $\arctan(\sqrt{2})$, in precise agreement with the area asymptotics in Theorem \ref{Thm:areas}.

The principal analytical effort in establishing embeddedness lies in controlling the transition between these two asymptotic regimes, as well as the relative positioning of the spherical caps at higher order in $t.$ This analysis is carried out in Section \ref{ssec:emb}, relying on several preparatory estimates established earlier in the section, e.g., Proposition \ref{pro:uniform}, Lemma \ref{lem:sphericallemmaT}, and Theorem \ref{The:limit-surfacerp2T}.

\begin{figure}[ht]\centering
\includegraphics[width=0.5\textwidth]{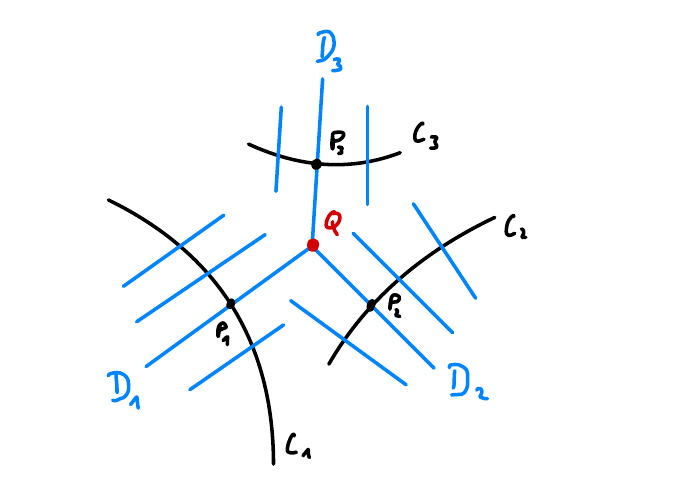}
\includegraphics[width=0.419\textwidth]{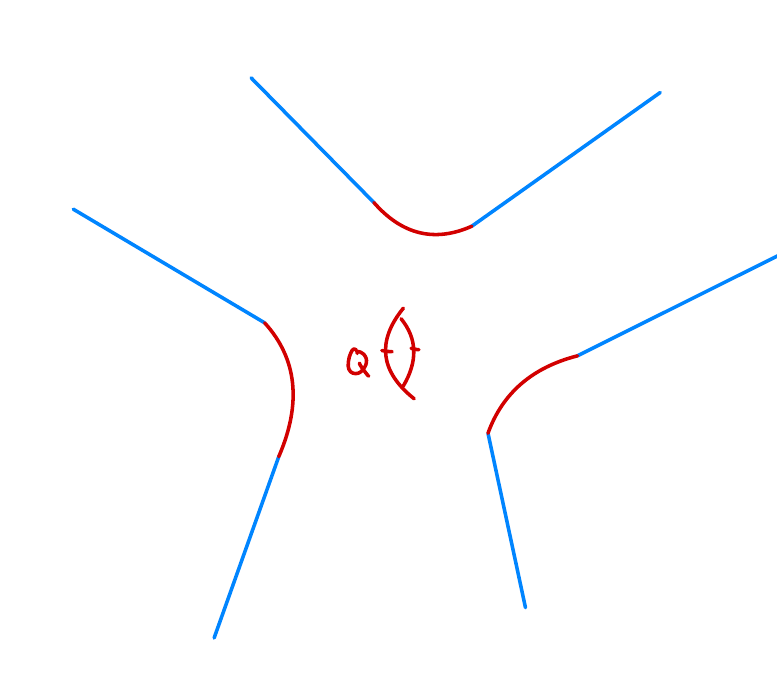}
\caption{Three {\color{blue} spherical caps} with centers $P_j$ along the coordinate circles $C_{j}$, respectively,  meet at one {\color{red} point $Q$} which lies on the Clifford torus.
At the  {\color{red} points $Q$} on the Clifford torus, the surface becomes smooth after a blow-up. It is a doubly periodic minimal Lagrangian surface in $\C^2$ analogous to the doubly periodic Scherk surface, and has three asymptotically 
{\color{blue} planar ends}.}
\label{Fig:lim}
\end{figure}

\subsection{Computation of the parameters to first order}\label{ssec:mon-exp}
The special Legendrian surface $\hat f_k$ is constructed via loop group factorization from a twisted potential $d+\omega$. The potential $\omega$ itself is determined by $t=\tfrac{1}{k}$ and $a(t),b(t),c(t)\in \mathcal{W}^\tau_\R$, where $a(t),b(t)$ and $c(t)$ are obtained in Proposition \ref{ProIFT}
via the implicit function theorem. 
In particular, the functions $a,b,c\colon [0,\epsilon)\to\mathcal W^\tau_\R$ are real analytic in $t$, and their values at $t=0$ are
$a_0=\tfrac{1}{\sqrt{6}}, $ $b_0=-\tfrac{1}{\sqrt{6}},$ and $c_0=\tfrac{1}{\sqrt{3}}.$
We now compute their first-order derivatives at $t=0$, by expanding the solution given by 
the implicit function.
We do the computations on the 3-fold covering 
\[\pi_3\colon \CP^1\to \CP^1, z\mapsto z^3.\]
We start with the potential $d+\eta:=d+\eta_{t,a_0,b_0,c_0}$ defined in \eqref{Lawson-type-potential}, i.e.
we do not use the potential $d+\eta_{t,a(t),b(t),c(t)}$ at the moment.
The pullback of $d + \eta$ by 
$\pi_3$ followed by the desingularization gauge
$\tilde g=\mathrm{diag}(1,z,z^{-1})$
yields
\[\pi_3^*(d+\eta).\tilde g=d+t\,(Q_1\frac{dz}{z-1}+Q_2\frac{dz}{z-\zeta}+Q_3\frac{dz}{z-\zeta^{2}}),\]
where as before we have 
$\zeta=\exp(\tfrac{2\pi i}{3}) $. Setting $\hat\zeta=\text{diag}(1,\zeta,\zeta^2)$ we have
\begin{equation*}
Q_1= 
\begin{psmallmatrix}
 0 & \frac{p }{\lambda ^3} & \frac{q }{\lambda ^3} \\
 -p  & 0 & ic  \\
 -q  & -ic  & 0 \\
\end{psmallmatrix},
\qquad Q_{2}= \hat\zeta^{-1}Q_1\hat\zeta \, , \qquad Q_3=\hat\zeta^{-1}Q_2\hat\zeta
\\
\end{equation*}
for
$p=\tfrac{1}{\sqrt{6}}(1-\lambda^3)$
and
$q=-i\, \tfrac{1}{\sqrt{6}} (1+\lambda^3).$

For computational convenience, we apply another coordinate transformation. Set 
\begin{equation*}\label{eq:nondpw}\hat\eta:=-Q_1\tfrac{dz}{1-z}+Q_3\tfrac{dz}{z}\end{equation*}
with residue $Q_2$ at $z=\infty.$ Thus,
the connection $d+t\hat \eta$ is the pullback of the $\Z_3$-equivariant connection
\[d+t\,\sum_jQ_j \tfrac{dz}{z-\zeta^{j-1}}\]
under the M{\"o}bius transformation 
\begin{equation}\label{mobius}
z\mapsto\tfrac{2 i z-i+\sqrt{3}}{\left(\sqrt{3}-i\right) z+2 i}
\end{equation}
which maps
$0$ to $\zeta^2$, $1$ to $1$, and $\infty$ to $\zeta.$
Our new basepoint $q_0$ is therefore
\[q_0=\tfrac{1}{2} \left(1+i \sqrt{3}\right),\]
which is mapped under the M{\"o}bius transformation to $0$.
All logarithms below use the principal branch (real on $(0,\infty)$), and all dilogarithms $\mathrm{Li}_2$
use the standard branch with cut $[1,\infty)$; this fixes the monodromy signs unambiguously.

In order to determine the first-order terms in the expansion of
the real analytic functions $a,b,c$ at $t=0$
we first consider a solution
\[\tilde\Psi=\tilde\Psi_0+ \tilde\Psi_1t+\tilde\Psi_2t^2+\dots\]
of the linear ODE
\[d\tilde\Psi+ t \hat\eta\tilde\Psi=0\qquad\text{
with initial condition}
\qquad\tilde\Psi(\tfrac{1}{2})=\Id.\]
The choice of the basepoint $\frac{1}{2}$ is made in order to apply well-known results about the monodromies of iterated integrals of the (poly)logarithm; see \cite{Hain}. To preempt any possible confusion, we remark this is neither the same as nor related to the basepoint $\frac{1}{2}$ on the furthest
quotient thrice-punctured sphere where the relative character variety was taken.

Expanding the linear ODE in $t$ gives
\begin{equation*}
\begin{split}
0&=d\tilde\Psi_0,\qquad
0=d\tilde\Psi_1+\hat\eta\tilde\Psi_0,\qquad
0=d\tilde\Psi_2+\hat\eta\tilde\Psi_1,\qquad
\dots
\end{split}
\end{equation*}
Hence,
$\tilde\Psi_0=\Id,$
and using the principal branch of the logarithm, which is real on the positive real axis,
\begin{equation}\label{eq:tildePsi1}\tilde\Psi_1=-\int_{\tfrac{1}{2}}^z\hat\eta=-Q_1\log(1-z)-Q_3\log(z)-(Q_1+Q_3)\log(2)\end{equation}
as well as
\begin{equation*}
\begin{split}
\tilde\Psi_2=&-\int_{\tfrac{1}{2}}^z \hat\eta\tilde\Psi_1\\
=&\int_{\tfrac{1}{2}}^z (-Q_1\tfrac{dz}{1-z}+Q_3\tfrac{dz}{z})(Q_1\log(1-z)+Q_3\log(z))+\int_{\tfrac{1}{2}}^z (-Q_1\tfrac{dz}{1-z}+Q_3\tfrac{dz}{z})(Q_1+Q_3)\log(2)\\
=&\tfrac{1}{2}Q_1^2\log^2(1-z)-Q_3Q_1\Li2(z)-Q_1Q_3\Li2(1-z) +\tfrac{1}{2}Q_3^2\log^2(z)+\tfrac{1}{2}(Q_1^2+Q_3^2)\log^2(2)\\&-(Q_1Q_3+Q_3Q_1)\Li2(\tfrac{1}{2})+(Q_1\log(1-z)+Q_3 \log(z))(Q_1+Q_3)\log(2)\\&-(Q_1+Q_3 )(Q_1+Q_3)\log^2(2)\\
\end{split}
\end{equation*}
where \[\Li2(\tfrac{1}{2})=\tfrac{1}{12} \left(\pi ^2-6 \log ^2(2)\right)\qquad\text{
and}\qquad\li_1(z)=-\log(1-z),\quad d\li_k(z)=\tfrac{dz}{z}\li_{k-1}(z).\]

The monodromies along the loop 
\[\tilde\gamma_1\colon [0,1]\to\C\setminus\{0,1\};\,s\mapsto 1-\tfrac{1}{2}\exp(2\pi i s)\]
based at $\frac{1}{2}$ of the involved functions are given by (see e.g. \cite{Hain})
\begin{equation*}
\begin{split}
\tilde\gamma_1^*\log(z)&=\log(z)\\
\tilde\gamma_1^*\log^2(z)&=\log^2(z)\\
\tilde\gamma_1^*\Li2(1-z)&=\Li2(1-z)\\
\tilde\gamma_1^*\log(1-z)&=\log(1-z)+ 2 \pi i\\
\tilde\gamma_1^*\log^2(1-z)&=\log^2(1-z)+4 \pi i \log(1-z)-4 \pi^2\\
\tilde\gamma_1^*\Li2(z)&=\Li2(z)-2 \pi i\log(z).\\
\end{split}
\end{equation*}
From  this we obtain
\begin{equation*}
\begin{split}
\tilde\gamma_1^*\tilde\Psi_1=\tilde \Psi_1-2\pi iQ_1
\end{split}
\end{equation*}
and
\begin{equation*}
\begin{split}
\tilde\gamma_1^*\tilde\Psi_2=\dots
=&\tilde\Psi_{2}+\tilde\Psi_1(-2\pi iQ_1)-2 \pi^2 Q_1^2+ 2\pi i[Q_1,Q_3]\log(2).\\
\end{split}
\end{equation*}
Therefore 
\begin{equation*}
\begin{split}
\tilde\gamma_1^*\tilde\Psi=&\tilde\Psi(1-2 \pi iQ_1 t+(-2 \pi^2 Q_1^2+ 2\pi i[Q_1,Q_3]\log(2))t^2+\dots)
\end{split}
\end{equation*}
and hence the expansion of the monodromy along $\tilde\gamma_1$ based at $\tfrac{1}{2}$ is
\[M(\tilde\gamma_1)=1-2 \pi iQ_1 t+(-2 \pi^2 Q_1^2+ 2\pi i[Q_1,Q_3]\log(2))t^2+\dots ).\]

To compute the monodromy based at $q_0$,
we expand the parallel transport
from $\tfrac{1}{2}$ to $q_0$, i.e., we consider the end point of the parallel transport $P$ along the vertical line from $\tfrac{1}{2}$ to
$q_0=\tfrac{1+i\sqrt{3}}{2}$:
\begin{equation*}
\begin{split}
P:=&\tilde\Psi(q_0)
=\Id+ t\tilde\Psi_1(q_0)+t^2\tilde\Psi_2(q_0)+\dots.
\end{split}
\end{equation*}

Thus
\[\Psi=\tilde\Psi P^{-1}\]
is a solution of
$d\Psi+\hat\eta\Psi=0$
with initial value
$\Psi(q_0)=\Id,$
and we obtain for the monodromy along $\tilde\gamma_1$ (based at $q_0$ and winding around the singular point $z_0=1$):
\begin{equation*}
\begin{split}
M(\tilde\gamma_1)=&PM(\tilde\gamma_1) P^{-1}\\
=&(\Id+t\tilde\Psi_1(q_0)+\dots)(\Id-2 \pi i Q_1 t+t^2(-2 \pi^2 Q_1^2+ 2\pi i[Q_1,Q_3]\log(2))+\dots)(\Id-t\tilde\Psi_1(q_0)+\dots)\\
=&\Id- 2\pi i Q_1 t+t^2(-2 \pi^2 Q_1^2+ 2\pi i[Q_1,Q_3]\log(2)- 2\pi i [\tilde\Psi_1(q_0),Q_1])+O(t^3)
\end{split}
\end{equation*}
where we have by \eqref{eq:tildePsi1} that
\[\tilde\Psi_1(q_0)=(Q_1-Q_3)\tfrac{\pi i}{3}-(Q_1+Q_3)\log(2).\]

Plugging in the values of $p,q,c$ at $t=0$, we obtain the expansion
\begin{equation*}\label{eq:mon-expq0}
\begin{split}
M_t(\tilde\gamma_1)^{-1}\frac{d}{dt}M_t(\tilde\gamma_1)=& 
\begin{psmallmatrix}
 0 & \frac{i \sqrt{\frac{2}{3}} \pi  \left(\lambda ^3-1\right)}{\lambda ^3} & -\frac{\sqrt{\frac{2}{3}} \pi  \left(\lambda ^3+1\right)}{\lambda ^3} \\
 -i \sqrt{\frac{2}{3}} \pi  \left(\lambda ^3-1\right) & 0 & \frac{2 \pi }{\sqrt{3}} \\
 \sqrt{\frac{2}{3}} \pi  \left(\lambda ^3+1\right) & -\frac{2 \pi }{\sqrt{3}} & 0 \\
\end{psmallmatrix}\\
&+ t\begin{psmallmatrix}
 -\frac{4 i \pi ^2 \left(\lambda\,^6+1\right)}{3 \sqrt{3} \lambda ^3} & 0 & 0 \\
 0 & \frac{2 i \pi ^2 \left(\lambda\,^6+1\right)}{3 \sqrt{3} \lambda ^3} & 0 \\
 0 & 0 & \frac{2 i \pi ^2 \left(\lambda\,^6+1\right)}{3 \sqrt{3} \lambda ^3} \\
\end{psmallmatrix}\;+ O(t^2),
\end{split}
\end{equation*}
which is unitary to second order, and satisfies the extrinsic closing condition to second order.
From these computations, we are able to deduce the following theorem.
\begin{The}\label{thm:expansionorder2}
The linear terms of the expansion in $t$ of $a,b,c$ of the monodromy problem vanish, i.e., $a'(0)=b'(0)=c'(0)=0$. The positive unitarizer 
$R\in {\bf\SL}(\mathbb{D})$ at $q_0$ can be chosen to be diagonal with positive real entries at $\lambda=0.$ Moreover, $R(t)$ is uniquely
determined by these properties,
and its Taylor expansion at $t=0$ is
$R(t)=\Id+O(t^2).$
\end{The}
\begin{proof}
We first prove the uniqueness part. Assume $\tilde R=\tilde R(t)$ is a second unitarizer with the same properties.
We use the fact that the monodromy is irreducible
for generic $\lambda$ with $0<|\lambda|\le 1$. This can be shown similarly to the proof of 
Proposition \ref{pro:int-extcond}. We omit the details. (As opposed to $\Sigma_k/\Gamma_{deck}$, the monodromy on $\Sigma_k/(\Z_k\times\Z_k)$ is reducible for some spectral parameters, for example
for the Sym point $\lambda_0.$) 
In particular, the monodromy representation is unitary for generic $\lambda\in\S^1.$
Thus, $(R^{-1}\tilde R)(\lambda)$
is unitary for generic $\lambda\in \S^1$, and hence for all $\lambda\in\S^1.$ Since $R,\tilde R\in{\bf\SL}(\mathbb{D})$ are both
normalized, the claim follows from the uniqueness of the Iwasawa decomposition.

Recall that $d+t \hat\eta$ is the desingularization by $\tilde g$ of the pullback of the potential $d+\eta_{t,a_0,b_0,c_0}$.
Consider also the desingularization by $\tilde g$ of the pullback of the potential $d+\eta_{t,a(t),b(t),c(t)}$ constructed by the implicit function theorem.
After conjugation with its unique positive unitarizer the latter potential has unitary monodromy at the base point $q_0$ (corresponding to $z=0$ after the M\"obius transformation). 
By $\Z_3$-equivariance with respect to $z\mapsto\zeta z$ 
and by uniqueness of the positive unitarizer, $R$ commutes with $\hat\zeta$; hence $R$ is diagonal.
Its expansion in $t$ is given by \[R(t)=\Id +t \, \mathrm{diag}(d_1,d_2,-d_1-d_2)+O(t^2)\]
where $d_1,d_2$ are holomorphic functions in $\lambda^3$ which evaluate to real positive values at $\lambda=0.$
As in Lemma \ref{lem10:derivmono} we  compute the contributions of $a'(0),b'(0),c'(0)$ 
in the second order term of the expansion of the monodromy. 
We will use the coordinate transformation
\[a=\tfrac{1}{2} (p+i q)\quad\text{and}\quad b= -\tfrac{-p+i q}{2 \lambda ^3}\]
and denote the first-order derivatives in $t$ at $t=0$ by $\dot p$ and $\dot  q$.
Then, $\dot p,\dot q$ and $\dot c$ are holomorphic functions in $\lambda^3$.
After conjugation with $R(t)$, the expansion of the monodromy $\tilde M$ of $d+\eta_{t,a(t),b(t),c(t)}$  then satisfies
\begin{equation}\label{eq:mon-expq0primes}
\begin{split}
&R^{-1}(t)M_t(\gamma_1)^{-1}\frac{d}{dt}M_t(\gamma_1)R(t)\\
=&
\begin{psmallmatrix}
 0 & \frac{i \sqrt{\frac{2}{3}} \pi  \left(\lambda ^3-1\right)}{\lambda ^3} & -\frac{\sqrt{\frac{2}{3}} \pi  \left(\lambda ^3+1\right)}{\lambda ^3} \\
 -i \sqrt{\frac{2}{3}} \pi  \left(\lambda ^3-1\right) & 0 & \frac{2 \pi }{\sqrt{3}} \\
 \sqrt{\frac{2}{3}} \pi  \left(\lambda ^3+1\right) & -\frac{2 \pi }{\sqrt{3}} & 0 \\
\end{psmallmatrix}
+ t\,
\begin{psmallmatrix}
 -\frac{4 i \pi ^2 \left(\lambda\,^6+1\right)}{3 \sqrt{3} \lambda ^3} & 0 & 0 \\
 0 & \frac{2 i \pi ^2 \left(\lambda\,^6+1\right)}{3 \sqrt{3} \lambda ^3} & 0 \\
 0 & 0 & \frac{2 i \pi ^2 \left(\lambda\,^6+1\right)}{3 \sqrt{3} \lambda ^3} \\
\end{psmallmatrix}\\
&
-4\pi\, t\,
\begin{psmallmatrix}
 0 & \frac{i \left(\sqrt{6} \left(\lambda ^3-1\right) (d_1-d_2)+6 \dot p\right)}{6 \lambda ^3} & \frac{-\sqrt{6} \left(\lambda ^3+1\right) (2 d_{1}+d_{2})+6 i \dot q}{6 \lambda ^3} \\
 \frac{i \left(\lambda ^3-1\right) (d_{1}-d_{2})}{\sqrt{6}}-i \dot p & 0 & \frac{d_{1}+2 d_{2}}{\sqrt{3}}-\dot c \\
 -\frac{\left(\lambda ^3+1\right) (2 d_{1}+d_{2})}{\sqrt{6}}-i \dot q & \frac{d_{1}+2 d_{2}}{\sqrt{3}}+ \dot c & 0 \\
\end{psmallmatrix}
\;+ O(t^2)
\end{split}
\end{equation}
The diagonal part is already in ${\bf su}_3^{\hat\tau}(\S^1).$ Since $\dot c$ and $d_1$ and $d_2$ are holomorphic functions on an open neighbourhood of 
the unit disc, one can easily deduce that $d_1+2d_2$ and $\dot c$ must be constants in order for \eqref{eq:mon-expq0primes} to be ${\bf su}_3^{\hat\tau}(\S^1)$-valued. Furthermore, we see that $d_1+2d_2$ is imaginary and $\dot c$ must be real, and therefore  $d_1+2d_2=0$
(as $d_1(0),d_2(0)\in\R$  by assumption). A similar, but slightly more involved derivation shows
that if \eqref{eq:mon-expq0primes} is ${\bf su}_3^{\hat\tau}(\S^1)$-valued then
\begin{equation*}
\begin{split}
\dot p&=p_1 (\lambda^3-1),\qquad
\dot q=q_1 (\lambda^3+1),\quad\text{and}\quad
d_1=d_2=0.
\end{split}
\end{equation*}
In particular, $R(t)=\mathrm{Id}+O(t^2).$ By \eqref{finite-group-rep1}, the monodromy on $\Sigma_k/(\Z_k\times\Z_k)$ at the Sym point $\lambda_0=\exp(\tfrac{2\pi i}{12})$ is diagonal up to conjugation. Hence,
the three residues of the non-resonant Fuchsian system at the Sym point commute. This then implies
$\dot q=-i\,\dot p,$  
and
$ \dot c= \sqrt{2} \dot p.$
Expanding the eigenvalues of the residues of $d+\eta_{t,a(t),b(t),c(t)}$ in $t$ then gives $\dot p=0,$ as claimed.
\end{proof}

\subsection{Limit construction of the surfaces}
To prove embeddedness, we need precise control of the surfaces for large $k$.
The integrable-systems setup allows us to compute their limit explicitly.
We carry out the  main computations on the $3$-fold cover
$
\pi_3:\CP^1\to\CP^1,\; z\mapsto z^3,
$
which is branched over $0$ and $\infty$. After pullback and desingularization by $\tilde g$, the three singular points 
of the potential are located at $1,\zeta,\zeta^2$.
We consider the $k^2$-fold branched covering $\pi_k\colon\Sigma_k\to\CP^1=\Sigma_k/(\Z_k\times\Z_k),$ and the induced
topological covering  (compare with \eqref{uni-topo-cov})
\[\widetilde\pi_k\colon \widetilde S:=\widetilde{\CP^1\setminus\{1,\zeta,\zeta^2\}}\longrightarrow \Sigma_k\setminus \pi_k^{-1}\{1,\zeta,\zeta^2\}.\]
The composition of  $\hat f_k\circ \widetilde \pi_k$  is well-defined on $\widetilde S$ for all $k\geq k_0.$
Furthermore, we can smoothly interpolate between those surfaces on $\widetilde S$:
Consider
\[\tilde g=\mathrm{diag}(1,z,z^{-1})\qquad \text{and}\qquad h=
\begin{psmallmatrix}
 \frac{1}{\lambda } & 0 & 0 \\
 0 & \lambda  & i z \\
 0 & 0 & 1 \\
\end{psmallmatrix}
\]
and the $\hat\tau$-twisted potential
\begin{equation*}d+\omega_t:=(d+\pi_3^*\eta_{t,a(t),b(t),c(t)}).\tilde g h\,.
\end{equation*} 
Note that
\[
\tilde gh=
\begin{psmallmatrix}
 \frac{1}{\lambda } & 0 & 0 \\
 0 & \lambda  & i z^3 \\
 0 & 0 & 1 \\
\end{psmallmatrix}
\begin{psmallmatrix}
 1 & 0 & 0 \\
 0 & z & 0 \\
 0 & 0 & \frac{1}{z} \\
\end{psmallmatrix},
\]
so we work with the pull-back of the $\hat\tau$-twisted potential in Lemma \ref{harmonic-family-form}, see \eqref{pi3om}.
Note  that on $\CP^1\setminus\{1,\zeta,\zeta^2\},$
the base point $q_0$  in Theorem~\ref{thm:expansionorder2} corresponds to $z=0$.
\begin{Lem}\label{Lem:tau-twist-unitaizer}
For each $t\in[0,\epsilon)$, the normalized positive diagonal unitarizer  $R(t)$  
from Theorem~\ref{thm:expansionorder2} 
at the base point $z=0$ 
is $\hat\tau$-twisted.
\end{Lem}
\begin{proof}
Since
$
h(0)=\diag(\lambda^{-1},\lambda,1)
$
is diagonal and unitary, the same loop $R(t)$ also unitarizes the monodromy of $d+\omega_t$ at $z=0$.

Let $M_t(\lambda)$ denote the monodromy of $d+\omega_t$ at the base point $z=0$. Because $d+\omega_t$ is $\hat\tau$-twisted, its monodromy satisfies
\[
M_t(\xi\lambda)=\hat\tau\bigl(M_t(\lambda)\bigr).
\]
Hence $\hat\tau(R(t,\lambda))$ and $R(t,\xi\lambda)$ are both normalized positive unitarizers of $M_t(\xi\lambda)$. By uniqueness of the normalized positive unitarizer, they coincide:
\[
\hat\tau(R(t,\lambda))=R(t,\xi\lambda).
\]
Therefore $R(t)\in {\bf SL}^{\hat\tau}(\mathbb D)$.
\end{proof}

Now let $\Omega_t$ be the parallel frame of $d+\omega_t$ on $\widetilde S$ normalized by
\[
\Omega_t(z=0)=R(t).
\]
Since both the potential $d+\omega_t$ and the initial value $R(t)$ are $\hat\tau$-twisted, uniqueness of solutions of the ODE implies that $\Omega_t$ is $\hat\tau$-twisted as well. We then apply the $\hat\tau$-twisted Iwasawa decomposition and write
\[
\Omega_t=B_tF_t,
\qquad
B_t\in {\bf SL}^{\hat\tau}_{B}(\mathbb D),\quad
F_t\in {\bf SU}^{\hat\tau}(\S^1).
\]

By Remark~\ref{Rem:sllift}, the map
\begin{equation}\label{t-reconstruction}
\hat f_t:=F_t(\lambda_0)^{-1}e_1\colon \widetilde S\to \S^5
\end{equation}
is a special Legendrian immersion. For $t=\frac1k$, Corollary~\ref{recon-formula} yields
\begin{equation}\label{t-reconstruction0}
\hat f_t=\hat f_k\circ \widetilde\pi_k
\end{equation}
up to an ambient $\SU_3$-transformation.

\subsection{The Scherk-type limit surface}\label{ssec:scherk}
As a first consequence of Theorem \ref{thm:expansionorder2}, we determine the limit surface for $k\to\infty$ on 
compact subsets of the thrice-punctured sphere. 

We record the following observation: Since $\hat f_k$ is a horizontal lift of $f_k\colon\Sigma_k\to\CP^2$ through the Hopf fibration $\S^5\to\CP^2$, we have
$d_p\hat f_k(T_p\Sigma_k)\subset \mathcal H_q=q^\perp\subset T_q\S^5. $

\begin{The}\label{the:scherk}
Let $\mathbb K\subset \CP^1\setminus\{1,\zeta,\zeta^2\}$ be simply connected and compact with $0\in \mathbb K.$
Consider $\mathbb K$ as a subset of the universal covering $\widetilde S$ by fixing a lift of $0.$
On $\mathbb K$, we have uniform convergence
\[
d\hat f_t\to 0 \qquad (t\to0).
\]
Since the family is normalized at $z=0$, it follows that $\hat f_t$ converges uniformly on $\mathbb K$ to a constant map with value $q\in \S^5$.
Denoting the derivative with respect to $t$ at $t=0$ by $\dot{\hat f}$, we then  have
\[d\dot{\hat{f}}=\sqrt{\tfrac{3}{2}}\begin{pmatrix}e^{-\tfrac{\pi i}{3}} z \\ -e^{\tfrac{\pi i}{3}}\end{pmatrix}
\frac{dz}{z^3-1}
+\sqrt{\tfrac{3}{2}}\begin{pmatrix}e^{\tfrac{\pi i}{6}} \\ -e^{-\tfrac{\pi i}{6}} \bar z\end{pmatrix}
\frac{d\bar z}{\bar z^3-1} \in\Omega^1({\mathbb K},\C^2)\]
where we identify $\mathcal H_q=\C^2$ with its standard Hermitian structure.
\end{The}
\begin{proof}
The proof is a direct computation, analogous to the CMC and minimal-surface cases in $\mathbb R^3$ and $\S^3$ treated in \cite{Traizet,CHHT}.
We work with the continuous parameter $t=\tfrac{1}{k}$.
Recall that the surface $\hat f_{t}$ is constructed from  the $\hat\tau$-twisted potential
\begin{equation*}d+\omega_t:=(d+\pi_3^*\eta_{t,a(t),b(t),c(t)}).\tilde g h.
\end{equation*} 
We  first need to expand the solution of
$(d+\omega_t)\Omega_t=0$
with diagonal initial condition 
\[\Omega_t(z=0)=R(t)=\Id+O(t^{2}).\] 
Then we split 
$\Omega_t=B_tF_t$
where $B_t$ is positive and normalized at $\lambda=0$ and $F_t$ is unitary along $\lambda\in\S^1$.
To understand the limiting surface, we use \eqref{t-reconstruction} and expand $F_t$ in time $t$ .

By construction, the monodromy of $\Omega_t$ with base point $z=0$ is unitary.
Since the Iwasawa decomposition is unique, the positive part
 $B_t$ has trivial monodromy on the surface, both on the thrice-punctured sphere and on the covering $\Sigma_k$  when $t=\tfrac{1}{k}$.
By Theorem \ref{thm:expansionorder2}, the Taylor expansion of $\omega_t$ is given by
\[\omega_t=
\begin{psmallmatrix}
 0 & 0 & 0 \\
 0 & 0 & \frac{i}{\lambda } \\
 0 & 0 & 0 \\
\end{psmallmatrix}dz+
\begin{psmallmatrix}
 0 & -\frac{\sqrt{\frac{3}{2}} \left(\lambda ^3-1\right)}{\lambda  \left(z^3-1\right)} & -\frac{i \sqrt{6} \lambda  z}{z^3-1} \\
 \frac{\sqrt{6} \lambda  z}{z^3-1} & -\frac{\sqrt{3} z^2}{z^3-1} & -\frac{i \sqrt{3}}{\lambda } \\
 \frac{i \sqrt{\frac{3}{2}} \left(\lambda ^3+1\right)}{\lambda  \left(z^3-1\right)} & -\frac{i \sqrt{3} \lambda  z}{z^3-1} & \frac{\sqrt{3} z^2}{z^3-1} \\
\end{psmallmatrix}dz\, t+O(t^3).\]
With this, we expand as before
\[\Omega_t=\begin{psmallmatrix}1& 0&0\\ 0& 1&-\tfrac{i}{\lambda}z\\0&0&1\end{psmallmatrix}(\Id+t \xi_1+O(t^{2}))\qquad\text{with}\qquad
\xi_1=
\begin{pmatrix}
 0 &y^T \\
 x & 
*  \\
\end{pmatrix}
\]
where
\[y^T={\tiny\begin{pmatrix}\tfrac{\left(\lambda ^3-1\right) \left(\sqrt{3} \left(2 \log (1-z)-\log \left(z^2+z+1\right)\right)-6 \tan ^{-1}\left(\tfrac{2 z+1}{\sqrt{3}}\right)+\pi \right)}{6 \sqrt{2} \lambda },-\tfrac{i \left(\lambda ^3+1\right) \left(\sqrt{3} \log \left(z^2+z+1\right)-2 \sqrt{3} \log (1-z)-6 \tan ^{-1}\left(\tfrac{2 z+1}{\sqrt{3}}\right)+\pi \right)}{6 \sqrt{2} \lambda ^2}\end{pmatrix}}\]
and
\[x=
\begin{pmatrix}   
\frac{\left(\lambda ^3-1\right) \left(\sqrt{3} \log \left(z^2+z+1\right)-2 \sqrt{3} \log (1-z)-6 \tan ^{-1}\left(\frac{2 z+1}{\sqrt{3}}\right)+\pi \right)}{6 \sqrt{2} \lambda ^2}\\
-\frac{i \left(\lambda ^3+1\right) \left(\sqrt{3} \left(2 \log (1-z)-\log \left(z^2+z+1\right)\right)-6 \tan ^{-1}\left(\frac{2 z+1}{\sqrt{3}}\right)+\pi \right)}{6 \sqrt{2} \lambda }
\end{pmatrix}.
\]
Consider the expansion into positive and unitary parts
\begin{equation}\label{eq:texpposun}
\Omega_t=\Omega_0(1+t\xi_1+\dots)=B_tF_t=B_0(1+t b_1+\dots)F_0(1+t X_1+\dots)\end{equation}
with explicit factorization at $t=0$
\begin{equation}\label{eq:fact=0}
\begin{split}
\Omega_0&=\begin{psmallmatrix} 1& 0&0\\ 0& 1&-i z\lambda^{-1}\\0&0&1\end{psmallmatrix}=B_0F_0=
\begin{psmallmatrix}
 1 & 0 & 0 \\
 0 & \sqrt{z \bar z+1} & 0 \\
 0 & \frac{i \lambda  \bar z}{\sqrt{z \bar z+1}} & \frac{1}{\sqrt{z \bar z+1}} \\
\end{psmallmatrix}\,
\begin{psmallmatrix}
 1 & 0 & 0 \\
 0 & \frac{1}{\sqrt{z \bar z+1}} & -\frac{i z}{  \sqrt{z \bar z+1}} \lambda^{-1}\\
 0 & -\frac{i  \bar z}{\sqrt{z \bar z+1}} \lambda& \frac{1}{\sqrt{z \bar z+1}} \\
\end{psmallmatrix}.
\end{split}
\end{equation}
We therefore obtain from
\eqref{eq:texpposun}
that
\[\xi_1=F_0^{-1}b_1F_0+X_1,\]
where
$b_1$ and $X_1$ are uniquely determined by the property of being elements in the positive and in the unitary loop algebra, respectively. Solving the linear system yields
\[b_1=
\begin{pmatrix}
 0 & \beta_1 & \beta_2 \\
 \alpha_1 & * & * \\
 \alpha_2 & * & * \\
\end{pmatrix}\]
with
\begin{equation*}
\begin{split}
\alpha_1&= \tfrac{\lambda  \left(-\sqrt{3} (z+1) \left(-\log \left(z^2+z+1\right)+2 \log (1-z)-\log \left(\bar z^2+\bar z+1\right)+2 \log (1-\bar z)\right)-6 (z-1) \tan ^{-1}\left(\tfrac{2 \bar z+1}{\sqrt{3}}\right)+6 (z-1) \tan ^{-1}\left(\tfrac{2 z+1}{\sqrt{3}}\right)\right)}{6 \sqrt{2 z \bar z+2}}\\
\alpha_2&= \tfrac{i \lambda ^2 \left(\sqrt{3} (\bar z-1) \left(-\log \left(z^2+z+1\right)+2 \log (1-z)-\log \left(\bar z^2+\bar z+1\right)+2 \log (1-\bar z)\right)+6 (\bar z+1) \tan ^{-1}\left(\tfrac{2 z+1}{\sqrt{3}}\right)-6 (\bar z+1) \tan ^{-1}\left(\tfrac{2 \bar z+1}{\sqrt{3}}\right)\right)}{6 \sqrt{2 z \bar z+2}}\\
\beta_1&= -\tfrac{\lambda ^2 \left(\sqrt{3} (\bar z-1) \left(-\log \left(z^2+z+1\right)+2 \log (1-z)-\log \left(\bar z^2+\bar z+1\right)+2 \log (1-\bar z)\right)+6 (\bar z+1) \tan ^{-1}\left(\tfrac{2 z+1}{\sqrt{3}}\right)-6 (\bar z+1) \tan ^{-1}\left(\tfrac{2 \bar z+1}{\sqrt{3}}\right)\right)}{6 \sqrt{2 z \bar z+2}}\\
\beta_2&= -\tfrac{i \lambda  \left(-\sqrt{3} (z+1) \left(-\log \left(z^2+z+1\right)+2 \log (1-z)-\log \left(\bar z^2+\bar z+1\right)+2 \log (1-\bar z)\right)-6 (z-1) \tan ^{-1}\left(\tfrac{2 \bar z+1}{\sqrt{3}}\right)+6 (z-1) \tan ^{-1}\left(\tfrac{2 z+1}{\sqrt{3}}\right)\right)}{6 \sqrt{2 z \bar z+2}}
\end{split}
\end{equation*}
and
\[
X_1=\begin{pmatrix} 0& y_1&y_2\\x_1&*&*\\
x_2&*&*\end{pmatrix}\]
for
\begin{equation*}
\begin{split}
y_1&= \tfrac{\pi  \left(\lambda ^3-1\right)+\sqrt{3} \left(\log \left(z^2+z+1\right)-2 \log (1-z)+\lambda ^3 \left(\log \left(\bar z^2+\bar z+1\right)-2 \log (1-\bar z)\right)\right)+6 \tan ^{-1}\left(\tfrac{2 z+1}{\sqrt{3}}\right)-6 \lambda ^3 \tan ^{-1}\left(\tfrac{2 \bar z+1}{\sqrt{3}}\right)}{6 \sqrt{2} \lambda }\\
y_2&= -\tfrac{i \left(\pi  \left(\lambda ^3+1\right)+\sqrt{3} \left(\log \left(z^2+z+1\right)-2 \log (1-z)+\lambda ^3 \left(2 \log (1-\bar z)-\log \left(\bar z^2+\bar z+1\right)\right)\right)-6 \tan ^{-1}\left(\tfrac{2 z+1}{\sqrt{3}}\right)-6 \lambda ^3 \tan ^{-1}\left(\tfrac{2 \bar z+1}{\sqrt{3}}\right)\right)}{6 \sqrt{2} \lambda ^2} \\
x_1&= \tfrac{\pi  \left(\lambda ^3-1\right)+\sqrt{3} \left(-\log \left(z^2+z+1\right)+2 \log (1-z)+\lambda ^3 \left(2 \log (1-\bar z)-\log \left(\bar z^2+\bar z+1\right)\right)\right)+6 \tan ^{-1}\left(\tfrac{2 z+1}{\sqrt{3}}\right)-6 \lambda ^3 \tan ^{-1}\left(\tfrac{2 \bar z+1}{\sqrt{3}}\right)}{6 \sqrt{2} \lambda ^2}  \\
x_2&= -\tfrac{i \left(\pi  \left(\lambda ^3+1\right)+\sqrt{3} \left(-\log \left(z^2+z+1\right)+2 \log (1-z)+\lambda ^3 \left(\log \left(\bar z^2+\bar z+1\right)-2 \log (1-\bar z)\right)\right)-6 \tan ^{-1}\left(\tfrac{2 z+1}{\sqrt{3}}\right)-6 \lambda ^3 \tan ^{-1}\left(\tfrac{2 \bar z+1}{\sqrt{3}}\right)\right)}{6 \sqrt{2} \lambda }.
\end{split}
\end{equation*}
By \eqref{t-reconstruction}, the horizontal lift $\hat f_t$ of the minimal Lagrangian surface $f_t\colon \widetilde{S} \to \CP^2$ is given by
\[F^{-1}_t(\lambda_0\,) \begin{psmallmatrix}1\\0\\0\end{psmallmatrix}\colon \widetilde S\to \S^5\]
where $\lambda_0=\exp(\tfrac{2\pi i}{12}).$ 

By \eqref{eq:fact=0}, $\hat f_{t=0}$  is constant on simply connected compact subsets ${\mathbb K}$, i.e., $d\hat f_{t=0}=0.$
Furthermore, taking the derivative with respect to $t$ at $t=0$ yields
\begin{align}\label{eq:sherk}
&\dot{\hat f} \colon \widetilde S \to \C^2\cong (1,0,0)^\perp;\quad z\mapsto  \\
&\begin{psmallmatrix}
\tfrac{(-1)^{2/3} \left(\sqrt{3} \left(\log \left(z^2+z+1\right)-2 \log (1-z)+i \log \left(\bar z^2+\bar z+1\right)-2 i \log (1-\bar z)\right)-6 \tan ^{-1}\left(\tfrac{2 z+1}{\sqrt{3}}\right)+6 i \tan ^{-1}\left(\tfrac{2 \bar z+1}{\sqrt{3}}\right)+\pi  (1-i)\right)}{6 \sqrt{2}}\\
\tfrac{\sqrt[3]{-1} \left(\sqrt{3} \left(\log \left(z^2+z+1\right)-2 \log (1-z)-i \log \left(\bar z^2+\bar z+1\right)+2 i \log (1-\bar z)\right)+6 \tan ^{-1}\left(\tfrac{2 z+1}{\sqrt{3}}\right)+6 i \tan ^{-1}\left(\tfrac{2 \bar z+1}{\sqrt{3}}\right)+\pi  (-1-i)\right)}{6 \sqrt{2}}\end{psmallmatrix}\nonumber
\end{align}
with differential
\begin{equation}\label{eq:sherkder}
\begin{split}
d\dot{\hat f}&=\sqrt{\tfrac{3}{2}}\begin{pmatrix}e^{-\tfrac{\pi i}{3}} z \\ -e^{\tfrac{\pi i}{3}}\end{pmatrix}
\frac{dz}{z^3-1}
+\sqrt{\tfrac{3}{2}}\begin{pmatrix}e^{\tfrac{\pi i}{6}} \\ -e^{-\tfrac{\pi i}{6}} \bar z\end{pmatrix}
\frac{d\bar z}{\bar z^3-1}\end{split},
\end{equation}
as claimed.
\end{proof}

It is well-known (\cite{HR0,HR1}) that minimal Lagrangian surfaces $f$ in $\C^2$ admit a Weierstrass-type representation in terms of holomorphic 1-forms.
Using a different complex structure $j$ on the real vector space $\C^2=\R^4=\mathbb H$, 
the surface $f$ becomes a holomorphic curve in $\C^2\cong(\mathbb H,j).$ In fact, 
\[d\dot{\hat f}=(dH_1,dH_2)+(-i d\bar H_2,i d\bar H_1)\]
for
\[d H_1=\sqrt{\tfrac{3}{2}}e^{-\frac{\pi  i}{3}}\frac{zdz}{z^3-1}\quad\text{and}\quad d H_2=-\sqrt{\tfrac{3}{2}}e^{\frac{\pi  i}{3}}\frac{dz}{z^3-1}.\]
Equivalently, we have
\begin{equation}\label{eq:rotdfdot}
\begin{psmallmatrix}
 \frac{1}{\sqrt{2}} & 0 & 0 & \frac{1}{\sqrt{2}} \\
 0 & \frac{1}{\sqrt{2}} & \frac{1}{\sqrt{2}} & 0 \\
 0 & -\frac{1}{\sqrt{2}} & \frac{1}{\sqrt{2}} & 0 \\
 -\frac{1}{\sqrt{2}} & 0 & 0 & \frac{1}{\sqrt{2}} \\
\end{psmallmatrix} d\dot{\hat f}=\sqrt{2}\begin{pmatrix}dH_1\\dH_2\end{pmatrix}.\end{equation}
As the monodromy of $d\dot{\hat f}$ is abelian, $\dot{\hat f}$ is already a well-defined immersion (without nontrivial translational periods) on the universal abelian covering
\[M:=\widetilde S/\Lambda^{com} \to \CP^1\setminus\{1,\zeta,\zeta^2\}\]
where $\widetilde S\to\CP^1\setminus\{1,\zeta,\zeta^2\}$ is the universal covering, and $\Lambda^{com}:=[\pi_1,\pi_1]\subset \pi_1(\CP^1\setminus\{1,\zeta,\zeta^2\})$
is the commutator subgroup.
As the periods of $dH_1,dH_2$ vanish on $\Lambda^{com}$,  $\dot{\hat f}$ descends to $M$.

\begin{Cor}\label{cor:scherk}
The surface $\dot{\hat f}\colon M\to \C^2 $ as defined in \eqref{eq:sherk} with differential \eqref{eq:sherkder} is properly embedded.
We call it the Scherk minimal Lagrangian surface.
\end{Cor}
\begin{proof}
This follows from \eqref{eq:rotdfdot}. In fact, after applying the same M{\"o}bius transformation \eqref{mobius} as before, the differential of the $j$-holomorphic curve becomes
$A\frac{dz}{z}+B\frac{dz}{z-1}$ with linearly independent vectors 
\[A:=(\tfrac{1+i \sqrt{3}}{2 \sqrt{6}},\tfrac{-1 + i \sqrt{3}}{2 \sqrt{6}}) \qquad\text{ and }\qquad B:=(\tfrac{1-i \sqrt{3}}{2 \sqrt{6}},\tfrac{-1 - i\sqrt{3}}{2 \sqrt{6}}).\]
Hence on the universal abelian cover, where $\log z$ and $\log(z-1)$ are single-valued,
\[(H_1,H_2)=A\log z+B\log(z-1)+\text{const}\]
Since $A$ and $B$ are linearly independent, this is an affine embedding of the universal abelian cover into $\C^2$. That gives injectivity immediately, and properness follows from the additional fact that 
$p\mapsto \bigl(\log z(p),\log(z-1)(p)\bigr)$
is a holomorphic embedding of M into $\C^2$.
\end{proof}

\subsection{The limiting structure of the Riemann surfaces, and the normalization in space}\label{sec:RsandN}
Recall that the underlying Riemann surface $\Sigma_k$ is the Fermat curve
\[\{[X:Y:Z]\in\CP^2\mid X^k+\zeta\, Y^k+\zeta^2\, Z^k=0\}\]
where $\zeta=\exp(\tfrac{2\pi i}{3}).$
The group
$\Gamma_{deck}
$
of deck transformations of $\Sigma_k\to\CP^1\cong\Sigma_k/\Gamma_{deck}$
acts by extrinsic symmetries on minimal Lagrangian immersions into $\CP^2$ as well as on their special Legendrian lifts in $\S^5.$
On the holomorphic curve $\Sigma_k\subset\CP^2$, this group
is generated by
$\varphi_1, \varphi_2, \varphi_3:=(\varphi_2\varphi_1)^{-1}$ and $\sigma$ defined in \eqref{eq:autos}.
Their orders are $k,\,k,\,k$, and $3$, respectively.

The monodromy of $\nabla_{\lambda_0}$ at $\lambda_0=\exp(\tfrac{2\pi i}{12})$ 
on the thrice-punctured sphere is the representation $\rho$ in \eqref{finite-group-rep1} generated by
$\sigma\cong M_0$ and $\varphi_1\cong M_1.$ Furthermore, the symmetries $\varphi_j$ for $j=1, 2, 3$ pairwise
commute, while $\varphi_1$ and $\sigma$ do not commute; for example, $\varphi_2=\sigma^{-1}\circ\varphi_1\circ\sigma.$

We label the special points on 
$\Sigma_k, $ i.e., the fixed points of certain elements of $\langle\varphi_1,\sigma\rangle\cong\Gamma_{deck}^{op}$ as follows:
\begin{equation*}
\begin{split}
p^1_1&=[0,1,\exp(\tfrac{\pi i}{3k})],\quad\quad
q_{1,1}=[1,1,1]\\
p^1_\ell&
=\varphi_3^{\ell-1}\,p^1_1,\quad
  p^2_\ell  =\sigma^{-1}\,p^1_\ell,\quad
p^3_\ell
=\sigma\, p^1_\ell,\quad
q_{j,\ell}=\varphi_2^{\ell-1}\varphi_1^{j-1} \, q_{1,1}
\end{split}
\end{equation*}
where $\ell, j\in\{1,\dots,k\}.$  Correspondingly, we denote the symmetries of $\S^5$ as follows:
\begin{equation}\label{eq:symSU3}\tilde\varphi_1=\begin{psmallmatrix}1&0&0\\0&\exp(\tfrac{2\pi i}{k})&0\\0&0&\exp(-\tfrac{2\pi i}{k})\end{psmallmatrix},\; \tilde\varphi_2=\begin{psmallmatrix}\exp(\tfrac{2\pi i}{k})&0&0\\0&\exp(-\tfrac{2\pi i}{k})&0\\0&0&1\end{psmallmatrix},\;
\tilde\varphi_3=(\tilde\varphi_2\tilde\varphi_1)^{-1},\; 
\tilde\sigma=
\begin{psmallmatrix} 0&1&0\\0&0&1\\1&0&0\end{psmallmatrix}.\end{equation}
Note that $\tilde\varphi_1=M_1^{-1}$ and $\tilde\sigma=M_0^{-1}$ and recall (e.g. \eqref{eq:equivara}) that the deck transformations $\gamma\in\Gamma_{deck}$
 act via $\rho(\gamma)^{-1}$ on the special Legendrian immersion.

\begin{Rem}\label{rem:normalization}
We will see that the surfaces $\hat f_k$ converge to infinitely many spherical caps, which are {\em generated} by the infinitesimal symmetries $\tilde\varphi_j$ for $k\to\infty$. Moreover, the surfaces $\hat f_k$
are normalized through the condition that
the distinguished points $p^j_\ell$ and $q_{j,\ell}$, each fixed by an appropriate nontrivial element of $\Gamma_{deck}$,
 are mapped to fixed points of 
the corresponding $\SU_3$ transformations acting on $\S^5.$ Because the $\SU_3$-representation $\rho$ 
of $\Gamma_{deck}$ is irreducible, 
and the special Legendrian lift is unique up to global multiplication by a cube root of unity, the images
of $p^j_\ell$ and $q_{j,\ell}$ in $\S^5$ are uniquely determined up to overall multiplication by $\zeta$ or $\zeta^2,$ 
which can be ignored as this action is discrete.
\end{Rem}

\subsubsection{The coordinate $w$}
We consider
the order $k$ automorphism
$\varphi:=\varphi_1\circ\varphi_2^{-1}.$
If $k$ is odd, the automorphism $\varphi$ has no fixed points. This yields an unramified cover
$\Sigma_k\to\underline{\Sigma}_k$ where $\underline{\Sigma}_k$ is the hyperelliptic curve given by
\[v^2=w^k-1.\]
The Weierstrass points of $\underline{\Sigma}_k$ are the $k$th roots of unity and $\infty.$ Removing the preimages of the unit circle yields three disjoint discs
\[\mathbb D_\pm,\mathbb D_\infty\subset \underline{\Sigma}_k.\]
Each of these three discs has exactly $k$ connected components as preimages in $\Sigma_k$. Although $\underline\Sigma_k$ does not admit a well-defined $\Z_3$ action 
permuting $\mathbb D_\pm,\mathbb D_\infty,$ 
the corresponding preimages in $\Sigma_k$ are mapped by $\hat f_k$ to extrinsically isometric regions in $\S^5$.
Note that \[w\colon \mathbb D_\pm\to \{w\in\C\mid |w|<1\}\] is a natural coordinate on $\mathbb D_\pm.$ Indexing the discs by $\{1,2,3\}$ (up to permutation, i.e.\ after applying the symmetry $\sigma$), we also obtain a natural coordinate, still denoted by $w$, on $\mathbb D_\infty$, whose image is again the unit disc.
The preimages of the three discs $\mathbb D_\pm$, $\mathbb D_\infty$ in $\Sigma_k$ consist of $3k$ many connected components.
 They are labeled by $\mathbb D^j_\ell$ according to their centers,
 \[p_\ell^j\in \mathbb D^j_\ell\subset \Sigma_k,\]
 where $j \in \{1, 2, 3\}$ and $\ell \in \{1, 2, ..., k\}$.
 Note that for every $k$, the union of these $3k$ discs is open and dense in $\Sigma_k$. 
 For all of them, we use $w$ as a natural coordinate which maps them to the unit disc.
 Note that on $\mathbb D^j_\ell$, $j=1,2,3$, we have
 \begin{equation}\label{eq:zw}z=\zeta^{j-1}-w^k\end{equation}
 where $z$ is the standard affine coordinate on $\CP^1=\Sigma_k/(\Z_k\times\Z_k)$, normalized by the condition that the 
 three branch values 
 of the degree $k^2$ map  $\Sigma_{k}\to\CP^1$
 are at $1,\zeta,\zeta^2.$ Furthermore
 \begin{equation*}w_{\mid \mathbb D^j_\ell}\circ\varphi_j=\exp(\tfrac{2\pi i}{k}) w_{\mid \mathbb D^j_\ell},\end{equation*}
 while $w$ is invariant under some appropriate $\varphi_{\tilde j},$ e.g., on $\mathbb D^1_\ell$ the coordinate $w$ is invariant under $\varphi_3$ (as $\varphi_3$ maps 
 $\mathbb D^1_\ell$ to $\mathbb D^1_{\ell+1}$).
 
 Similarly, for even $k$, we also obtain $3k$ disjoint discs $\mathbb D^j_\ell\subset\Sigma_k$ with centers $p^j_\ell$.  On each $\mathbb D^j_\ell$ there exists
 a coordinate $w$ which maps biholomorphically to the unit disc and satisfies \eqref{eq:zw}. Furthermore, the union of the  $3k$ discs 
 is open and dense in $\Sigma_k.$
 
\subsubsection*{The special Legendrian caps}
Note that the unique (up to multiplication by $\zeta$ or $\zeta^2$)  totally geodesic special Legendrian closed cap $\hat{\mathbb D}_1^1$
in $\S^5$ of radius $\arctan(\sqrt{2})$
with center on $\S^1\oplus0\oplus0\subset \S^5\subset\C^3$ and with a boundary point that is a fixed point of $\tilde\sigma$ (defined in \eqref{eq:symSU3}) is
contained in and therefore determined by the plane \[E=\mathrm{span}_\R(\,\begin{psmallmatrix}i\\0\\0\end{psmallmatrix}, \begin{psmallmatrix}0\\i\\i\end{psmallmatrix}, \begin{psmallmatrix}0\\1\\-1\end{psmallmatrix}\,)\subset\C^3.\]
Explicitly
\[\hat{\mathbb D}_1^1=\{\cos(u)\begin{psmallmatrix}i\\0\\0\end{psmallmatrix}+\sin(u)(\tfrac{\cos(v)}{\sqrt{2}}\begin{psmallmatrix}0\\i\\i\end{psmallmatrix}+\tfrac{\sin(v)}{\sqrt{2}}\begin{psmallmatrix}0\\1\\-1\end{psmallmatrix})\,\mid\, u\in[0,\arctan(\sqrt{2})],\,v\in[0,2\pi]\,\}.\]

As a consequence of the above discussion, we define
\[P^1_1=\begin{psmallmatrix}i\\0\\0\end{psmallmatrix}\quad\text{and}\quad Q_{1,1}=\tfrac{i}{\sqrt{3}}\begin{psmallmatrix}1\\1\\1\end{psmallmatrix}\]
as well as
\begin{equation}\label{eq:S5liftedpoints}
P^1_\ell
=\tilde\varphi_3^{\ell-1}\,P^1_1,\quad
  P^2_\ell
  =\tilde\sigma^{-1}\,P^1_\ell,\quad
P^3_\ell
=\tilde\sigma^{-2}\, P^1_\ell,\quad\text{and}\quad
Q_{j,\ell}=\tilde\varphi_2^{\ell-1}\tilde\varphi_1^{j-1} \, Q_{1,1}.
\end{equation}

Finally, we define the following $3k$ spherical caps with centers $P^1_\ell,P^2_\ell, P^3_\ell$, $\ell=1,\dots,k$,
\[\hat{\mathbb D}^1_\ell=\tilde\varphi_3^{\ell-1}\, \hat{\mathbb D}_1^1,\quad \hat{\mathbb D}^2_\ell=\tilde\sigma^{-1}\hat{\mathbb D}^1_\ell,\quad\text{and}\quad \hat{\mathbb D}^3_\ell=\tilde\sigma\,\hat{\mathbb D}^1_\ell.\]
As discussed in Remark \ref{rem:normalization}, the very form of the symmetries already fixes the surfaces $\hat f_k$ in $\S^5$ up to multiplication by $\zeta$ or $\zeta^2$. 
It is not clear whether all the points in \eqref{eq:S5liftedpoints} are
contained in the image of the normalized surface \(\hat f_k\) for every
finite \(k\geq k_0\).  However, their projections to \(\CP^2\) lie on the
corresponding minimal Lagrangian surfaces \(f_k\).  The points themselves
appear in the limiting surface \(\hat f_0=\lim_{k\to\infty}\hat f_k\), as
we prove below.

In fact, it follows from our analysis below that the limit of $\hat f_k$ for $k\to\infty$ is given by
\begin{itemize}
\item the closed discs $\hat {\mathbb D}^1_\ell$, $\hat {\mathbb D}^2_\ell$, and $\hat {\mathbb D}^3_\ell$ for  $\ell=1,2,\dots$ with centers 
equidistributed on the three circles $\S^1\oplus0\oplus0$, $0\oplus\S^1\oplus0$ and $0\oplus0\oplus \S^1;$
\item the boundaries of the discs which all lie on the Clifford torus generated from $Q_{1,1}$ by applying the symmetries $(\tilde\varphi_1)^j\circ (\tilde\varphi_2)^\ell$ for $j,\ell=1,2,\dots\;\,$.
\end{itemize}
\begin{Rem}\label{rem:intersection}
For the model caps defined above, and for fixed finite $k$,
\[\hat {\mathbb D}^j_\ell\cap \hat {\mathbb D}^j_{\tilde\ell}=\emptyset \quad\text{for all}\quad  j=1,2,3,\ \ell\neq\tilde\ell,\]
while for $j_1\neq j_2$ and all $\ell,\tilde\ell$ the closures of the discs
$\hat {\mathbb D}^{j_1}_\ell$ and $\hat {\mathbb D}^{j_2}_{\tilde \ell}$ intersect in precisely one point $a\in\S^5$, which is contained in the intersection of their boundaries. Furthermore, there is a unique $\hat\ell\in\{1,\dots,k\}$
such that for $\{j_1,j_2,j_3\}=\{1,2,3\}$ 
\[\hat {\mathbb D}^{j_3}_{\hat\ell}\cap \hat {\mathbb D}^{j_1}_\ell\cap \hat {\mathbb D}^{j_2}_{\tilde \ell}=\{a\}.\]
\end{Rem}

\subsection{The spherical part}
In this subsection we determine the limiting behaviour of the special Legendrian surfaces $\hat f_k$ restricted to the discs 
as $k\to\infty$. 
We first derive the asymptotic expansion of the normalized parallel frame, which then yields the limit of $\hat f_k$ on the discs.
 We finally establish uniform convergence of the surfaces up to the boundary of the discs.

We continue to work with the interpolating family
\[
\hat f_t=F_t(\lambda_0)^{-1}e_1\colon \widetilde S\to \S^5
\]
from \eqref{t-reconstruction}. For $t=\frac1k$, this agrees, up to an ambient $\SU_3$-transformation, with $\hat f_k\circ\widetilde\pi_k$.
To analyze the component lying over $z=1$, we use its natural coordinate $w$
on the discs ${\mathbb D}^1_\ell \subset\Sigma_k$ introduced in \eqref{eq:zw}. For $t=\frac1k$ this is characterized by
\[
z=1-w^k.
\]
For the continuous $t$-family it is convenient to write instead
\[
z=1-e^u,\qquad w=e^{tu}.
\]
Consider the limit  
\begin{equation}\label{eq:tildeOmegalimit}
\tilde\Omega(w):=\lim_{t\to 0} \Omega_t(1-w^{\tfrac{1}{t}}).
\end{equation}
\begin{Lem}\label{lem:sphericallemma}
The limit in \eqref{eq:tildeOmegalimit}
converges locally uniformly on $\{w\in\C\mid 0< |w|<1\}$ to
\begin{equation*}
\tilde\Omega(w)=
\begin{psmallmatrix}
 1 & 0 & 0 \\
 0 & 1 & -\frac{i }{\lambda } \\
 0 & 0 & 1 \\
\end{psmallmatrix}
\exp(-
A_1 \log w)\qquad\text{for}\quad
A_1:=\begin{psmallmatrix}
 0 & -\frac{\lambda ^3-1}{\sqrt{6} \lambda  } & -\frac{i \left(\lambda ^3+1\right)}{\sqrt{6} \lambda ^2 } \\
 \frac{\lambda ^3-1}{\sqrt{6} \lambda ^2 } & 0 & \frac{i}{\sqrt{3} \lambda  } \\
 \frac{i \left(\lambda ^3+1\right)}{\sqrt{6} \lambda  } & -\frac{i \lambda }{\sqrt{3} } & 0 \\
\end{psmallmatrix}.
\end{equation*}
Furthermore,  \eqref{eq:tildeOmegalimit}
 converges uniformly for $w\in[\tfrac12,1]\subset\R.$ 
\end{Lem}
\begin{Rem}
The singularity of $\tilde \Omega$ at $w=0$ will be removed through the loop group factorization. 
\end{Rem}
 \begin{proof}
 By Theorem \ref{thm:expansionorder2}, the functions $a,b,c$ are, up to second order in $t$, given by the central values, i.e.,
 \[a(t)=\tfrac{1}{\sqrt{6}}+O(t^2),\; b(t)=-\tfrac{1}{\sqrt{6}}+O(t^2), \quad\text{and}\quad c(t)=\tfrac{1}{\sqrt{3}}+O(t^2).\]
 Define \[\phi:=
\begin{psmallmatrix}
 1 & 0 & 0 \\
 0 & 1 & -\frac{i z}{\lambda } \\
 0 & 0 & 1 \\
\end{psmallmatrix}\]
and $\mu_t$ via $d+\mu_t=(d+\omega_t).\phi.$
Thus, we have
on the thrice-punctured sphere with coordinate $z$
\begin{equation}
\begin{split}\label{eq:muzexp}
\mu_t&=\omega_t.\phi=t
\begin{psmallmatrix}
 0 & -\frac{\sqrt{\frac{3}{2}} \left(\lambda ^3-1\right)}{\lambda  \left(z^3-1\right)} & -\frac{i \sqrt{\frac{3}{2}} \left(\lambda ^3+1\right) z}{\lambda ^2 \left(z^3-1\right)} \\
 \frac{\sqrt{\frac{3}{2}} \left(\lambda ^3-1\right) z}{\lambda ^2 \left(z^3-1\right)} & 0 & \frac{i \sqrt{3}}{\lambda  \left(z^3-1\right)} \\
 \frac{i \sqrt{\frac{3}{2}} \left(\lambda ^3+1\right)}{\lambda  \left(z^3-1\right)} & -\frac{i \sqrt{3} \lambda  z}{z^3-1} & 0 \\
\end{psmallmatrix}dz+O(t^3).
\end{split}
\end{equation}
Set
\[\Psi_t:=\phi^{-1}\Omega_t.\]
Since $\phi(z=0)=\Id,$ it then satisfies the ODE
\[d\Psi_t+\mu_t\Psi_t=0\qquad\text{
with initial value}
\qquad\Psi_t(z=0)=R(t)=\Id+O(t^2).\]

Using $z=1-w^{\tfrac{1}{t}}$, 
we then obtain on every compact subset of the unit disc 
 in the $w$-coordinate (including the pole at $w=0$)
 uniform convergence
 \begin{equation}\label{eq:mut}\mu_t=A_{1}\frac{dw}{w}+O(t^2).
 \end{equation}

Thus, consider
\[\phi_1(z):=\exp(-t A_{1}\log(1-z))=\exp(-A_{1} \log w).\]
For each $\lambda$, the matrix $A_1$ is conjugate to $\diag(1,0,-1)$; in particular,
$\exp(-2\pi i A_1)=\Id$, so the expression $\exp(-A_1\log w)$ is single-valued on $0<|w|<1$.

The map $\phi_1(z)$ is uniformly bounded in $\Lambda\mathrm{SL}(3,\C)$ on the real interval 
$z\in[0,1-\tfrac{1}{2^{1/t}}]$
and satisfies \begin{equation}\label{eq:Gz0P}\phi_1(z=0)=\Id=\Psi_t(z=0).\end{equation}
Write
\[\Psi_t=\phi_1 Y_t.\]
As in  \cite[Section 3.5, Lemma 24]{CHHT} we use the variation of constants method and Gr{\"o}nwall's inequality to prove
\[Y_t(z=1-\tfrac{1}{2^{1/t}})\to_{t\to 0}\Id.\]
In fact, we obtain from \eqref{eq:muzexp}
\begin{equation}\label{eq:dYBY}dY_t=Q_tY_t\end{equation}
where $Q_t$ is uniformly bounded (note that $Q_0$ can be computed explicitly in terms of the residues  at $z=\xi$ and $z=\xi^2$ of the potential $\mu_t$, for details see Section \ref{sec:firstorderexpansion} below).
Integrating \eqref{eq:dYBY}, using \eqref{eq:Gz0P} and taking norms, we obtain for $z\in[0,1-\tfrac{1}{2^{1/t}}]$
\[\parallel Y_t(z)-\Id\parallel\leq C |t|\int_0^z\parallel Y_t(z)\parallel\]
for some constant $C>0.$
Gr{\"o}nwall's inequality yields
\[\parallel Y_t(z=1-\tfrac{1}{2^{1/t}})-\Id\parallel\leq C |t|\exp(C|t|)\to_{t\to0}0.\]
Hence,
\[\lim_{t\to0}\Psi_t(z=1-\tfrac{1}{2^{1/t}})=\lim_{t\to0}\phi_1(z=1-\tfrac{1}{2^{1/t}})=\exp(A_{1}\log 2).\]
Furthermore, we obtain from the locally uniform equation \eqref{eq:mut} with respect to the coordinate $w$ with $z=1-w^{1/t}$ and on compact subsets of $|w|<1$ that
\[\lim_{t\to0}\Psi_t=\exp(-A_{1} \log w)\]
or equivalently
$\tilde\Omega(w)=\phi(z=1) \exp(-A_{1} \log w)$
as claimed. 

 That \eqref{eq:tildeOmegalimit}
 converges uniformly for $w\in[\tfrac12,1]\subset\R$
then clearly follows from the first part of the proof, i.e., from our uniform estimate along the $z$-segment $[0,1-2^{-\tfrac{1}{t}}],$
together with the locally uniform estimate for $0<|w|<1.$
\end{proof}
With this preparation, we can compute the limit surface:
\begin{The}\label{The:limit-surfacerp2}
For each $k$, let $\mathbb D_1^1\subset \Sigma_k$ be the distinguished disc containing $p_1^1$, and let
\[
w:\mathbb D_1^1\to \mathbb D:=\{\,|w|<1\,\}
\]
be its natural coordinate. 
Then, with the normalization fixed in Section \ref{sec:RsandN},
\[
\hat f_k\circ w^{-1}\colon \mathbb D\to \S^5
\]
converge locally uniformly, as $k\to\infty$, to
\begin{equation}\label{eq:spherecappara}
\hat f
=\begin{pmatrix}\frac{2 i\left(\sqrt{3}+2\right)}{x^2+y^2+\sqrt{3}+2}-i\\
\frac{\left(\sqrt{3}+1\right) (ix+y)}{x^2+y^2+\sqrt{3}+2}\\
\frac{\left(\sqrt{3}+1\right) (ix-y)}{x^2+y^2+\sqrt{3}+2}\end{pmatrix} \qquad\text{with}\quad w=x+iy.\end{equation}
\end{The}
\begin{Rem}
Note that the closure of $\hat f({\mathbb D})$ is exactly $\hat {\mathbb D}_1^1,$
and that the area of this spherical cap is exactly
$2\pi(1-\sqrt{\tfrac{1}{3}}),
$
which perfectly fits with Theorem \ref{Thm:areas}.
\end{Rem}
\begin{proof}
We factor $\tilde\Omega(w)$ into positive and unitary parts. Write 
\[w=r\exp(i\varphi)\]
to obtain
\begin{equation*}
\begin{split}
&\tilde\Omega(w)=
\begin{psmallmatrix}
 1 & 0 & 0 \\
 0 & 1 & -\frac{i }{\lambda } \\
 0 & 0 & 1 \\
\end{psmallmatrix} \exp(-A_{1}\log r)\exp(-i A_{1}\varphi)
\end{split}
\end{equation*}
where the last factor 
\[F_\varphi:=\exp(-iA_{1}\varphi)
\]
maps into
${\bf SU}^{\hat\tau}(\S^1)$
because $iA_{1}\in{\bf su}_3^{\hat\tau}(\S^1).$ 
Hence, it is enough to factor \[
\begin{psmallmatrix}
 1 & 0 & 0 \\
 0 & 1 & -\frac{i }{\lambda } \\
 0 & 0 & 1 \\
\end{psmallmatrix} \exp(-A_{1}\log r)=BF_0\]
into its (twisted) positive part $B$  and its unitary part $F_0$.
We only state the result of the computation:
\begin{equation}\label{eq:comF0}
\begin{split}
&B=
\begin{psmallmatrix}
 1 & -\frac{\left(\sqrt{3}+1\right) \lambda ^2 \left(r^2-1\right)}{\sqrt{2} \left(r^2+\sqrt{3}+2\right)} & -\frac{i \sqrt{2} \left(\sqrt{3}+1\right) \lambda  \left(r^2-1\right)}{r^2+\sqrt{3}+2} \\
 \frac{\lambda  \left(r^2-1\right)}{\sqrt{3} r} & \frac{-\frac{\sqrt{3} \left(4 \sqrt{3}+7\right) \lambda ^3 \left(r^2-1\right)^2}{r^2+\sqrt{3}+2}+\left(2 \sqrt{3}+3\right) r^2+7 \sqrt{3}+12}{6 \left(\sqrt{3}+2\right)^{3/2} r} & -\frac{i \sqrt{\frac{2}{3}+\frac{1}{\sqrt{3}}} \lambda ^2 \left(r^2-1\right)^2}{r \left(r^2+\sqrt{3}+2\right)} \\
 0 & -\frac{i \left(7 \sqrt{3}+12\right) \lambda  r}{\left(\sqrt{3}+2\right)^{3/2} \left(r^2+\sqrt{3}+2\right)} & \frac{2 \sqrt{3 \left(\sqrt{3}+2\right)} r}{r^2+\sqrt{3}+2} \\
\end{psmallmatrix}\\
&\\
&F_0=\\
&\begin{psmallmatrix}
 \frac{r \left(\sqrt{3} -r+2 \sqrt{3}+6\right)+2 \sqrt{3}+3}{3 \left(r^2+\sqrt{3}+2\right)} & -\frac{\left(\lambda ^3+1\right) (r-1) \left(\left(\sqrt{3}+3\right) r+5 \sqrt{3}+9\right)}{6 \sqrt{\sqrt{3}+2} \lambda  \left(r^2+\sqrt{3}+2\right)} & -\frac{i \left(\lambda ^3-1\right) (r-1) \left(\left(\sqrt{3}+3\right) r+5 \sqrt{3}+9\right)}{6 \sqrt{\sqrt{3}+2} \lambda ^2 \left(r^2+\sqrt{3}+2\right)} \\
 \frac{\lambda  (r-1) \left(\sqrt{3} r+2 \sqrt{3}+3\right)}{3 \left(r^2+\sqrt{3}+2\right)} & \frac{\sqrt{3} r^2-\left(2 \sqrt{3}+3\right) \lambda ^3 (r-1)^2+6 r+4 \sqrt{3} r+7 \sqrt{3}+12}{6 \sqrt{\sqrt{3}+2} \left(r^2+\sqrt{3}+2\right)} & -\frac{i \left(\sqrt{3} r^2+\left(2 \sqrt{3}+3\right) \lambda ^3 (r-1)^2+6 r+4 \sqrt{3} r+7 \sqrt{3}+12\right)}{6 \sqrt{\sqrt{3}+2} \lambda  \left(r^2+\sqrt{3}+2\right)} \\
 -\frac{i (r-1) \left(r+\sqrt{3}+2\right)}{\sqrt{3} \lambda  \left(r^2+\sqrt{3}+2\right)} & \frac{\frac{i \left(2 \sqrt{3}+3\right) (r-1)^2}{\lambda ^2}-i \lambda  \left(r \left(\sqrt{3} r+4 \sqrt{3}+6\right)+7 \sqrt{3}+12\right)}{6 \sqrt{\sqrt{3}+2} \left(r^2+\sqrt{3}+2\right)} & \frac{\frac{\left(2 \sqrt{3}+3\right) (r-1)^2}{\lambda ^3}+\sqrt{3} r \left(r+2 \sqrt{3}+4\right)+7 \sqrt{3}+12}{6 \sqrt{\sqrt{3}+2} \left(r^2+\sqrt{3}+2\right)} \\
\end{psmallmatrix}.
\end{split}
\end{equation}
Note that $\exp(iA_{1}\varphi)(\lambda_0)$, where $\lambda_0=\exp(\tfrac{2 \pi i}{12}),$ is an $\SU_3$ rotation with eigenvalues $1, \exp(i\varphi),\exp(-i\varphi)$ and corresponding eigenvectors
\[\tfrac{1}{\sqrt{3}}\begin{psmallmatrix}1\\-(-1)^{11/12}\\-(-1)^{1/12}\end{psmallmatrix},\tfrac{1}{\sqrt{3}}\begin{psmallmatrix}-(-1)^{7/12}\\\sqrt[6]{-1}\\1\end{psmallmatrix},\tfrac{1}{\sqrt{3}}\begin{psmallmatrix}1\\-(-1)^{\tfrac{1}{4}}\\-(-1)^{\tfrac{3}{4}}\end{psmallmatrix}.\]
Up to $\SU_3$, the special Legendrian surface in $\S^5\subset\C^3$ is therefore given by
\[\tilde f=(F_0 \cdot F_\varphi)^{-1}(\lambda_0)\begin{psmallmatrix} 1\\0\\0\end{psmallmatrix}.\]

To align the surface with our previous normalization (coming from a specific form of its symmetries, see Section \ref{sec:RsandN}), the point $z=0$ needs to be mapped to
$\tfrac{i}{\sqrt{3}}\begin{psmallmatrix} 1\\1\\1\end{psmallmatrix}$ and 
the limit $w\to0,$ equivalently $z\to1,$ is mapped to 
$\begin{psmallmatrix}i\\0\\0\end{psmallmatrix}$. To achieve this, we apply the 
(constant) $\SU_3$ transformation
\[T=
\begin{psmallmatrix}
 \frac{i}{\sqrt{3}} & \frac{\left(\frac{1}{6}+\frac{i}{6}\right) \left(\sqrt{3}+3 i\right)}{\sqrt{2}} & -\frac{(-1)^{5/12}}{\sqrt{3}} \\
 \frac{i}{\sqrt{3}} & -\frac{1+i}{\sqrt{6}} & -\frac{1-i}{\sqrt{6}} \\
 \frac{i}{\sqrt{3}} & \frac{\left(\frac{1}{6}+\frac{i}{6}\right) \left(\sqrt{3}-3 i\right)}{\sqrt{2}} & \frac{(3+3 i)+(1-i) \sqrt{3}}{6 \sqrt{2}} \\
\end{psmallmatrix},\]
to obtain
\[\hat f=(T(F_\varphi)^{-1}(\lambda_0) T^{-1} )\cdot (T F_0^{-1}(\lambda_0))\begin{psmallmatrix}1\\0\\0\end{psmallmatrix}.\]
Note that
$T(F_\varphi)^{-1}(\lambda_0) T^{-1} =
\mathrm{diag}
(1,
 e^{-i \varphi } ,e^{i \varphi })$
and from \eqref{eq:comF0} we then obtain
\[(T F_0^{-1}(\lambda_0))\begin{psmallmatrix}1\\0\\0\end{psmallmatrix}=i\begin{psmallmatrix}\frac{2  \left(\sqrt{3}+2\right)}{r^2+\sqrt{3}+2}-1
\\ \frac{\left(\sqrt{3}+1\right) r}{r^2+\sqrt{3}+2}\\
\frac{\left(\sqrt{3}+1\right) r}{r^2+\sqrt{3}+2}\end{psmallmatrix}.\]
Therefore, the parametrized limit surface on the disc  $\{w\in\C\mid 0<|w|<1\}\subset\Sigma_k$ is given in conformal coordinates $w=x+iy$ 
by \eqref{eq:spherecappara}.

It remains to show that this locally uniform limit on $\{w\in\C\mid 0<|w|<1\}\subset\Sigma_k$ extends uniformly to $w=0.$
By the construction in Lemma \ref{harmonic-family-form} and by adapting the gauges \eqref{eq:gkdesinG}, 
the pole of \eqref{eq:mut} at $w=0$ is apparent for all small $t$:
there exists a smooth family of meromorphic gauge transformations $g_t\colon\{w\in\C\mid 0<|w|<\tfrac12\}\to{\bf SL}^{\hat\tau}(\D)$
(with pole at $w=0$)
such that
\[(d+\mu_t).g_t=d+\tilde A(w) dw +t^2 R_t(w) dw\]
for holomorphic $\tilde  A(w)$ and  $R_t(w)$ with 
uniform bound $|R_t(w)|<\tilde C$ for all $|w|<\tfrac12$ and small $t.$
Applying Gr{\"o}nwall's inequality and the Iwasawa decomposition proves the result.
\end{proof}

We improve the previous result to include the boundary of the disc.
\begin{Pro}\label{pro:uniform}
With the same notations as in Theorem \ref{The:limit-surfacerp2},
the surfaces $\hat f_k\circ w^{-1}\colon {\mathbb D}\to\S^5$ converge uniformly to $\hat f$ in Theorem \ref{The:limit-surfacerp2}.
\end{Pro}
\begin{proof}
To extend convergence to the boundary, we control the behaviour for complex $w$ near $|w|=1.$
Let $y_k\in {\mathbb D}\subset\Sigma_k$,  be a sequence of points with $y:=\lim_k y_k\in\partial {\mathbb D}$
such that 
\[\lim_k \hat f_k(y_k)\neq  \hat f(y).\]
Recall that each surface $\hat f_k$ has a rotational $\Z_k$-fold symmetry around the origin, where $\Z_k$ acts on the disc ${\mathbb D}$ in the obvious manner by a $k$-fold rotation $w\mapsto\exp(\tfrac{2\pi i}{k})w$.
Hence, without loss of generality (i.e., after applying the symmetries if necessary)
\[\arg(y_k)\in[0,\tfrac{2\pi}{k})\]
which gives $y=\lim_k y_k=1$ as the limit. 

Using $t=1/k$, we obtain  
 from Lemma \ref{lem:sphericallemma}, that $\tilde \Omega$ is the uniform limit of
$\Omega_{1/k}$  along the real line $w\in(0,1]$ for $k\to\infty$. 
In particular, $\lim_k \hat f_k(y_k)= \hat f(y)$  is true if $y_k\in(0,1]$ for all $k$ and $\lim y_k=1$ (as a consequence of
Theorem \ref{The:limit-surfacerp2}). We introduce a new coordinate $u$ to control the case when $y_k$ is not real.
For small $1/k=t>0$, and $0<|w|<1$ write
\begin{equation}\label{eq:coowu}w
= \exp( t u),\end{equation}
 where
$u\in\mathbb R^{<0}\times i[0,2\pi t^{-1})$.
In particular $z=1-e^u.$
Because $\arg(y_k)\in[0,\tfrac{2\pi}{k})$ the points $y_k$ correspond (via the coordinate transformation \eqref{eq:coowu}) to points
\[u_k\in \R^{<0}\times i[0, 2\pi). \]
On this domain, the DPW potential $\mu_t$ has a uniform expansion as
\[\mu_t=t\begin{psmallmatrix}
 0 & -\frac{\sqrt{\frac{3}{2}} (\lambda ^3-1) }{\lambda  (e^u (e^u-3)+3)} & \frac{i \sqrt{\frac{3}{2}} (\lambda ^3+1)  (e^u-1)}{\lambda ^2 (e^u (e^u-3)+3)} \\
 -\frac{\sqrt{\frac{3}{2}} (\lambda ^3-1)  (e^u-1)}{\lambda ^2 (e^u (e^u-3)+3)} & 0 & \frac{i \sqrt{3} }{\lambda  (e^u (e^u-3)+3)} \\
 \frac{i \sqrt{\frac{3}{2}} (\lambda ^3+1) }{\lambda  (e^u (e^u-3)+3)} & \frac{i \sqrt{3} \lambda (e^u-1)}{e^u (e^u-3)+3} & 0 \\
\end{psmallmatrix} du+ O(t^3).\]
In particular, there exists a constant  $C$ such that (with respect to the loop-algebra norm)
\[|\mu_t\tfrac{1}{du}|\leq  C t \quad\forall u\in \R^{<0}\times i[0, 2\pi).\]
Set 
\[u_k=x_k+i\varphi_k \quad\mathrm{ where} \quad x_k\in\R^{<0},\, \varphi_k\in[0,2\pi).\]
Hence, we can apply Gr{\"o}nwall's inequality once more to deduce that
\[|\ \Omega_t(x_k+i\varphi_k)-\Omega_t(x_k)|< \tilde C \,t.\]
Combining Theorem \ref{The:limit-surfacerp2} with the real-analyticity of the Iwasawa factorization 
and the fact that $\Omega_t(x_k+i\varphi_k)$ stays bounded if the corresponding sequence in $\D$ satisfies $y_k\to 1$
 completes  the proof.
\end{proof}

In summary, the local limit on each disc ${\mathbb D}^j_\ell$ is a totally geodesic special Legendrian cap, and the collection of these caps assembles -- according to the group symmetries -- into the global limiting configuration described in Remark \ref{rem:intersection}.

\subsection{First order expansion on the disc ${\mathbb D}$}\label{sec:firstorderexpansion}
To interpolate between the two asymptotic regimes -- spherical caps on the discs $\mathbb D$ at zeroth order and Scherk-type surfaces on simply connected domains of the thrice-punctured sphere at first order -- we must analyze the first-order variation of the limiting spherical caps. This analysis is an essential ingredient in establishing embeddedness.

We first determine the behaviour of
$\Omega_t$ up to first order $t$ in the disc ${\mathbb D}$. 
As before, we use
\[\Psi_t=\phi^{-1}\Omega_t\] which solves
$d\Psi_t+\mu_t\Psi_t=0$
with initial value $\Psi_t(z=0)=\Id+O(t^2).$
Set
\[\nu:=t A_1\frac{dz}{z-1}
\]
with primitive
\[\int_0^z\nu=t A_1\log(1-z).\]
Define
\[\phi_1=\exp(-\int_0^z\nu)\]
which then satisfies
$d\phi_1+\nu\phi_1=0.$
Set
\[\alpha_t=\mu_t-\nu.\]

Then a direct computation yields
\begin{equation}\label{eq:phi1twispot}
\begin{split}
(d+\mu_t).\phi_1&=d+\phi_1^{-1}(\mu_t)\phi_1+\phi_1^{-1}d\phi_1\\
&=d+{\phi_1}^{-1}\alpha_t \phi_1.
\end{split}
\end{equation}

\begin{Lem}\label{lem:secondestimate}
Let $u$ be the coordinate as above with 
$z=1-e^u.$
Then
\[{\phi_1}^{-1}\alpha_t \phi_1=t\beta_1+t^2 \beta_2+t^3\hat \beta(t)\]
where
\begin{equation*}
\begin{split}
\beta_1&=
\begin{psmallmatrix}
 0 & \frac{\left(\lambda ^3-1\right) e^u \left(e^u-3\right)}{\sqrt{6} \lambda  \left(e^u \left(e^u-3\right)+3\right)} & \frac{i \left(\lambda ^3+1\right) e^{2 u}}{\sqrt{6} \lambda ^2 \left(e^u \left(e^u-3\right)+3\right)} \\
 -\frac{\left(\lambda ^3-1\right) e^{2 u}}{\sqrt{6} \lambda ^2 \left(e^u \left(e^u-3\right)+3\right)} & 0 & -\frac{i e^u \left(e^u-3\right)}{\sqrt{3} \lambda  \left(e^u \left(e^u-3\right)+3\right)} \\
 -\frac{i \left(\lambda ^3+1\right) e^u \left(e^u-3\right)}{\sqrt{6} \lambda  \left(e^u \left(e^u-3\right)+3\right)} & \frac{i \lambda  e^u \left(e^u-3\right)}{\sqrt{3} \left(e^u \left(e^u-3\right)+3\right)} & 0 \\
\end{psmallmatrix}du\,,\\
\beta_2&=\begin{psmallmatrix}
 \frac{\left(\lambda\,^6+1\right) u}{\lambda ^3 (-2 \sinh (u)+4 \cosh (u)-3)} & 0 & 0 \\
 0 & -\frac{\left(\lambda\,^6+1\right) u}{2 \lambda ^3 (-2 \sinh (u)+4 \cosh (u)-3)} & 0 \\
 0 & 0 & -\frac{\left(\lambda\,^6+1\right) u}{2 \lambda ^3 (-2 \sinh (u)+4 \cosh (u)-3)} \\
\end{psmallmatrix}du\,.
\end{split}\end{equation*}
Moreover, there exists a constant $C_3$ with
\[||\hat\beta(t,u)(\tfrac{\partial}{\partial u})\,||<C_3\]
for all small $t>0$ and for all
\begin{equation}\label{eq:udomain} 
u\in\C
 \quad\text{with}\quad 
 -\tfrac{1}{t}\log(2)<\Re(u)<0.
 \end{equation}
\end{Lem}
\begin{proof}
As in the proof of Lemma \ref{lem:sphericallemma}, $\phi_1$ is bounded for all $t>0$ and all $u$ satisfying \eqref{eq:udomain}.
We expand (where $\hat\alpha$ extends real-analytically to $t=0$)
\[\alpha_t =t\alpha_1+ t^3\hat\alpha(t).\]
By construction $\alpha_1$ and $\hat\alpha$ are uniformly bounded for all (small) $t>0$ and all $u$ with $\Re(u)<0.$
The very form of $\beta_1$ and $\beta_2$ then follows by direct computation.

The norm of $\hat \beta(t)$ can be estimated because the norm of $\phi_1(u)=\phi_1(t,u)$ is uniformly bounded for all $t>0$ and for all
$u\in\C$ with $-\tfrac{1}{t}\log(2)<\Re(u)<0,$ and 
because $\alpha_1$ and $\hat \alpha(t)$
are uniformly bounded by construction.
\end{proof}

\begin{Pro}\label{pro:sphericalT}
There exists a constant $C_4>0$ 
and a smooth map \[\hat U\colon(0,\epsilon)\times \{u\in\C\mid \Re(u)<0\}\to {\bf sl}_3^{\hat\tau}(\S^1)\]
such that
 for all small $t>0$ and all
$u\in\C$ with $ -\tfrac{1}{t}\log(2)<\Re(u)<0$ the following holds
\begin{equation}\label{eq:tildeOmegaTU}
\Omega_t(u)=\begin{psmallmatrix}
 1 & 0 & 0 \\
 0 & 1 & -\frac{i }{\lambda } \\
 0 & 0 & 1 \\
\end{psmallmatrix}\exp(-tA_1u)
\,\left(\Id-t\,U+t^2\,\hat U(t,u)\,\right)
\end{equation}
where
\begin{equation*}
\begin{split}
&U=\\
&\hspace{-1.cm}
{\tiny\begin{psmallmatrix}
 0 & \frac{\left(\lambda ^3-1\right) \left(\log \left(-3 e^u+e^{2 u}+3\right)+2 \sqrt{3} \tan ^{-1}\left(\frac{3-2 e^u}{\sqrt{3}}\right)\right)}{2 \sqrt{6} \lambda } & \frac{i \left(\lambda ^3+1\right) \left(\log \left(-3 e^u+e^{2 u}+3\right)-2 \sqrt{3} \tan ^{-1}\left(\frac{3-2 e^u}{\sqrt{3}}\right)\right)}{2 \sqrt{6} \lambda ^2} \\
 -\frac{\left(\lambda ^3-1\right) \left(\log \left(-3 e^u+e^{2 u}+3\right)-2 \sqrt{3} \tan ^{-1}\left(\frac{3-2 e^u}{\sqrt{3}}\right)\right)}{2 \sqrt{6} \lambda ^2} & 0 & -\frac{i \left(\log \left(-3 e^u+e^{2 u}+3\right)+2 \sqrt{3} \tan ^{-1}\left(\frac{3-2 e^u}{\sqrt{3}}\right)\right)}{2 \sqrt{3} \lambda } \\
 -\frac{i \left(\lambda ^3+1\right) \left(\log \left(-3 e^u+e^{2 u}+3\right)+2 \sqrt{3} \tan ^{-1}\left(\frac{3-2 e^u}{\sqrt{3}}\right)\right)}{2 \sqrt{6} \lambda } & -\frac{1}{6} i \lambda  \left(6 \tan ^{-1}\left(\frac{3-2 e^u}{\sqrt{3}}\right)-\sqrt{3} \log \left(-3 e^u+e^{2 u}+3\right)\right) & 0 \\
\end{psmallmatrix}}
\end{split}
\end{equation*}

 and
\[||\hat U(t,u)||<C_4.\]
\end{Pro}

\begin{proof}
By construction, i.e., from \eqref{eq:phi1twispot},
\[\Omega_t(u)=\phi\;\phi_1\,\phi_2\]
where $\phi$ and $\phi_1$ are as above (i.e. the first two factors in \eqref{eq:tildeOmegaTU}), and
$\phi_2$ solves
\[d\phi_2+\beta\phi_2=0\quad\quad\text{with}\quad\quad \phi_2(z=0)=R(t)=\Id+O(t^2).\]

Note that the primitives
\[U=\int_0^u\beta_1\quad \quad\text{and}\quad\quad\int_0^u\beta_2\]
of $\beta_1$ and $\beta_2$ are uniformly bounded for small $t>0$ and for $u$ satisfying \eqref{eq:udomain}.
Therefore, the statement follows from Lemma \ref{lem:secondestimate} using Gr\"onwall inequality once more.
\end{proof}

Define
\begin{equation}\label{defWw}W:=-\lim_{u\to-\infty}U(u)=-\begin{psmallmatrix}
0&\frac{\left(\lambda ^3-1\right) \left(\pi +\sqrt{3} \log (3)\right)}{6 \sqrt{2} \lambda }& -\frac{i \left(\lambda ^3+1\right) \left(\pi -\sqrt{3} \log (3)\right)}{6 \sqrt{2} \lambda ^2}\\
\frac{\left(\lambda ^3-1\right) \left(\pi -\sqrt{3} \log (3)\right)}{6 \sqrt{2} \lambda ^2}&0& -\frac{i \left(\pi +\sqrt{3} \log (3)\right)}{6 \lambda }\\
-\frac{i \left(\lambda ^3+1\right) \left(\pi +\sqrt{3} \log (3)\right)}{6 \sqrt{2} \lambda }& -\frac{1}{6} i \lambda  \left(\pi -\sqrt{3} \log (3)\right)&0\\
\end{psmallmatrix}.\end{equation}
We then obtain the following refinement of Lemma \ref{lem:sphericallemma}:
\begin{Lem}\label{lem:sphericallemmaT}
Consider the coordinate $w$ with $z=1-w^{\tfrac{1}{t}}$ on the disc ${\mathbb D}$. 
It holds locally uniformly on $\{w\in\C\mid 0<|w|<1\}$ and for all small $t>0:$
\begin{equation*}\Omega_t(w)=\begin{psmallmatrix}
 1 & 0 & 0 \\
 0 & 1 & -\frac{i }{\lambda } \\
 0 & 0 & 1 \\
\end{psmallmatrix}\exp(-A_1\log w)(\Id+t W+O(t^2)\,).\end{equation*}
\end{Lem}
\begin{proof}
Recall that in the domain $\{w\in\C\mid |w|<1\}$ it holds locally uniformly
$\mu_t=\nu+ O(t^2).$
Since $w=\exp(tu)$ and for any $w_0$ with $|w_0|<1$
\[\int_{w=1}^{w_0}w^*\beta_1=\int_{u=0}^{-\infty}\beta_1=
\lim_{u\to-\infty}U(u)=-W\]
we obtain the statement directly from Proposition \ref{pro:sphericalT}. 
\end{proof}

Using Lemma \ref{lem:sphericallemmaT}, we can study the first-order deformation in $t$ of the special Legendrian surfaces  $\hat f_k$, $k=t^{-1}$,
at $t=0.$
Namely, the following theorem shows that on the $w$-disc the first-order deformation of the surface is just given by a conformal reparametrization
by an inward-pointing conformal vector field.
\begin{The}\label{The:limit-surfacerp2T}
Consider $\hat f_t\colon {\mathbb D}\to \S^{5}$  
as in Theorem \ref{The:limit-surfacerp2}. Then
\[\dot{\hat f}_{t=0}= d\hat f(X),\qquad\quad\text{
where}
\quad X=-\tfrac{\log(3)}{2}r\tfrac{\partial}{\partial r}=-\tfrac{\log(3)}{2} (x\tfrac{\partial}{\partial x}+y\tfrac{\partial}{\partial y})\]
is a conformal vector field on the disc with conformal coordinate $w=x+iy$.
\end{The}
\begin{Rem}
This first-order deformation  is
locally uniform, but not uniform, on the disc ${\mathbb D}.$ We will use Proposition \ref{pro:sphericalT} to improve the statement.
\end{Rem}
\begin{proof}
Consider the solution $\Omega_t$ of $d\Omega_t+\omega_t\Omega_t=0$ with appropriate initial condition $R(t)$
which guarantees unitary monodromy. Consider the Iwasawa decomposition
$\Omega_t=B_tF_t$ where $B_t$ is (normalized) positive in the loop group and $F_t$ is unitary.  
We have computed $B_0$ and $F_0$ in the proof of Theorem \ref{The:limit-surfacerp2}.
Expand $B_t=B_0(1+t P+\dots)$ and $F_t=F_0(1+t \widetilde X+\dots)$
where
\[P\colon {\mathbb D}\to {\bf sl}_3^{\hat\tau}(\D)\quad\text{
and}\quad 
\widetilde X\colon {\mathbb D}\to{\bf su}_3^{\hat\tau}(\S^1),\]
and use
\begin{equation*}
\begin{split}
\Omega_t&=\Omega_0(1+t W+\dots)=B_0F_0(1+t W+\dots)=B_0(1+t P+\dots) F_0(1+t \widetilde X+\dots)
\end{split}
\end{equation*}
to obtain the equation
\[F_0W=PF_0+F_0 \widetilde X.\]
Recall from \eqref{defWw} that $W\in  {\bf sl}_3^{\hat\tau}(\S^1)$ is independent of $w\in \mathbb D.$
Split
\[W=W_u+W_s\] with
$W_u,(iW_s)\in{\bf su}_3^{\hat\tau}(\S^1).$
We can ignore $W_u$ in the following since it is unitary and independent of $w, $ i.e., it only corresponds to an infinitesimal symmetry of $\S^5.$

Then $W_s$ satisfies $(W_s)^*=W_s$. This property is preserved under conjugation with elements of $\Lambda\SU_3.$ Consider
\begin{equation}\label{eq:Wsplitting}
\begin{split}\widetilde W:=F_0W_sF_0^{-1}=P+F_0 \hat X F_0^{-1}\end{split}
\end{equation}
where $\hat X=\widetilde X-W_u$, and where the left hand side $\widetilde W$ is symmetric. Splitting
\[\widetilde W=\widetilde W^++\widetilde W^0+\widetilde W^-\]
into positive, constant and negative parts and using $F_0 \hat X F_0^{-1}\in {\bf su}_3^{\hat\tau}(\S^1)$ therefore gives:
\[P=2 \widetilde W^++\widetilde W^0\quad
\text{and}
\quad F_0 \hat X F_0^{-1}=-\widetilde W^++\widetilde W^-.\]
Then, (neglecting the infinitesimal rotational part) a direct computation yields
\begin{equation*}
\begin{split}
\dot{\hat f}^{t=0}&=\dot{(F_t)^{-1}(\lambda_0)}\begin{psmallmatrix}1\\0\\0\end{psmallmatrix}=-(\hat X F_0^{-1})(\lambda_0)\begin{psmallmatrix}1\\0\\0\end{psmallmatrix}\\
&=\tfrac{\left(\sqrt{3}+2\right) \log (3)}{2\left(x^2+y^2+\sqrt{3}+2\right)^2}
\begin{psmallmatrix}4 i(x^2+y^2)\\
(\sqrt{3}-1) (ix+y) (x^2+y^2-\sqrt{3}-2)\\
 (\sqrt{3}-1) (ix-y) (x^2+y^2-\sqrt{3}-2)\end{psmallmatrix}=d{\hat f}^{t=0}(-\tfrac{\log(3)}{2} (x\tfrac{\partial}{\partial x}+y\tfrac{\partial}{\partial y}))
\end{split}
\end{equation*}
as claimed.
\end{proof}

 \subsection{Proof of Embeddedness}\label{ssec:emb}
 Before proving embeddedness of the Legendrian lifts, we note that the projected minimal Lagrangian immersions in $\CP^2$ cannot be embedded. Indeed, for any compact oriented Lagrangian surface $\Sigma\subset\CP^2$ one has
\[
[\Sigma]\cdot[\Sigma]=-\chi(\Sigma)=2g-2,
\]
see for example \cite{Zhang}. Since $g(\Sigma_k)=\frac12(k-1)(k-2)$, the algebraic self-intersection number of $f_k(\Sigma_k)$ equals $k^2-3k$.
 It can be shown that the $3k$ many points 
 $p^j_\ell\in\Sigma_k$ are mapped via $f_k$ to the three points $[1,0,0],$ $[0,1,0]$ and $[0,0,1]$, respectively, and each
 of the three
 contributes with self-intersection number $-k.$
 Similarly, each of the
  $k^2$-many preimages of $0\in\CP^1\in\Sigma_k/\Gamma_{deck}$ is mapped to the same point in $\CP^2$ as
  exactly one of the $k^2$-many preimages of $\infty\in\CP^1\in\Sigma_k/\Gamma_{deck}$, and the corresponding
  self-intersection number is $1.$ Using the symmetries of $\hat f_k$, one sees that these points get separated through
  the horizontal special Legendrian lift $\hat f_k.$ In fact, we have:

\begin{The}\label{the:embedded}
For all large $k$ and $c_0=\tfrac{1}{\sqrt{3}}$, the special Legendrian surfaces $\hat f_k=\hat f_k^+\colon\Sigma_k\to\S^5$ constructed in Corollary \ref{cor:SLag} and Theorem \ref{mainT} are embedded.
\end{The}
\begin{proof}
The proof is by contradiction based on Theorem \ref{the:scherk}, Theorem \ref{The:limit-surfacerp2} and Theorem \ref{The:limit-surfacerp2T}. Let
\[a_k\neq b_k\in\Sigma_k\]
be a sequence of distinct points satisfying
\[c_k:=\hat f_k(a_k)=\hat f_k(b_k)\in\S^5 \qquad \text{for all sufficiently large } k.\]
Since $\S^5$ is compact, after taking a subsequence if necessary, the sequence $c_k$ has a limit point, denoted by $c.$

We make use of the notations introduced in Section \ref{sec:RsandN}.
In particular, we recall that 
the centers of the discs ${\mathbb D}^j_\ell$, $j=1,2,3$, are mapped to equidistant points 
on the circles \[C_1=\S^5\cap \C\oplus0\oplus 0,\quad  C_2=\S^5\cap 0\oplus \C\oplus0,\quad  C_3=\S^5\cap 0\oplus0\oplus \C\]
respectively. 
Composing, for each $k$, with a suitable element of $\Gamma_{deck}$ and the corresponding $\SU_3$-symmetry of $\S^5$, we may assume
\[
a_k\in \overline{\mathbb D^1_1},\qquad \arg w(a_k)\in\Bigl[0,\frac{2\pi}{k}\Bigr).
\]
Since these ambient symmetries are isometries of $\S^5$, the equality
$\hat f_k(a_k)=\hat f_k(b_k)$ is preserved.
In particular, Proposition \ref{pro:uniform} then implies that $c$ is contained in the image of the geodesic segment
from $(i,0,0)$ to $\tfrac{i}{\sqrt{3}}(1,1,1).$

We distinguish two cases.

{\bf Case 1: $c\neq \tfrac{i}{\sqrt{3}}(1,1,1)$} 

Using Proposition \ref{pro:uniform}, we have $a_k\in {\mathbb D}^1_1$. Let $b_k\in \bar{\mathbb D}^{j(k)}_{\ell(k)}$ 
where $j(k)\in\{1,2,3\}$ and $\ell(k)\in\{1,\dots,k\}.$
For $k$ large, Proposition \ref{pro:uniform} (together with the symmetries) implies that $j(k)=1$. 
Moreover, if \[\lim_{k\to\infty}|w(b_k)|=1\] it follows that the limit
\[\lim_{k\to\infty} \hat f_k(b_k)\in\{ \psi^s\phi^t \tfrac{i}{\sqrt{3}}(1,1,1)\mid\; s,t\in[0,1]\}=:\text{Cliff}\]
is contained in the Clifford torus\footnote{This Clifford torus is also minimal Legendrian, but with different Legendrian phase.} in $\S^5$, where 
$
\psi^s=\text{diag}(1,\exp(2\pi i s),\exp(-2\pi i s))
$ and $
\phi^t=\text{diag}(\exp(2\pi i t),1,\exp(-2\pi i t))
$
are the continuous symmetries corresponding to the discrete symmetries $\tilde\varphi_1$ and $\tilde\varphi_3,$ respectively.  Hence, there exists  $r<1$ such that for all large enough $k$
\[|w(b_k)|<r.\] 

By construction, i.e.
because the infinitesimal generator of $\phi^t$ is transverse to the compact cap,
 (compare also with Remark \ref{rem:intersection}),  there is a constant $c_0>0$ such that for small $t>0$
\begin{equation}\label{eq:distcap}\text{dist}_{\S^5}(\,\hat f({\mathbb D}),\phi^t(\hat f({\mathbb D}))\,)\geq c_0 t.\end{equation}
Denote
\[{\mathbb D}^1_\ell(r):=\{b\in {\mathbb D}_\ell^1\mid |w(b)|<r<1\}.\]
From Proposition \ref{pro:uniform} and Theorem \ref{The:limit-surfacerp2T} together with the symmetries 
and \eqref{eq:distcap}
we therefore
obtain
\begin{equation}\label{eq:drift}\text{dist}_{\S^5}(\,\hat f_k({\mathbb D}^1_1(r)),\hat f_k({\mathbb D}^1_\ell(r))\,)\geq c_1 \tfrac{1}{k} \end{equation}
for a suitable $c_1>0$, $k$ large and every $\ell\in\{2,\dots,k\}.$ We therefore conclude that $b_k\in {\mathbb D}^1_1(r)$ for all large $k.$
Therefore, Proposition \ref{pro:uniform} and Theorem \ref{The:limit-surfacerp2T}
imply $a_k=b_k$ for large $k,$ showing that {Case 1} is not possible.

{\bf Case 2: $c= \tfrac{i}{\sqrt{3}}(1,1,1)$} 

Similar to the proof of \eqref{eq:drift}, we  assume without loss of generality that 
$b_k\in \bar {\mathbb D}^{j(k)}_{\ell(k)}$ for all large $k$ with
$\ell(k)/k\to 0$, $j(k)\in\{1,2,3\}$ and
\[\lim_k\arg(w(b_k))=0\]
where $w$ is the natural holomorphic coordinate on the corresponding discs.
Likewise, this can also be deduced from Corollary \ref{cor:scherk}.

We distinguish three cases by whether the projections of $a_k,b_k$ remain in a fixed compact subset of the base $\CP^1\setminus\{1,\zeta,\zeta^2\}$.
\begin{enumerate}
\item[(i)]
the points $a_k$ and $b_k$ project to some compact subset $\mathbb K\subset\CP^1\setminus\{1,\zeta,\zeta^2\}=\Sigma_k/(\Z_k\times\Z_k)$ for all large $k;$
\item[(ii)] for every compact subset $\mathbb K\subset\CP^1\setminus\{1,\zeta,\zeta^2\}$ both $a_k$ and $b_k$ are not contained in $\mathbb K$ for all large $k;$
\item[(iii)] for every (large enough) compact subset $\mathbb K\subset\CP^1\setminus\{1,\zeta,\zeta^2\}$ exactly one of $a_k$ and $b_k$ is contained in $\mathbb K$ for all large $k.$
\end{enumerate}
We will show in all three cases that $a_k=b_k$ for large $k$. Going to subsequences if necessary, this then proves the general case.

{\bf Case 2 (i)}
In this case, there exists $\delta>0$ such that
\[\text{dist}(\hat f_k(a_k),c)<\delta \tfrac{1}{k}\quad\text{ and }\quad\text{dist}(\hat f_k(b_k),c)<\delta \tfrac{1}{k}\]
for all sufficiently large $k.$ 
By Theorem \ref{the:scherk}, after identifying the horizontal space $\mathcal H_c\subset T_c\S^5$ 
of the Hopf fibration with $\C^2$, the rescaled maps
\[
k\,\exp_c^{-1}\circ \hat f_k
\]
converge in $C^1$ on  $\mathbb K$ to the embedded Scherk surface (Corollary \ref{cor:scherk}).
 Hence, we deduce $a_k=b_k$ for large $k.$

 {\bf Case 2 (ii).}
Put $t=\tfrac{1}{k}$, and write $\hat f_t=\hat f_k$.  On
$\mathbb D^1_1\setminus\{p^1_1\}$ we use the logarithmic coordinate
\[
w=e^{tu},\qquad z=1-w^k=1-e^u,\qquad u=x+iy.
\]
Thus the annular region near the point $z=1$ is described by
large negative values of $x=\Re u$.  For $r>0$, set
\[
S_{r,t}:=
\left\{
u=x+i y\;\middle|\;
-\tfrac{\log 2}{t}<x<-r
\right\},
\qquad
A_{r,t}:=\{\,e^{tu}\mid u\in S_{r,t}\,\}.
\]
Thus we have
$
A_{r,t}=\{\,w\in\mathbb D\mid \tfrac12<|w|<e^{-rt}\,\}.
$
When a branch of $u=t^{-1}\log w$ is chosen, the estimates below are
independent of the branch; on adjacent branches one obtains the same
estimate by the corresponding $\mathbb Z_k$-symmetry.

We first recall what the assumption of Case 2(ii) means in this
coordinate.  Since $c=\frac{i}{\sqrt 3}(1,1,1)$, Proposition
\ref{pro:uniform} and the fact that $\hat f\colon\overline{\mathbb D}\to
\S^5$ is an embedding imply
$ w(a_k)\longrightarrow 1$.
Hence
\[
        \tfrac1k\,\Re u(a_k)=\log |w(a_k)|\longrightarrow 0 .
\]
On the other hand, Case 2(ii) says that the projections of $a_k$ leave
every compact subset of
$\CP^1\setminus\{1,\zeta,\zeta^2\}$.  Since $a_k\in\mathbb D^1_1$, this
means $z(a_k)\to 1$, or equivalently
$e^{u(a_k)}=1-z(a_k)\longrightarrow 0.$
Therefore
\[
\Re u(a_k)\longrightarrow -\infty .
\]
Combining the two conclusions, we obtain: for every fixed $r>0$, there
exists $k_0$ such that
\begin{equation}\label{transition-regime}
        -\log(2)\,k < \Re u(a_k)<-r
\end{equation}
for all $k\geq k_0$.  The same statement holds for $b_k$, in the
corresponding logarithmic coordinate on the disc containing $b_k$.

We use the following uniform first-order estimate on the
transition annulus.  For every $\epsilon>0$, there are $r>0$ and
$t_0>0$ such that, for all $0<t<t_0$ and all $u\in S_{r,t}$,
\begin{equation}\label{eq:firstorddefo}
\hat f_t(e^{tu})
=
\hat f(e^{tu})
+
t\,d\hat f_{e^{tu}}\left(X(e^{tu})\right)
+
t\,Y_t(u),
\qquad
\|Y_t\|_{C^1(S_{r,t})}<\epsilon .
\end{equation}
Here $X=-\frac{\log 3}{2}\,r\partial_r$ is the vector field from Theorem
\ref{The:limit-surfacerp2T}, and we view all maps as $\C^3$-valued maps
using the inclusion $\S^5\subset\C^3$.

Indeed, Proposition \ref{pro:sphericalT} gives, on the $u$-domain,
\[
\Omega_t(u)=
\begin{psmallmatrix}
1&0&0\\
0&1&-\frac{i}{\lambda}\\
0&0&1
\end{psmallmatrix}
\exp(-tA_1u)
\left(\Id-t\,U(u)+t^2\hat U(t,u)\right),
\]
with $\hat U$ uniformly bounded.  On the other hand, the expansion
used in Theorem \ref{The:limit-surfacerp2T} is obtained by replacing
$-U(u)$ by the limiting value $W$, where
$ W=-\lim_{\Re u\to-\infty}U(u).
$
Thus
\[U(u)+W\longrightarrow 0\]
as $\Re u\to-\infty$, uniformly in $C^1$ on the strips
$\{\Re u<-r\}$.  Since the Iwasawa decomposition is real analytic 
on a
neighbourhood of the compact family of frames under consideration, the
induced maps into $\S^5$ differ in $C^1$ by
\[O\left(t\,\|U+W\|_{C^1(\{\Re u<-r\})}\right)+O(t^2).
\]
Choosing $r$ large and then $t$ small gives \eqref{eq:firstorddefo}.
Theorem \ref{The:limit-surfacerp2T} identifies the first-order term coming
from $W$ precisely with $d\hat f(X)$.

Let $ \Theta_s(w)=e^{-\frac{\log 3}{2}s}w
$
 denote the flow of $X$. Since
\[
\hat f(\Theta_t(w))
=
\hat f(w)+t\,d\hat f_w(X(w))+O(t^2)
\]
uniformly on $\overline{\mathbb D}$, \eqref{eq:firstorddefo} can be
rewritten on $A_{r,t}$ as
\begin{equation}\label{eq:annular-reparam} \hat f_t(w)
        =\hat f(\Theta_t(w))
       + t\,Y_t(t^{-1}\log w)
        +O(t^2),
        \qquad
        \|Y_t\|_{C^1}<\epsilon .
\end{equation}
Thus, in the transition region, the first-order deformation is a
reparametrization of the limiting cap, up to an arbitrarily small
$C^1$-error.

We now exclude the possible positions of $b_k$.  First suppose that
$b_k$ lies in a disc different from $\mathbb D^1_1$.  The leading term
in \eqref{eq:annular-reparam} is the corresponding model cap, moved by the
inward flow $\Theta_t$ and then by the appropriate ambient symmetry.
For the model caps this gives a separation of order $t$.  More precisely,
after possibly increasing $r$, there is a constant $c_r>0$ such that
for all sufficiently small $t$,
\[
\operatorname{dist}_{\S^5}
\Bigl(
\hat f(\Theta_t(A_{r,t})),
\,
\hat{\mathbb D}^{j}_{\ell}\text{ moved by }\Theta_t
\Bigr)
\geq c_r t
\]
whenever $(j,\ell)\neq(1,1)$ and the corresponding boundary point tends
to $c=\frac{i}{\sqrt 3}(1,1,1)$.  This follows directly from the explicit model caps: 
If \(j\neq1\), the two model caps meet at
\(\frac{i}{\sqrt3}(1,1,1)\) at the common boundary point.
  In the transition region the points have moved a positive
distance 
 inward from the common boundary point.
Since the inward directions of the two caps are separated by a fixed
angle, the points move away from each other. Thus $b_k\notin \mathbb D^j_\ell$ for large $k.$

If
$j=1$ but $\ell\neq1$, the infinitesimal $\tilde\varphi_3$-rotation
moves the boundary point with non-zero speed, giving separation at least
of order $|\ell-1|t$, hence at least of order $t$.
The error term in \eqref{eq:annular-reparam} is bounded by
$\epsilon t+O(t^2)$.  Choosing $\epsilon<c_r/4$ and then $t$ small,
this error cannot close the order-$t$ gap between the different cap
pieces.  Hence $b_k$ cannot lie in any disc different from
$\mathbb D^1_1$.

It remains to consider the case $b_k\in\mathbb D^1_1$.  By
\eqref{transition-regime}, both $a_k$ and $b_k$ lie in the annulus
$A_{r,\tfrac1k}$ for all sufficiently large $k$.  The map
\[w\longmapsto \hat f(\Theta_{\tfrac1k}(w))\]
is an embedding of the closed annulus
$\{\,w\in\mathbb D\mid \tfrac12\leq |w|\leq e^{-r\tfrac1k}\,\}.$
Since the limit surface $\hat f$ is an embedding of the closed disc, this embedding is
stable under sufficiently small $C^1$-perturbations.  Therefore, after
choosing $\epsilon$ and $\tfrac1k$ sufficiently small, the map
$\hat f_{k}$ is injective on $A_{r,\tfrac1k}$.  Thus
$\hat f_k(a_k)=\hat f_k(b_k)$
forces $w(a_k)=w(b_k)$, and hence $a_k=b_k$, because $w$ is a
coordinate on $\mathbb D^1_1$.  This contradicts the assumption that
$a_k\neq b_k$.  Therefore Case 2(ii) cannot occur.

  {\bf Case 2 (iii)}
 This case follows by combining the arguments of  Case 2 (i) and Case 2 (ii):
 Choose a compact set $\mathbb K$ so large that points projecting into $\mathbb K$ are in the Scherk regime of Case 2(i), while points outside $\mathbb K$ lie in the boundary-annulus regime of Case 2(ii). The rescaled Scherk image of the compact part and the annular cap image of the noncompact part are separated by a distance bounded below by $c/k$ for some $c>0$, for all sufficiently large $k$. 
 Hence  $\hat f_k(a_k)=\hat f_k(b_k)$ is impossible.
 \end{proof}

 \bigskip
 \noindent \footnotesize \textsc{Beijing Institute of Mathematical Sciences and Applications, Beijing 101408, China
}\\
\emph{E-mail address:}  \verb|sheller@bimsa.cn|
 
 \bigskip
 \noindent \footnotesize \textsc{Department of Mathematics, Washington University in St. Louis}\\
\emph{E-mail address:}  \verb|ouyang@math.wustl.edu|

\bigskip
\noindent \footnotesize \textsc{Department of Mathematics and Statistics, University of Massachusetts, Amherst}\\
\emph{E-mail address:}  \verb|pedit@math.umass.edu|

\end{document}